\documentclass[a4paper]{amsart}

\usepackage[applemac]{inputenc}
\usepackage[T1]{fontenc}
\usepackage[english]{babel}
\usepackage{amsfonts}
\usepackage{amsmath}
\usepackage{amsthm}
\usepackage{amssymb}
\usepackage{graphicx}
\usepackage{enumerate}
\usepackage[all,cmtip]{xy}
\usepackage{url}

\DeclareRobustCommand{\gobblefour}[4]{}
\newcommand*{\SkipTocEntry}{\addtocontents{toc}{\gobblefour}}

\newcommand{\A}{\mathbf{A}}

\newcommand{\C}{\mathbf{C}}

\newcommand{\F}{\mathbf{F}}
\newcommand{\Gm}{\mathbf{G}_m}

\newcommand{\N}{\mathbf{N}}
\renewcommand{\P}{\mathbf{P}}
\newcommand{\Q}{\mathbf{Q}}
\newcommand{\R}{\mathbf{R}}

\newcommand{\U}{\mathbf{U}}

\newcommand{\Z}{\mathbf{Z}}

\renewcommand{\o}{\mathfrak{o}}

\renewcommand{\O}{\cal{O}}

\renewcommand{\d}{\textup{d}}
\newcommand{\Vv}{\textup{V}}

\newcommand{\mini}{\textup{min}}

\renewcommand{\ss}{\textup{ss}}

\newcommand{\too}{\longrightarrow}
\newcommand{\into}{\hookrightarrow}

\newcommand{\bs}{\boldsymbol}

\renewcommand{\phi}{\varphi}
\renewcommand{\epsilon}{\varepsilon}

\newcommand{\df}{:=}
\renewcommand{\ker}{\Ker}

\renewcommand{\em}{\textit}
\renewcommand{\hat}{\widehat}

\renewcommand{\frak}{\mathfrak}
\newcommand{\cal}{\mathcal}

\newcommand{\iso}{\simeq}

\newcommand{\quotss}[2]{{#1}/\!\!/{#2}}

\newcommand{\quadrato}{\Box}
\newcommand{\triangleup}{\Delta} %delta
\renewcommand{\triangledown}{\nabla} %nabla
\DeclareMathOperator{\m}{m}
\DeclareMathOperator{\mult}{mult}

\DeclareMathOperator{\rk}{rk}
\DeclareMathOperator{\Tr}{Tr}

\DeclareMathOperator{\Sym}{Sym}

\DeclareMathOperator{\Spec}{Spec}

\DeclareMathOperator{\Proj}{Proj}

\DeclareMathOperator{\GLs}{\mathbf{GL}}

\DeclareMathOperator{\SLs}{\mathbf{SL}}

\DeclareMathOperator{\Grass}{\mathbf{Gr}}
\DeclareMathOperator{\End}{End}

\DeclareMathOperator{\Hom}{Hom}

\DeclareMathOperator{\Ker}{Ker}

\DeclareMathOperator{\id}{id}

\DeclareMathOperator{\degar}{\hat{deg}}
\DeclareMathOperator{\muar}{\hat{\mu}}

\DeclareMathOperator{\vol}{vol}

\DeclareMathOperator{\ind}{ind}
\DeclareMathOperator{\pr}{pr}

\DeclareMathOperator{\Wr}{Wr}

\newcommand{\ul}{\underline}
\newcommand{\ol}{\overline}

\newenvironment{npar}[1]{\begin{paragrafo_numerato_nome}\textbf{#1.}}{\end{paragrafo_numerato_nome}}
\newenvironment{np}{\begin{paragrafo_numerato}}{\end{paragrafo_numerato}}

\theoremstyle{plain}
     \newtheorem{theorem}{Theorem}[section]
    \newtheorem{prop}[theorem]{Proposition}
    \newtheorem{lem}[theorem]{Lemma}
    \newtheorem{cor}[theorem]{Corollary}

\theoremstyle{definition}

\newtheoremstyle{note}{11pt}{11pt}{}{}{\bfseries}{.}{.5em}{}
\newtheoremstyle{paragrafo}{11pt}{11pt}{}{}{\bfseries}{.}{.5em}{}

\theoremstyle{note}
    \newtheorem{deff}[theorem]{Definition}
    \newtheorem{rem}[theorem]{Remark}

\theoremstyle{paragrafo}
    \newtheorem{paragrafo_numerato}{}[subsection]
    \newtheorem{paragrafo_numerato_nome}[paragrafo_numerato]{}
      
  \numberwithin{equation}{subsection}

\begin{document}

\title{Geometric Invariant Theory and Roth's Theorem}

\author{Marco Maculan}

\thanks{Partially supported by ANR Projet Blanc ``Positive'' ANR-2010-BLAN-0119-01}

\maketitle

\begin{abstract}
We present a proof of Thue-Siegel-Roth's Theorem (and its more recent variants, such as those of Lang for number fields and that ``with moving targets'' of Vojta) as an application of Geometric Invariant Theory (GIT). Roth's Theorem is deduced from a general formula comparing the height of a semi-stable point and the height of its projection on the GIT quotient. In this setting, the role of the zero estimates appearing in the classical proof is played by the geometric semi-stability of the point to which we apply the formula.
\end{abstract}

 \tableofcontents

\setcounter{section}{-1}

\section{Introduction}\label{sec:intro}

In its original form, Roth's Theorem states that given a real algebraic number $\alpha \in \R$ which is not rational and a real number $\kappa > 2$, there exist only finitely many rational numbers $p/q \in \Q$ such that 
$$  \left|\alpha - \frac{p}{q} \right| < \frac{1}{|q|^\kappa} $$
where $p, q$ are coprime integers. 

The general strategy to prove Roth's Theorem stems back to the work of Thue. The main ingredient is the construction of an ``auxiliary'' polynomial in several variables $f$ which vanishes at high order at $(\alpha, \dots, \alpha)$: the crucial step is to prove that it does not vanish too much at rational points which ``approximate'' $(\alpha, \dots, \alpha)$. 

The original argument of Roth (generalizing those of Thue, Siegel and Gel'fond) involves arithmetic considerations about the height of the rational approximations. On the other hand, in the work of Dyson --- who proved an earlier version of Roth's Theorem --- the non-vanishing result (usually called ``Dyson's Lemma'') takes place over the complex numbers: being free from arithmetic constraints, it is said to be of geometric nature. The task to generalize Dyson's Lemma from $2$ to several variables was accomplished by Esnault-Viehweg \cite{esnault-viewheg}; afterwards Nakamaye \cite{nakamaye_id} was able to give a proof of it relying on a variant of Faltings' Product Theorem and ``elementary'' concepts of intersection theory.

The advantage of having a geometric proof of Dyson's Lemma was exploited by Bombieri in the remarkable paper \cite{bombieri}: he showed that these methods lead to new effective results in diophantine approximation available before only through the linear forms of logarithms of Baker.

Using an arithmetic variant of the Product Theorem, Faltings and W\"ustholz \cite{faltings-wustholz} gave a new proof of Schmidt's Subspace Theorem, sensibly different from the original one. Let us remark that their Zero Lemma, as in Roth and Schmidt, is of arithmetic nature. Their proof involves a notion of semi-stability for filtered vector spaces (see also \cite{faltings_icm}). The role played by semi-stability is anyway rather different from the one in the present paper: here it collects all the geometric informations coming from Dyson's Lemma (hence from the Product Theorem); in their paper it represents a combinatorial assumption that permits to perform an inductive step based on the Product Theorem.

Inspired by work of Osgood \cite{Osgood} and Steinmetz \cite{Steinmetz} Vojta proved in \cite{vojta_movingtargets} a generalised version of Roth's Theorem --- called ``with moving targets'' --- where the algebraic point can vary along with the rational approximations. Its proof is based on the use of Schmidt's Subspace Theorem. However it has been noticed by Bombieri and Gubler \cite[Theorem 6.5.2 and \S 6.6]{bombieri-gubler} that the techniques employed to prove Roth's Theorem suffice to prove the version ``with moving targets'' without recurring to Schmidt's Subspace Theorem.

\

The connections between Geometric Invariant Theory and Arakelov Geometry have been studied by several authors in the last twenty years (Burnol \cite{burnol92}, Bost \cite{bost94}, Zhang \cite{zhang94}, Gasbarri \cite{gasbarri00} and Chen \cite{chen_ss}).

The application of these techniques to diophantine approximation was largely inspired by \cite{bost94}, where Bost proves a lower bound for the height of cycles with semi-stable Chow point. Generalizing these arguments Gasbarri gave in \cite{gasbarri00} a general lower bound for the height of semi-stable points for a large class of representations. An explicit version of the latter has been then proved by Chen \cite{chen_ss} by means of Classical Invariant Theory.

In this article, we show how a simple version of this general lower bound on the height of (geometrically) semi-stable point leads to a general lower bound on the height of suitable families of points $(x_1, \dots, x_n)$ and $(a_1, \dots, a_n)$ in $\P^1(K)^n$ and $\P^1(K')^n$ respectively to the diverse $v$-adic distances (where $K$ is a number field and $K'$ is an extension of degree $\ge 2$). This lower bound, which constitutes the main result of the present note, has been established in the case $n = 2$ by Bombieri \cite[Theorem 2]{bombieri}, is effective and implies the version of Roth's Theorem we present here.

\vspace{10pt}This paper is organized as follows.

In Section \ref{sec:StatementOfResults} we review some material concerning Roth's Theorem and we state the main result of this paper (the Main Theorem, see Theorem \ref{thm:ApplyingComparisonFormula}). More precisely, we show that Roth's Theorem with moving targets is a consequence of an effective statement (the Main Effective Lower Bound, see Theorem \ref{thm:fundamental_effective_LB}) and how this last result is obtained from the Main Theorem for a suitable choice of parameters.

In Section \ref{sec:GITArakelov} we introduce the main tool of Geometric Invariant Theory (the Fundamental Formula, see Theorem \ref{thm:ComparisonFormula}) that we shall later apply to a specific ``moduli problem'' in order to get the Main Theorem: it is a formula relating the height of a semi-stable point with the height of its projection on the GIT quotient. In this general framework we also state and prove a lower bound of the height on the quotient (see Theorem \ref{thm:LBHeightQuotient}). This section resumes all the results of GIT \em{\`a la} Arakelov needed for the proof in the following sections and they are presented in a more general setting.

In Section \ref{sec:FundamentalFormulaToMainTheorem} we introduce the situation of Geometric Invariant Theory that we are interested in. Admitting the semi-stability of the point that we introduce and some intermediate computations, we show that the Fundamental Formula translates into the Main Theorem. 

These intermediate computations (upper bounds of the height and the instability measure of the point) are developed in detail in Sections \ref{sec:UBHeight} and \ref{sec:UBInstabilityMeasure}.

Finally, in Section \ref{sec:Semi-stability}, we show the semi-stability of the point, which is the crucial result in order to apply the Fundamental Formula. Our proof is based on the Higher Dimensional Dyson's Lemma by Esnault-Viehweg-Nakamaye (Theorem \ref{thm:DysonsLemma}). We also give an alternative proof in dimension $2$ based on the classical constructions of Wronskians. This argument provides a simple ``GIT proof'' of the classical Theorem of Dyson-Gelfon'd.

\SkipTocEntry\subsection*{Acknowledgements}

The results presented here are part of my doctoral thesis \cite{maculan_these} supervised by J.-B. Bost. It is a pleasure for me to thank him for his guidance and his steady encouragement. During the preparation of the present article I have been stimulated by discussions with several people: I warmly thank A. Chambert-Loir, C. Gasbarri and M. Nakamaye. This paper also benefited from the sharp advices of J. Fres\'an. Finally, I want to sincerely thank the referee for his careful reading and his series of remarks that sensibly improved the quality of the present paper.

\SkipTocEntry\subsection*{Conventions}

We list here some conventions and definitions that are used throughout the paper. 

\begin{np} Since we are interested on the action of $\SLs_2$ on the projective line $\P^1$, we cannot confuse the projective line and the dual one. If $A$ is a ring and $n$ is a positive integer we denote by $A^{n\vee}$ the dual of the $A$-module $A^n$,
$$A^{n \vee} \df \Hom_A(A^n, A).$$
With this notation the projective line $\P^1_A$ over the ring $A$ is the $A$-scheme $$\P^1_A = \Proj (\Sym A^{2\vee}).$$
\end{np}

\begin{np} Let $A$ be a ring, $M$ be an $A$-module and $n$ be a negative integer. We set
$$ M^{\otimes n} \df M^{\vee \otimes - n} = \Hom_A(M, A)^{\otimes - n}.$$
\end{np}

\begin{np} \label{par:ConvetionHermitianNorms} Let $E$, $F$ be finite dimensional complex vector spaces equipped respectively with hermitian norms $\| \cdot \|_E$, $\| \cdot \|_F$ and associated hermitian forms $\langle - , - \rangle_E$, $\langle - , - \rangle_F$. Let $r$ be a non-negative integer.
\begin{itemize}
\item On the tensor power $E \otimes_\C F$ we consider the hermitian norm $\| \cdot \|_{E \otimes F}$ associated to the hermitian form
$$ \langle v \otimes w, v' \otimes w' \rangle_{E \otimes F} \df \langle v, v' \rangle_E \cdot \langle w, w' \rangle_F $$
where $v, v' \in E$ and $w, w' \in F$.
\item On the $r$-th symmetric power $\Sym^r E$ we consider the quotient norm $\| \cdot \|_{\Sym^r E}$ with respect to the canonical surjection $E^{\otimes r} \to \Sym^r E$. If $e_1, \dots, e_n$ denotes an orthonormal basis of $E$, where $n = \dim_\C E$, for every $n$-tuple of non-negative integers $(r_1, \dots, r_n)$ such that $r_1 + \cdots + r_n = r$ we have:
$$ \| e_1^{r_1} \cdots e_n^{r_n}\|_{\Sym^r E} = \binom{r}{r_1, \dots, r_n}^{-1/2} \df \left( \frac{r!}{r_1! \cdots r_n!}\right)^{-1/2}.$$
This norm is hermitian and it is sub-multiplicative in the following sense: if $f \in \Sym^r E$ and $g \in \Sym^s E$ we have
$$ \| fg \|_{\Sym^{r+s}E} \le \| f\|_{\Sym^r E} \| g\|_{\Sym^s E}.$$
Let us also mention that the norm $\| \cdot \|_{\Sym^r E}$ is bigger than the sup-norm on the unit ball: for $f \in \Sym^r E$ we have
$$ \| f \|_{\sup} \df \sup_{0 \neq x \in E^\vee} \frac{|f(x)|}{\| x \|^r_{E^\vee}} \le \| f \|_{\Sym^r E}.$$
\item On the $r$-th external power $\bigwedge^r E$ we consider the hermitian norm $\| \cdot \|_{\bigwedge^r E}$ associated to the hermitian form
$$ \langle v_1 \wedge \cdots \wedge v_r, w_1 \wedge \cdots \wedge w_r \rangle_{\bigwedge^r E} = \det \left(  \langle v_i, w_j \rangle_E : i, j = 1, \dots, r \right)$$
where $v_1, \dots, v_r$ and $w_1, \dots, w_r$ are elements of $E$. With this notation Hadamard inequality\footnote{This inequality is Hadamard's bound of the volume of a basis of a Euclidean space and not Hermite-Hadamard's inequality concerning convex functions.} reads : 
\begin{equation} \label{eq:HadamardInequality} \| v_1 \wedge \cdots \wedge v_r \|_{\bigwedge^r E} \le \prod_{i = 1}^r \| v_i\|_E. \end{equation}
The hermitian norm $\| \cdot \|_{\bigwedge^r E}$ is \em{not} the quotient norm with respect to the canonical surjection $E^{\otimes r} \to \bigwedge^r E$, but it is $\sqrt{r!}$ times the quotient norm (see \cite[Lemma 4.1]{chen_ss}).
\item For every linear homomorphism $\phi : E \to F$ we write $\phi^\ast$ for the adjoint homomorphism (with respect to the hermitian norms $\| \cdot \|_E$ and $\| \cdot \|_F$). On the vector space $\Hom_\C(E, F)$ we consider the hermitian norm $\| \cdot \|_{\Hom(E,F)}$ associated to the hermitian form
$$ \langle \phi, \psi \rangle_{\Hom(E, F)} \df \Tr(\phi \circ \psi^\ast)$$
where $\phi, \psi \in \Hom_\C(E, F)$. If $e_1, \dots, e_n$ is an orthonormal basis of $E$ we have
$$ \| \phi \|_{\Hom(E, F)} \df \sqrt{\| \phi(e_1) \|_{F}^2 + \cdots + \| \phi(e_n) \|_{F}^2}.$$
With these conventions the natural isomorphism $E^\vee \otimes_\C F \to \Hom_\C(E, F)$ is isometric.
\end{itemize}
\end{np}

\begin{np} Let $K$ be a field complete with respect to a non-archimedean absolute value and let $\o$ be its ring of integers. In order to do some computations it is convenient to interpret $\o$-modules as $K$-vector spaces endowed with a non-archimedean norm. More precisely, for every torsion free $\o$-module $\cal{E}$ let us denote by $E \df \cal{E} \otimes_\o K$ its generic fiber and consider the following norm: for every $v \in E$ we set
$$ \| v\|_{\cal{E}} \df \inf \{ |\lambda| : \lambda \in K^\times, v / \lambda \in \cal{E} \}.$$
The norm $\| \cdot \|_{\cal{E}}$ is non-archimedean and its construction is compatible with operations on $\o$-modules: for instance, if $\phi : \cal{E} \to \cal{F}$ is an injective homomorphism with flat cokernel (resp. surjective homomorphism) between torsion free $\o$-modules then the norm $\| \cdot \|_{\cal{E}}$ induced on $E \df \cal{E} \otimes_\o K$ (resp. the norm $\| \cdot \|_{\cal{F}}$ induced on $F \df \cal{F} \otimes_\o K$) is the restriction of the norm $\| \cdot \|_{\cal{F}}$ on $F$ (resp. is the quotient norm deduced from $\| \cdot \|_{\cal{E}}$ and $\phi$, that is, the norm defined by
$$ w \mapsto \inf_{\phi(v) = w} \| v \|_{\cal{E}}$$
for every element $w$ of $F$.) 

It follows that, for a non-negative integer $r \ge 0$, the norm on symmetric powers $\Sym^{r} \cal{E}$ (resp. on exterior powers $\bigwedge^r \cal{E}$) is the norm deduced by the one on the $r$-th tensor power $\cal{E}^{\otimes r}$ through the canonical surjection $\cal{E}^{\otimes r} \to \Sym^r \cal{E}$ (resp. $\cal{E}^{\otimes r} \to \bigwedge^r \cal{E}$).  In particular, it is sub-multiplicative (resp. Hadamard inequality holds).
\end{np}

\begin{np} If $K$ is a number field, we denote by $\o_K$ its ring of integers and by $\Vv_K$ the set of its places. If $v$ is a place we denote by $K_v$ the completion of $K$ with respect to $v$ and by $\C_v$ the completion of an algebraic closure of $K_v$ endowed with the unique absolute value extending the one of $K_v$. If $v$ is an non-archimedean place extending a $p$-adic one, we normalize it by
$$ |p|_v = p^{- [K_v : \Q_p]}.$$
\end{np}

\begin{np} Let $K$ be a number field, $\o_K$ its ring of integers and $\Vv_K$ its set of places. An hermitian vector bundle $\ol{\cal{E}}$ is the data of a flat $\o_K$-module of finite type $\cal{E}$ and, for every complex embedding $\sigma : K \to \C$, an hermitian norm $\| \cdot \|_{\cal{E}, \sigma}$ on the complex vector space $\cal{E}_\sigma \df \cal{E} \otimes_\sigma \C$. These hermitian norms are supposed to be compatible to complex conjugation. For every place $v \in \Vv_K$, we denote by $\| \cdot \|_{\cal{E}, v}$ the norm induced on the $K_v$-vector space $\cal{E}_v \df \cal{E} \otimes_{\o_K} K_v$.

If $\ol{\cal{E}}$, $\ol{\cal{F}}$ are hermitian vector bundles over $\o_K$, a \em{homomorphism of hermitian vector bundles} $\phi : \ol{\cal{E}} \to \ol{\cal{F}}$ is a homomorphism of $\o_K$-modules such that, for every embedding $\sigma : K \to \C$, it decreases the norms: that is, for every $v \in \cal{E} \otimes_{\sigma} \C$ we have
$$ \| \phi(v)\|_{\cal{F}, \sigma} \le \| v \|_{\cal{E}, \sigma}. $$

If $\ol{\cal{L}}$ is an hermitian line bundle, that is an hermitian vector bundle of rank $1$, we define its \em{degree} by
$$ \degar ( \ol{\cal{L}}) \df \log \# ( \cal{L} / s \cal{L}) - \sum_{\sigma : K \to \C} \log \| s \|_{\cal{L}, \sigma} =- \sum_{v \in \Vv_K} \log \| s\|_{\cal{L}, v}$$
where $s \in \cal{L}$ is non-zero. It appears clearly from the second expression that this, according to the Product Formula, does not depend on the chosen section $s$. If $\ol{\cal{E}}$ is an hermitian vector bundle we define
\begin{itemize}
\item its \em{degree}: $$ \degar \ol{\cal{E}} \df \degar ( \textstyle \bigwedge^{\rk \cal{E}} \ol{\cal{E}} \displaystyle );$$
\item its \em{slope}:  $$ \muar (\ol{\cal{E}}) \df \frac{\degar (\ol{\cal{E}})}{\rk \cal{E}};$$
\item its \em{maximal slope}: 
$$ \muar_{\max} (\ol{\cal{E}}) \df \sup_{0 \neq \cal{F} \subset \cal{E}} \muar (\ol{\cal{F}}),$$
where the supremum is taken on all non-zero sub-modules $\cal{F}$ of $\cal{E}$ endowed with the restriction of the hermitian metric on $\cal{E}$.
\end{itemize}

\begin{prop}[{Slopes inequality, \cite{BostBourbaki}}] \label{prop:SlopeInequality} Let $\ol{\cal{E}}, \ol{\cal{F}}$ be $\o_K$-hermitian vector bundles and let $\phi : \cal{E} \otimes_{\o_K} K \to \cal{F} \otimes_{\o_K} K$ be an injective homomorphism of $K$-vector spaces. Then,
$$ \muar (\ol{\cal{E}}) \le \muar_{\max}(\ol{\cal{F}}) + \sum_{v \in \Vv_K} \log \| \phi\|_{\sup, v}$$
where for every place $v \in \Vv_K$ we set
$$ \| \phi \|_{\sup, v} \df \sup_{0 \neq s \in \cal{E}_v} \frac{\| \phi(s)\|_{\cal{F}, v}}{ \| s \|_{\cal{E}, v}}.$$
\end{prop}

\end{np}

\section{Statement of the results} \label{sec:StatementOfResults}

\subsection{Roth's Theorem with moving targets and the Main Effective Lower Bound}

\begin{npar}{Height and distance on the projective line} In order to state the results in their most precise way it is convenient to make the following definitions. 

\begin{deff} \begin{enumerate}[(1)]
\item For a point $x = (x_0 : x_1)$ of the projective line $\P^1_\Q$ defined on a number field $K$ we consider its \em{absolute (logarithmic) height}
$$ h(x) = \frac{1}{[K : \Q]} \sum_{v \in \Vv_K} \log \| (x_0, x_1)\|_v$$
where $\Vv_K$ denotes the set of places of $K$ and for every place $v$ we write
$$ \| (x_0, x_1)\|_v \df
\begin{cases}
\vspace{7pt} \max \{|x_0|_v, |x_1|_v \} & \textup{if $v$ is non-archimedean} \\ 
\sqrt{|x_0|_v^2 + |x_1|^2_v} & \textup{if $v$ is archimedean}.
\end{cases}
$$

\item If $K$ is a number field and $v \in \Vv_K$ is a place of $K$, we consider the \em{$v$-adic spherical distance} on $\P^1$. If $x= (x_0:x_1)$ and $y = (y_0 : y_1)$ are $\C_v$-points of the projective line $\P^1$ we set
$$ \d_v(x, y) \df \frac{|x_0 y_1 - x_1 y_0|_v}{\| (x_0, x_1)\|_v \| (y_0, y_1) \|_v} \in [0, 1].$$

\item Let $x, y$ be $K$-points of $\P^1$. Then $\d_v(x, y) = 1$ for all but finitely many places $v$ of $K$ and for every subset $S \subset \Vv_K$ (not necessarily finite) we set
$$ \m_S(x, y) \df \sum_{v \in S} - \log \d_v(x, y) \in \R_{\ge 0}.$$
If $S = \{v\}$ is a singleton we just write $\m_v(x, y)$.
\end{enumerate}
\end{deff}

\begin{prop}[{\cite[Theorem 2.8.21]{bombieri-gubler}}] \label{prop:ProjectiveLiouvilleInequality} For two distinct points $x, y \in \P^1(K)$ we have 
$$ \frac{1}{[K : \Q]}\m_{\Vv_K}(x, y) = h(x) + h(y).$$ \end{prop}

\begin{deff} Let $a$ be a point of $\P^1$ defined over a finite extension $K'$ of $K$, $S \subset \Vv_K$ a finite subset of $\Vv_K$ and for every $v \in S$ let $\sigma_v : K' \to \C_v$ be a $K$-linear embedding. Then we denote by $a^{(\sigma_v)}$ the $\C_v$-point of $\P^1$ induced by $\sigma_v$ and we set:
$$ \m_S(a, x) \df \sum_{v \in S} \m_v(a^{(\sigma_v)}, x).$$
\end{deff}
\end{npar}

\begin{npar}{Roth's Theorem with moving targets} In this paper we prove the following form of Roth's Theorem with moving targets:

\begin{theorem}\label{thm:moving_targets} Let $K$ be a number field, $S \subset \Vv_K$ a finite subset, $K'$ a finite extension of $K$ and $\kappa > 2$ a real number. For every place $v \in S$ let us fix an embedding $\sigma_v : K' \to \C_v$ which respects $K$.

There is no sequence of couples $(x_i,a_i)$  with $i \in \N$ made of a $K$-rational point $x_i$ of $\P^1$ and a $K'$-rational point $a_i$ of $\P^1$ distinct from $x_i$ satisfying the following properties : 
\begin{enumerate}[(1)]
\item we have $h(a_i) = o(h(x_i))$ as $i$ goes to infinity;
\item for all $i \in \N$ the following inequality is satisfied:
$$ \frac{1}{[K : \Q]} \m_S(a_i, x_i) \ge \kappa h(x_i).$$
\end{enumerate}
\end{theorem}

%Standard techniques\footnote{Namely, imposing index not only at the point $a$ but at every point $a_v$ for every $v \in S$.} permit to modify our argument in order to obtain the following:

%\begin{theorem}[{\em{cf.} \cite[Theorem 6.5.2]{bombieri-gubler}}] \label{thm:moving_targets_BombieriGubler} Let $K$ be a number field, $S \subset \Vv_K$ a finite subset, $K'$ a finite extension on $K$ and $\kappa > 2$ a real number. For every place $v \in S$ let us fix an embedding $K' \to \C_v$.

%There is no sequence of $(\# S + 1)$-tuples $(x_i, \{ a_{i,v} : v \in S\})$  with $i \in \N$ made of a $K$-rational point $x_i$ of $\P^1$ and for every $v \in S$ of a $K'$-rational point $a_{i, v}$ of $\P^1$ distinct from $x$ satisfying the following properties : 
%\begin{enumerate}[(1)]
%\item for every $v \in S$ we have $h(a_{i, v}) = o(h(x_i))$ as $i$ goes to infinity;
%\item for all $i \in \N$ the following inequality is satisfied:
%$$$ \frac{1}{[K : \Q]} \sum_{v \in S} \m_v(a_{i,v}^{(\sigma_v)}, x) \ge \kappa h(x_i).$$
%\end{enumerate}
%\end{theorem}

Vojta's original form of Roth's Theorem with moving targets is more general, in the sense that it allows the target points to be $K$-rational:

\begin{theorem}[{\textit{cf.} \cite[Theorem 1]{vojta_movingtargets}}] \label{thm:VojtaRothMovingTargets}Let $K$ be a number field, $S \subset \Vv_K$ a finite subset, $q \ge 1$ a positive integer and $\kappa > 2$ a real number.  

There is no sequence of couples $(x_i, a^{(1)}_i, \dots, a^{(q)}_i)$ with $i \in \N$ of $(q+1)$-tuples made of pairwise distinct\footnote{\em{i.e.} for all $i \in \N$ we have $a^{(\sigma)}_i \neq a^{(\tau)}_i$ for every $\sigma \neq \tau$ and we have $x_i \neq a^{(\sigma)}_i$ for every $\sigma = 1, \dots, q$.} $K$-rational points of $\P^1$ satisfying the following properties : 
\begin{enumerate}[(1)]
\item for all $\sigma = 1, \dots, q$ we have $h(a_i^{(\sigma)}) = o(h(x_i))$ as $i$ goes to infinity;
\item for all $i \in \N$ the following inequality is satisfied:
$$ \frac{1}{[K : \Q]} \sum_{\sigma = 1}^q \m_S(a_i^{(\sigma)}, x_i) \ge \kappa h(x_i).$$
\end{enumerate}
\end{theorem}

Note that Theorem \ref{thm:VojtaRothMovingTargets} implies Theorem \ref{thm:moving_targets} by means of extending scalars from $K$ to a Galois closure of $K'$ over $K$ and taking the points $a_i^{(\sigma)}$ to be the conjugated points of the points $a_i$. We ignore at the moment if such a statement can be obtained by the methods expounded in the present paper.

Let us conclude this introduction remarking that for $q = 1, 2$ all these results are a straightforward consequence of Proposition \ref{prop:ProjectiveLiouvilleInequality}, which moreover gives an explicit upper bound for height of the points $x_i$ in terms of the height of the points $a_i^{(\sigma)}$. However for $q \ge 3$ this result is ineffective in the sense such that an explicit bound is not known.
\end{npar}

\begin{npar}{Main Effective Lower Bound} As explained above, there is an intermediate step in the proof of Roth's Theorem which is effective and implies Roth's Theorem through an elementary argument by contradiction that we shall repeat in the next paragraph --- this is the principal cause of loss of effectiveness.

This intermediate effective result is a lower bound of the height of $K$-rational points $x_1, \dots, x_n$ in terms of their $v$-adic distances from the algebraic points $a_1, \dots, a_n$. Although this type of lower bounds plays a crucial role in the seminal work of Bombieri \cite{bombieri}, it is rarely stated as a stand-alone theorem.

We name this lower bound ``Main Effective Lower Bound'' and the aim of this paper is to prove it by means of Geometric Invariant Theory. The statement of this result involves some auxiliary real numbers of geometric nature $r_1, \dots, r_n$: in the proof they are interpreted as the multi-degree of an invertible sheaf on $(\P^1)^n$.

To state it we need to introduce the crucial concepts that govern the combinatorics in Roth's Theorem:

\begin{deff} Let $q, n \ge 1$ be positive integers and let $t \ge 0$ and $\delta \in [0, 1]$ be real numbers.
\begin{enumerate}[(1)]
\item Let us consider the set $ \triangleup_n(t) \df \{ (\zeta_1, \dots, \zeta_n) \in [0, 1]^n : \zeta_1 + \cdots + \zeta_n < t \}$.

\item Let  $t_{q, n}(\delta) \in [0, n]$ be the unique real number such that
$$ 1 - q \vol \triangleup_n(t_{q, n}(\delta)) = \delta,$$
the volume being taken with respect the Lebesgue measure of $\R^n$. 

The function $t_{q, n} : [0, 1] \to [0, n]$ defined in this way is continuous. 

\item Let $R_{q, n}(\delta)$ be the unique positive real number such that
$$ \left( 1 + \frac{q-1}{R_{q, n}(\delta)}\right)^{n-1} - 1 = \delta \sqrt[n]{\delta}.$$

\item If $r = (r_1, \dots, r_n)$ is a $n$-tuple of real numbers we write $|r| = r_1 + \cdots + r_n$.
\end{enumerate}
\end{deff}

We are now able to state the Main Effective Lower Bound (\textit{cf.} \cite[Theorem 2]{bombieri} for $n = 2$):

\begin{theorem}[Main Effective Lower Bound] \label{thm:fundamental_effective_LB} Let $K'$ be a finite extension of $K$ of degree $q \ge 2$ and let $S \subset \Vv_K$ be a finite set of finite places. 

Let $n\ge 2$ be a positive integer, let $0 <\delta <  1 / (2 \cdot n!)$ be a real number and let $r = (r_1, \dots, r_n)$ be an $n$-tuple of positive real numbers such that $r_i / r_{i + 1} > R_{q, n}(\delta)$ for all $i = 1, \dots, n -1$. 

Then, for all $i = 1, \dots, n$ and for all couples $(x_i, a_i)$ made of a $K$-rational point $x_i$ of $\P^1$ and a $K'$-rational point $a_i$ of $\P^1$ such that $K(a_i) = K'$, the following inequality holds:
\begin{multline*}
 \frac{1}{[K : \Q]} t_{q, n}(\delta)  \sum_{v \in S} \left( \min_{\sigma : K' \to \C_v} \left\{ \min_{i = 1, \dots, n} \left\{ r_i \m_v(a_i^{(\sigma)}, x_i) \right\} \right\} \right)
\\ \le (1 + 2q \sqrt[n]{\delta}) \sum_{i = 1}^n r_i h(x_i) + \frac{q}{\delta} \sum_{i = 1}^n r_i h(a_i)
+  \left( \frac{\log \sqrt{2q}}{\delta} + \log 8 \right) |r|,
\end{multline*}
where, for every place $v \in \Vv_K$, the embeddings $\sigma : K' \to \C_v$ are meant to be $K$-linear.
\end{theorem}
\end{npar}

\begin{npar}{Deducing Roth's Theorem from the Main Effective Lower Bound} \label{DeducingRothTheoremFromMELB} Let us show how the Main Effective Lower Bound (Theorem \ref{thm:fundamental_effective_LB}) implies Roth's Theorem with moving targets (Theorem \ref{thm:moving_targets}).

Let us begin with the following bound which goes back to the work of Roth and it is based on an explicit version of a phenomenon of concentration of measure (see \cite{milman}).  As we will see this is where the number $2$ in Roth's Theorem comes from.

\begin{lem} \label{lem:eq:MeasureConcentrationRoth} Let $q, n \ge 1$ be positive integers. We have $t_{q,n}(0) \ge n/2 - \sqrt{(n \log q) / 6}$.
In particular,
$$\liminf_{n \to \infty} \frac{n}{t_{q,n}(0)}= 2.$$
\end{lem}
\begin{proof} According to \cite[Lemma 6.3.5]{bombieri-gubler} for every $0 \le \epsilon \le 1/2$ we have:
$$ \vol \triangleup_n\left( \left( \frac{1}{2} - \epsilon \right) n \right) \le \exp(-6n \epsilon^2).$$
The result is obtained taking $ \epsilon \df  1/2 - t_{q,n}(0) / n$.
\end{proof}

\begin{proof}[Proof of Theorem \ref{thm:moving_targets}] By contradiction, suppose that there exists an admissible sequence $\{ (x_i, a_i)\}_{i \in \N}$ verifying the conditions in the statement of Theorem \ref{thm:moving_targets}. Up to extracting a subsequence and passing to a sub-extension of $K'$, we may assume that we have $K(a_i) = K'$ for all $i \in \N$. We may also assume $q = [K' : K] \ge 2$. 

Fix a positive real number $\epsilon > 0$. By a pigeonhole argument (the so-called ``Mahler's Trick'', see \cite[6.4.2]{bombieri-gubler}) there exists an infinite subset $I_\epsilon \subset \N$ such that, for every place $v \in S$ there exists a positive real number $\lambda(\epsilon, v)$ which verifies, for every $i \in I_\epsilon$,
$$ \lambda(\epsilon, v) \m_S(a_i, x_i) \le  \m_v(a_i, x_i) \le \left( \lambda(\epsilon, v) + \frac{\epsilon}{ \#S} \right) \m_S(a_i, x_i) $$
(where we dropped the writing of the embeddings $\sigma_v$'s) and  
$$1 - \epsilon \le \sum_{v \in S} \lambda(\epsilon, v) \le 1. $$
Therefore up to renumbering the subsequence $\{ (x_i, a_i)\}_{i \in I_\epsilon}$ we may assume that the previous conditions are satisfied for all $i \in \N$.

Take an integer $n \ge 2$, a positive real number $\delta$ and $n$-tuple of positive real numbers $r$ satisfying the conditions in the statement of Theorem \ref{thm:fundamental_effective_LB}. Applying it to the couples $(x_i, a_i)$ for $i = 1, \dots, n$ and using, for every place $v \in S$,
$$\min_{\sigma : K' \to \C_v} \left\{ \min_{i = 1, \dots, n} \left\{ r_i \m_v(a_i^{(\sigma)}, x_i) \right\} \right\}  \le \min_{i = 1, \dots, n} \left\{ r_i \m_v(a_i^{(\sigma_v)}, x_i) \right\}, $$
 we obtain:
\begin{align*}
( 1 + 2q \sqrt[n]{\delta}) \sum_{i = 1}^n r_i h(x_i) + \frac{1}{\delta} \left( \sum_{i = 1}^n r_i \left( q h(a_i)+ C \right) \right) \hspace{-100pt}&  \\ 
&\ge \frac{1}{[K : \Q]} t_{q, n}(\delta)  \sum_{v \in S} \min_{i = 1, \dots, n} \left\{ r_i \m_v(a_i, x_i)  \right\}  \\
&\ge \frac{1}{[K : \Q]} t_{q, n}(\delta) (1 - \epsilon) \min_{i = 1, \dots, n} \left\{ r_i \m_S(a_i, x_i)  \right\},
\end{align*}
where, being rough, we set $C \df \log \sqrt{2q} + \log 8$.

By hypothesis for all $i = 1, \dots, n$ we have $\m_S(a_i, x_i) \ge [K : \Q] \kappa h(x_i)$. Thus,
\begin{multline*}
\kappa t_{q, n}(\delta) (1 - \epsilon) \min_{i = 1, \dots, n} \left\{ r_i h(x_i) \right\} \\ \le ( 1 + 2q \sqrt[n]{\delta} ) \sum_{i = 1}^n r_i h(x_i) + \frac{1}{\delta} \left( \sum_{i = 1}^n r_i \left(q h(a_i)+ C \right) \right).
\end{multline*}

The key point is that, since we have infinitely many $x_i$, Northcott's Principle entails that, extracting a subsequence we may suppose that the ratios of the heights $h(x_{i + 1}) / h(x_i)$ are sufficiently big (namely bigger that $R_{q, n}(\delta)$) so that we can take $r$ such that
$$ r_i h(x_i) = r_j h(x_j),$$
for every $i, j = 1, \dots, n$. Dividing the preceding inequality by $r_1 h(x_1) = \min \{ r_i h(x_i)\}$, we get: 
\begin{align*} \kappa t_{q, n}(\delta) (1 - \epsilon) &\le ( 1 + 2q \sqrt[n]{\delta}) n + \frac{q}{\delta} \sum_{i = 1}^n \frac{h(a_i)}{h(x_i)} + \frac{r_1 + \cdots + r_n}{r_1} \frac{C}{\delta h(x_1)}
\end{align*}
Now, since $h(a_i)  = o (h(x_i))$ as $i$ goes to infinity, passing to a subsequence we may suppose that we have $h(a_i) \le \delta \sqrt[n]{\delta} h(x_i)$. By Northcott's Principle, the ratios $r_i / r_{i + 1}$ and the height $h(x_1)$ can be supposed arbitrarily big. Thus,
$$ \kappa t_{q, n}(\delta) (1 - \epsilon) \le ( 1 + 3q \sqrt[n]{\delta}) n. $$
Letting $\delta$ and $\epsilon$ go to $0$ and $n$ go to infinity, according to Lemma \ref{lem:eq:MeasureConcentrationRoth} we find
$$ \kappa \le \liminf_{n \to \infty} \frac{n}{t_{q, n}(0)} = 2$$
which contradicts the hypothesis $\kappa > 2$.
\end{proof}
\end{npar}

\subsection{Statement of the Main Theorem}

\begin{npar}{More combinatorial data} It is convenient to fix some more notations on the combinatorics appearing in the study. 

\begin{deff} \label{def:DefinitionOfMoreCombinatorics} Let $q, n \ge 1$ be positive integers, $r = (r_1, \dots, r_n)$ be a $n$-tuple of positive real numbers and $t \ge 0$ be a non-negative integer.

\begin{enumerate}[(1)]
\item We consider the following subsets of $\R^n$:
\begin{align*}
\square_{r} &\df \left\{ (\zeta_1, \dots, \zeta_n) \in \R^n : 0 \le \zeta_i \le r_i \textup{ for every } i = 1, \dots, n \right\} = \prod_{i = 1}^n [0, r_i] \\
\triangledown_{r}(t) & \df \left\{ (\zeta_1, \dots, \zeta_n) \in \square_{r} : \frac{\zeta_1}{r_1} + \cdots + \frac{\zeta_n}{r_n} \ge t \right\} \\
\triangleup_{r}(t) & \df \left\{ (\zeta_1, \dots, \zeta_n) \in \square_{r} : \frac{\zeta_1}{r_1} + \cdots + \frac{\zeta_n}{r_n} < t \right\}= \square_{r} - \triangledown_{n, r}(t).
\end{align*}
We add $\Z$ in superscript to denote the intersection of these sets with $\Z^n$ (we write $\square^\Z_{r}$, $\triangledown_{r}^\Z(t)$ and $\triangleup_{r}^\Z(t)$). If $r = (1, \dots, 1)$ we write $\square_n$, $\triangledown_n(t)$ and $\triangleup_n(t)$.

\item If $\lambda_n$ is the Lebesgue measure on $\R^n$, we consider the function $\mu_n : [0, n] \to \R$,
$$ \mu_n(t) \df \int_{\nabla_n(t)}  ( 2 \zeta_1  - 1 ) \ d \lambda_n = \int_{\triangleup_n(t)} ( 1 - 2 \zeta_1)  \ d \lambda_n.$$
\item We define:
$$ \epsilon_{q, r} \df \prod_{i = 1}^{n- 1} \left( 1 +  \max_{i + 1 \le j \le n} \left\{ \frac{r_j}{r_i} \right\} (q-1)\right) - 1.$$

\item Let us denote by $u_{q, r}(t)$ the unique real number in $[0, n]$ such that
$$ \vol \triangleup_{n}(u_{q, r}(t)) = \min \left\{ \max\left\{1 + \epsilon_{q, r} - q \vol \triangleup_n(t), 0 \right\}, 1 \right\}.$$

\end{enumerate}
\end{deff}

\begin{lem} \label{lem:PropertiesFunctionMu} The function $\mu_n : [0, n] \to \R$ is strictly concave\footnote{That is, for every $t_1 < t_2$ in $[0, n]$ and every $\xi \in \left]0, 1 \right[$ we have $$\mu_n(\xi t_1 + (1 - \xi) t_2) > \xi \mu_n(t_1) + (1 - \xi) \mu_n(t_2).$$}. Moreover the following properties are satisfied:
\begin{enumerate}[(1)]
\item If $t \in [0, 1]$ we have 
$$\mu_n(t) = \frac{t^n}{n!} \left( 1 - 2 \frac{t}{n+1} \right) $$
\item We have $\mu_n(t) \ge 0$ for every $t \in [0, n]$.
\item For every $t \in [0, n]$ we have $\mu_n(n - t) = \mu_n(t)$.
\item The function $\mu_n$ is increasing on $[0, n/2]$ and decreasing on $[n/2, n]$;
\end{enumerate}
\end{lem}

\begin{proof} (1) and (3) are elementary computations that we leave to the reader. When $n= 1$ statement (1) entails the strict concavity of $\mu_1$. For $n > 1$ arbitrary the strict concavity of $\mu_n$ is proved by induction thanks to the relation
$$ \mu_n(t) = \int_{0}^{\min\{t, 1\}} \mu_{n - 1}(t - \zeta_n) \ d \lambda_{1}(\zeta_n). $$
(2) follows from $\mu_n(0) = \mu_n(n) = 0$ and the concavity of $\mu_n$. (4) follows from (3) and the strict concavity of $\mu_n$.
\end{proof}
\end{npar}

\begin{npar}{Main Theorem} Keeping the notation just introduced, the main technical result of the present paper is the following:

\begin{theorem}[Main Theorem] \label{thm:ApplyingComparisonFormula} Let $K'$ be a finite extension of $K$ of degree $q \ge 2$ and let $S \subset \Vv_K$ be a finite subset. Let $n \ge 2$ be an integer, $t_{\bs{a}}, t_x \ge 0$ non-negative real numbers and let $r = (r_1, \dots, r_n)$ be an $n$-tuple of positive real numbers. If the following inequality is satisfied,
\begin{equation}
\tag{SS} \label{eq:SemiStabilityConditionMainTheorem} \mu_n(u_{q,r}(t_{\bs{a}})) > \mu_n(t_x) + \epsilon_{q, r},
\end{equation}
then, for all $i = 1, \dots, n$ and for all couples $(x_i, a_i)$ made of a $K$-point $x$ of $\P^1$ and $K'$-point $a_i$ of $\P^1$ such that $K(a_i) = K'$, the following inequality holds:
\begin{multline*}
\frac{1}{[K : \Q]}(1 - q \vol \triangleup_n(t_{\bs{a}})) t_{\bs{a}}  \sum_{v \in S} \left( \max_{\sigma : K' \to \C_v} \left\{ \min_{i = 1, \dots, n} \left\{ r_i \m_v(a_i^{(\sigma)}, x_i)  \right\} \right\} \right)
\\
\le C_{q,r}^{(1)}(t_{\bs{a}}, t_x) \sum_{i = 1}^n r_i h(x_i)
+ q C_{q,r}^{(2)}(t_{\bs{a}}, t_x)  \sum_{i = 1}^n r_i h(a_i)
+ C_{q,r}^{(3)}(t_{\bs{a}}, t_x) |r| ,
\end{multline*}
where
\begin{align*}
C_{q,r}^{(1)}(t_{\bs{a}}, t_x) &\df \int_{\triangledown_n(t_x)} \zeta_1 \ d\lambda_n + q \frac{\vol \triangleup(u_{q,r}(t_{\bs{a}})) - \mu_n(t_x)}{2}, \\
C_{q,r}^{(2)}(t_{\bs{a}}, t_x) &\df q \vol \triangleup_n(t_{\bs{a}}) + q \frac{\vol \triangleup(u_{q,r}(t_{\bs{a}})) - \mu_n(t_x)}{2}, \\
C_{q,r}^{(3)}(t_{\bs{a}}, t_x) &\df  \vol \triangleup_n(u_{q,r}(t_{\bs{a}})) \log \sqrt{6} + \vol \triangledown_n(t_x) \log \sqrt{8} + q \vol \triangleup_n(t_{\bs{a}}) \log \sqrt{2 q}.
\end{align*}
\end{theorem}

Theorem \ref{thm:ApplyingComparisonFormula} is interesting only when condition \eqref{eq:SemiStabilityConditionMainTheorem} is close to its limit of validity (that is, $\mu_n(u_{q,r}(t_{\bs{a}})) - \mu_n(t_x) - \epsilon_{q, r}$ is very small) and $1 - q \vol \triangleup_n(t_{\bs{a}})$ tends to zero. This will be the case that will lead us to Theorem \ref{thm:ApplyingComparisonFormula} in the proof that we shall give in the next paragraph.

The fact that this is the only interesting case may be formulated more precisely saying that, as soon as we set $1 - q \vol \triangleup_n(t_{\bs{a}}) = \delta$, Theorem \ref{thm:fundamental_effective_LB} entails Theorem \ref{thm:ApplyingComparisonFormula} with slightly bigger error terms, which are insignificant for applications and arise from simplifications in computations in the proof that follows.
\end{npar}

\subsection{From the Main Theorem to the Main Effective Lower Bound}

In this section we deduce the Main Effective Lower Bound (Theorem \ref{thm:fundamental_effective_LB}) from the Main Theorem (Theorem \ref{thm:ApplyingComparisonFormula}). First of all let us remark that since we supposed $r = (r_1, \dots, r_n)$ such that $r_i  / r_{i + 1} > R_{q, n}(\delta) $ for every $i = 1, \dots, n-1$, we have $ \epsilon_{q, r} < \delta \sqrt[n]{\delta}$.

\begin{npar}{Choice of the parameters} The Main Effective Lower Bound is deduced from Theorem \ref{thm:ApplyingComparisonFormula} setting
$$ t_{\bs{a}} \df t_{q, n}(\delta).$$
Let us also write $\tilde{u}_{q, r}(\delta) \df u_{q, r}(t_{q, n}(\delta))$.

\begin{lem} \label{lem:ProofOfTheMainTheoremEstimationOfTheParameters} With the notation introduced above, we have:
\begin{enumerate}[(1)]
\item $\vol \triangleup_n( \tilde{u}_{q, r}(\delta)) = \delta + \epsilon_{q, r} \le 1 / n!$, hence $\tilde{u}_{q, r}(\delta) \le 1$;
\item $\displaystyle \int_{\triangleup_{n}(\tilde{u}_{q, r}(\delta))} \zeta_1 \ d\lambda_n \le \frac{1}{2} (\delta + \epsilon_{q, r})^{\frac{n+1}{n}}$;
\item $\mu_n(\tilde{u}_{q,r}(\delta)) >  \epsilon_{q, r}$;
\item $\mu_n(\tilde{u}_{q,r}(\delta)) \le  \epsilon_{q, r} + \mu_n( \sqrt[n]{n! \delta})$.
\end{enumerate}
\end{lem}

\begin{proof} (1) follows from the definitions of $t_{q, n}(\delta)$ and $\tilde{u}_{q, r}(\delta)$ and the hypotheses on $\delta$ and $\epsilon_{q, r}$. Since $u_{q, r}(
\delta) \le 1$ we have $\tilde{u}_{q, r}(\delta) = \sqrt[n]{n!(\delta + \epsilon_{q, r})}$. 

(2) The latter expression of $\tilde{u}_{q, r}(\delta)$ gives
$$  \int_{\triangleup_{n}(\tilde{u}_{q, r}(\delta))} \zeta_1 \ d\lambda_n = \frac{\tilde{u}_{q, r}(\delta)^{n+1}}{(n+1)!} = (\delta + \epsilon_{q, r})^{\frac{n +1}{n}} \frac{\sqrt[n]{n!}}{n+1},$$
and we conclude by noticing $\sqrt[n]{n!} / (n+1) \le 1/2$ for all $n \ge 1$. 

(3) and (4) follow from the explicit expression given by Lemma \ref{lem:PropertiesFunctionMu} (1),
$$ \mu_{n}(\tilde{u}_{q, r}(\delta)) =  (\delta + \epsilon_{q, r}) \left( 1 - \frac{2}{n+1} \sqrt[n]{n!(\delta + \epsilon_{q, r})} \right)$$ and the hypotheses on $\delta$ and $\epsilon_{q, r}$. 
\end{proof}

For what concerns the choice of the parameter $t_x$, roughly speaking, we stress the validity of condition \eqref{eq:SemiStabilityConditionMainTheorem} to its limit. More precisely, since the function $\mu_n$ is strictly decreasing on $[n/2, n]$,  there exists a unique real number $w_{q, r}(\delta) \in [n /2 , n[ $ such that
$$ \mu_n(\tilde{u}_{q, r}(\delta))  =   \mu_n(w_{q, r}(\delta)) + \epsilon_{q, r}.$$

\begin{lem} \label{lem:EstimationCombinatorialTerms}With the notation introduced above we  have:
\begin{enumerate}[(1)]
\item $ \displaystyle \vol \triangledown_n(w_{q, r}(\delta)) \le \delta$;
\item $\displaystyle \int_{\triangledown_n(w_{q, r}(\delta))} \zeta_1 \ d\lambda_n \le \delta $;
\item $\displaystyle \vol \triangleup_n(\tilde{u}_{q, r}(\delta)) - \mu_n(w_{q, r}(\delta))  \le 3 \delta \sqrt[n]{\delta}$.
\end{enumerate}
\end{lem}

\begin{proof} (1) Indeed Lemma \ref{lem:ProofOfTheMainTheoremEstimationOfTheParameters} (3) entails $w_{q, r}(\delta) \ge n - \sqrt[n]{n! \delta}$. (2) This follows from (1). Indeed for every $t \in [n- 1, n]$ we have
$$ \int_{\vol \triangledown_n(t)} \zeta_1 \ d \lambda_n \le \vol \triangledown_n(t). $$

(3) By definition of $w_{q, r}(\delta)$ and by Definition \ref{def:DefinitionOfMoreCombinatorics} (2) we have:
\begin{align*}
\vol \triangleup_n(\tilde{u}_{q, r}(\delta)) - \mu_n(w_{q, r}(\delta))  &=   
\vol \triangleup_n(\tilde{u}_{q, r}(\delta)) + \epsilon_{q, r} -  \mu_n(\tilde{u}_{q, r}(\delta))  \\
&= 2 \int_{\triangleup_n(\tilde{u}_{q, r}(\delta))} \zeta_1 \ d \lambda_n  +  \epsilon_{q, r} \\
&\le (\delta + \epsilon_{q, r})^{\frac{n + 1}{n}} + \epsilon_{q, r},
\end{align*}
where in the last inequality we used Lemma \ref{lem:ProofOfTheMainTheoremEstimationOfTheParameters} (2). The result follows from the hypotheses $\epsilon_{q, r} < \delta \sqrt[n]{\delta}$ and $\delta < 1 / (2 \cdot n!)$. 
\end{proof}
\end{npar}

\begin{npar}{Application of the Main Theorem} Lemma \ref{lem:ProofOfTheMainTheoremEstimationOfTheParameters} (3) permits us to apply Theorem \ref{thm:ApplyingComparisonFormula}  with $t_{\bs{a}} = t_{q, n}(\delta)$ and $ t_x \in \left] w_{q, r}(\delta), n \right[$ close enough to $w_{q, r}(\delta)$. Letting $t_x$ tend to $w_{q, r}(\delta)$ and taking in account the estimates given by Lemma \ref{lem:EstimationCombinatorialTerms} we find:
\begin{multline*}
\frac{1}{[K : \Q]} \delta t_{q, n}(\delta) \sum_{v \in S} \left(  \max_{\sigma : K' \to \C_v}\left\{ \min_{i = 1, \dots, n} \left\{ r_i \m_v(a_i^{(\sigma)}, x_i) \right\} \right\} \right) \\
\le \delta \left(1 + \frac{3}{2}q \sqrt[n]{\delta} \right) \sum_{i = 1}^n r_i h(x_i) + q \sum_{i = 1}^n r_i h(a_i) + |r| C_{q, r}(\delta),
\end{multline*}
where we set $C_{q, r}(\delta) \df \delta( 1 + \sqrt[n]{\delta}) \log \sqrt{6} + \delta \log \sqrt{8} + (1 - \delta) \log \sqrt{2 q}$. This concludes the proof. \qed
\end{npar}

\section{Geometric Invariant Theory and Arakelov Geometry} \label{sec:GITArakelov}

\subsection{The Fundamental Formula}

Let $K$ be a number field and $\o_K$ its ring of integers.  

\begin{np} Let $\cal{X}$ be a projective and flat $\o_K$-scheme endowed with the action of an $\o_K$-reductive group\footnote{Over an algebraically closed field $k$ an algebraic group $G$ --- \em{i.e.} a smooth finite type affine $k$-group scheme --- is said to be \em{reductive} if it is connected and every normal smooth connected unipotent subgroup is trivial. Over an arbitrary scheme $S$ a group scheme $G$ is said to be \em{reductive} (or $G$ is a \em{$S$-reductive group}) if it satisfies the following conditions:
\begin{enumerate}
\item $G$ is affine, smooth and of finite type over $S$;
\item for all $s \in S$, the $\ol{s}$-group scheme $G_{\ol{s}} \df G \times_S \ol{s}$ is a reductive algebraic group (where $\ol{s}$ is the spectrum of an algebraic closure of the residue field $\kappa(s)$).
\end{enumerate}

Examples of $S$-reductive groups are $\GLs_{n, S}$, $\SLs_{n, S}$ and their products. In this paper we are interested in the $\o_K$-reductive group $\SLs_{2, \o_K}^n$. We refer the interested reader to \cite[Chapter IV]{borel} for the theory over a field and \cite{sga3, conrad} for the theory over an arbitrary scheme.} $\cal{G}$ and let $\cal{L}$ be a \em{very ample} $\cal{G}$-linearised invertible sheaf on $\cal{X}$. The global sections $\cal{E} = \Gamma(\cal{X}, \cal{L})$ are naturally endowed with a linear action of $\cal{G}$. Thus the reductive group $\cal{G}$ acts naturally on $\P(\cal{E}^\vee)$ and the invertible sheaf $\O_{\cal{E}^\vee}(1)$ is naturally $\cal{G}$-linearised. The closed embedding $j : \cal{X} \into \P(\cal{E}^\vee)$ and the natural isomorphism $j^\ast \O_{\cal{E}^\vee} (1)\iso \cal{L}$ are $\cal{G}$-equivariant.
\end{np}

\begin{np} A point $x \in \cal{X}$ is said to be \em{semi-stable} if there exists, for a sufficiently big $d \ge 1$, a $\cal{G}$-invariant global section $s \in \Gamma(\cal{X}, \cal{L}^{\otimes d})$ that does not vanish at $x$. 

Let us consider the $\o_K$-graded algebra of finite type $\cal{A} \df \bigoplus_{d \ge 0} \Gamma(\cal{X}, \cal{L}^{\otimes d})$. According to a theorem of Seshadri \cite[II.4, Theorem 4]{seshadri77} the graded algebra $$\cal{A}^{\cal{G}} = \bigoplus_{d \ge 0} \Gamma(\cal{X}, \cal{L}^{\otimes d})^{\cal{G}}$$ of $\cal{G}$-invariants of $\cal{A}$ is an $\o_K$-algebra of finite type and projective scheme $\cal{Y} \df \Proj \cal{A}^{\cal{G}}$ is the categorical quotient of the open subset $\cal{X}^\ss$ of semi-stable points  of $\cal{X}$ (with respect to the action of reductive group $\cal{G}$ and the invertible sheaf $\cal{L}$). For this reason we denote it  sometimes by $\quotss{\cal{X}}{\cal{G}}$ (or by $\quotss{(\cal{X}, \cal{L})}{\cal{G}}$ to keep track of the polarisation). Let $\pi : \cal{X}^\ss \to \cal{Y}$ be the quotient morphism. Since $\cal{Y}$ is of finite type, for every sufficiently divisible integer $D \ge 1$, there exists an ample invertible sheaf $\cal{M}_D$ on $\cal{Y}$ and a $\cal{G}$-equivariant isomorphism of invertible sheaves
$$ \phi_D : \pi^\ast \cal{M}_D \too \cal{L}^{\otimes D}_{\rvert \cal{X}^\ss}.$$
\end{np}

\begin{np} Let $\gamma : K \to \C$ be an embedding. Let $\| \cdot \|_{\cal{E}, \gamma}$ be an hermitian norm on $\cal{E} \otimes_\gamma \C$ which is invariant under the action of a maximal compact subgroup of $\cal{G}_\gamma(\C)$. We suppose that the family of norms $\{\| \cdot \|_{\cal{E}, \gamma} \}_{\gamma : K \to \C}$ is compatible under complex conjugation. 

Let $\| \cdot \|_{\O(1), \gamma}$ be the Fubini-Study metric on the invertible sheaf $\O_{\cal{E}^\vee}(1)$ associated to the hermitian norm $\| \cdot \|_{\cal{E}^\vee, \gamma}$ and let $\| \cdot \|_{\cal{L}, \gamma}$ be its restriction to $\cal{L}$. We denote by $\ol{\cal{L}}$ the hermitian line bundle on $\cal{X}$ obtained endowing $\cal{L}$ with the family of metrics $\{ \| \cdot \|_{\cal{L}, \gamma} \}_{\gamma : K \to \C}$.
For every $y \in \cal{Y}_\gamma(\C)$ and every $t \in y^\ast \cal{M}_D$ we set
$$ \| t \|_{\cal{M}_D, \gamma} (y) \df \sup_{\substack{x \in \cal{X}^\ss_\gamma(\C) \\ \pi(x) = y}} \| \phi_D(\pi^\ast t)\|_{\cal{L}^{\otimes D}, \gamma}(x). $$

\begin{lem} \label{lem:ExtensionOfGInvariantGlobalSections} Let $f \in \Gamma(\cal{Y}, \cal{M}_D)$ be a global section. 
\begin{enumerate}[(1)]
\item There exists a unique $\cal{G}$-invariant global section $\tilde{f} \in \Gamma(\cal{X}, \cal{L}^{\otimes D})$ which vanishes identically on $\cal{X} - \cal{X}^\ss$ and such that $\phi_D(\pi^\ast f) = \tilde{f}_{\rvert \cal{X}^\ss}$.
\item For every complex embedding $\gamma : K \to \C$ we have
$$ \sup_{y \in \cal{Y}_\gamma(\C)} \| f \|_{\cal{M}_D, \gamma}(y) = \sup_{x \in \cal{X}_\gamma(\C)} \|\tilde{f} \|_{\cal{L}^{\otimes D}, \gamma}(x). $$
\end{enumerate}
\end{lem}

In particular we have $\| t \|_{\cal{M}_D, \gamma} (y) < + \infty$ for every $t \in y^\ast \cal{M}_D$. Therefore the function $\| \cdot \|_{\cal{M}_D, \gamma}$ defines a metric on the invertible sheaf $\cal{M}_D$.

\begin{proof} (1) is a reformulation of the definition of $\cal{Y}$ and $\cal{M}_D$. (2) follows from (1).
\end{proof}
\end{np}

\begin{np} We denote by $\ol{\cal{M}}_D$ the associated hermitian invertible sheaf on $\cal{Y}$ and by $h_{\ol{\cal{M}}_D}$ the height function given by $\ol{\cal{M}}_D$ (see \cite[2.7.17]{bombieri-gubler}). Let us define
$$ h_{\min}\left( \quotss{(\cal{X}, \ol{\cal{L}})}{\cal{G}}\right) \df  \inf_{Q \in \cal{Y}(\ol{\Q})} \frac{1}{D} h_{\ol{\cal{M}}_D}(Q) \in [-\infty, + \infty[.$$
(which is clearly independent of $D$).

\begin{lem} \label{lem:LowestHeightOnTheQuotientIsReal} We have $h_{\min}\left( \quotss{(\cal{X}, \ol{\cal{L}})}{\cal{G}}\right) > -\infty$. \end{lem}

\begin{proof} Let $D$ be such that $\cal{M}_D$ is very ample and let $t_1, \dots, t_N \in \Gamma(\cal{Y}, \cal{M}_D)$ be a set of generators of the global sections. Let $K'$ be a finite extension of $K$, let $Q$ be a $K'$-point of $\cal{Y}$ and $\epsilon_Q$ the associated $\o_{K'}$-point of $\cal{Y}$ given by the valuative criterion of properness. There exists $i \in \{ 1, \dots, N\}$ such that $t_i$ does not vanish at $Q$. By definition of the height we have
\begin{align*} 
[K' : \Q] h_{\ol{\cal{M}}_D}(Q)&\df  \log \# \left( \epsilon_Q^\ast \cal{M}_D  / (\epsilon_Q^\ast t_i \cdot \epsilon_Q^\ast \cal{M}_D) \right) - \sum_{\gamma : K' \to \C} \log \| t_i \|_{\cal{M}_D, \gamma}(Q) \\
&\ \ge - [K' : K] \sum_{\gamma : K \to \C} \left(  \sup_{y \in \cal{Y}_\gamma(\C)} \log \| t_i \|_{\cal{M}_D, \gamma}(y) \right).
\end{align*}
It suffices to show that for every $i = 1, \dots, N$ and every $\gamma : K \to \C$ the function $\| t_i \|_{\cal{M}_D, \gamma}$ is uniformly bounded on $\cal{Y}_\gamma(\C)$. This follows from Lemma \ref{lem:ExtensionOfGInvariantGlobalSections} (2) and concludes the proof.
\end{proof}

\begin{rem} Proving this Lemma would have been unnecessary if knew that the metric $\| \cdot \|_{\cal{M}_D, \gamma}$ was continuous. This is actually the case but in an attempt to be self-contained we avoided the recourse to such a result (see Kirwan \cite[Chapter 8, \S2]{git}, Burnol \cite{burnol92} and Schwarz \cite[Chapter 5]{schwarz} for the continuity in this setting or Zhang \cite[Theorem 4.10]{ZhangSemistableVarieties} and \cite[Th\'{e}or\`{e}me II.2.18]{maculan_these} for a more general result).
\end{rem}
\end{np}

\begin{npar}{Instability measure} \label{par:DefinitionInstabilityMeasure} Let $v$ be a place of $K$. If $v$ is non-archimedean we denote by $\| \cdot \|_{\cal{L}, v}$ (resp. $\| \cdot \|_{\cal{M}_D, v}$) the continuous and bounded metric induced by the integral model $\cal{L}$ (resp. $\cal{M}_D$)\footnote{Let $x$ be a $\C_v$-point of $\cal{X}$. Since $\cal{X}$ is proper, the $\C_v$-point $x$ gives rise to a $\ol{\o}_v$-point $\epsilon_x$ of $\cal{X}$, where $\ol{\o}_v$ is the ring of integers of $\C_v$. The invertible sheaf $\epsilon_x^\ast \cal{L}$ is a free $\ol{\o}_v$-module of rank $1$ : choose a basis $s_0$. Every other element $s \in x^\ast \cal{L}$ can be written in a unique way as $s = \lambda s_0$ with $\lambda \in \C_v$. We set
$$ \| s \|_{\cal{L}, v} (x) \df |\lambda|_v.$$
Clearly this does not depend on the chosen basis $s_0$ of $\epsilon_x^\ast \cal{L}$. See also \cite[Example 2.7.20]{bombieri-gubler}.}. For a $\C_v$-point $x \in \cal{X}(\C_v)$ we define its \em{$v$-adic instability measure}:
$$\iota_v(x) \df - \log \sup_{g \in \cal{G}(\C_v)} \frac{\| g \cdot s\|_{\cal{L}, v}(g \cdot x)}{\| s \|_{\cal{L}, v}(x)} \in [-\infty, 0]. $$
where $s \in x^\ast \cal{L}$ is a non-zero section. 
Clearly this does not depend on the chosen section $s$. If $\hat{x}$ is a generator of the line $j(x) \in \P(\cal{E}^\vee)(\C_v)$ we have
$$ \iota_v(x) = \log \inf_{g \in \cal{G}(\C_v)} \frac{\| g \cdot \hat{x}\|_{\cal{E}^\vee, v}}{\| \hat{x}\|_{\cal{E}^\vee, v}}.$$

\begin{prop} Let $v$ be a place of $K$. For every $\C_v$-point $x \in \cal{X}^\ss(\C_v)$ and every non-zero section $t \in \pi(x)^\ast \cal{M}_D$ we have
$$\iota_v(x) \ge  - \frac{1}{D} \log \frac{\| t \|_{\cal{M}_D, v}(\pi(x))}{ \| \pi^\ast t \|_{\cal{L}^{\otimes D}, v}(x)}.$$
\end{prop}

\begin{proof} In the archimedean case this is clear by definition of the metric $\| \cdot \|_{\cal{M}_D}$ and the $\cal{G}$-invariance of $\pi$. Let us suppose that $v$ is non-archimedean. Up to taking a power of $\cal{M}_D$ we may assume that $\cal{M}_D$ is very ample.

Let $y \df \pi(x)$ and let $\epsilon_y \in \cal{Y}(\ol{\o}_v)$ the unique $\ol{\o}_v$-valued point of $\cal{Y}$ associated to $y$ by the valuative criterion of properness (where $\ol{\o}_v$ is the ring of integers of $\C_v$). Up to rescaling $t$ we may assume that $t$ is basis of the free $\ol{\o}_v$-module $\epsilon_y^\ast \cal{M}_D$ and thus $\| t \|_{\cal{M}_D, v}(y) = 1$. 
%With this notation we have
%$$ \log \frac{\| t \|_{\cal{M}_D, v}(y)}{ \| \pi^\ast t \|_{\cal{L}^{\otimes D}, v}(x)}  = 0,$$
%thus we are led back to prove $\iota_v(x) \ge 0$.

Since $\cal{M}_D$ is generated by its global sections, there exists $f \in \Gamma(\cal{Y}, \cal{M}_D) \otimes\ol{\o}_v$ such that $\epsilon_y^\ast f = t$. According to Proposition \ref{lem:ExtensionOfGInvariantGlobalSections}, the rational section $\pi^\ast f$ extends uniquely to a $\cal{G}$-invariant global section $\tilde{f} \in \Gamma( \cal{X}, \cal{L}^{\otimes D}) \otimes \ol{\o}_v$ which vanishes identically outside $\cal{X}^\ss$. 

Let us fix $g \in \cal{G}(\C_v)$. Since the section $\tilde{f}$ is integral, we have
$$\| \pi^\ast f \|_{\cal{L}^{\otimes D}, v}(g \cdot x) =  \| \tilde{f} \|_{\cal{L}^{\otimes D}, v}(g \cdot x) \le 1,$$
and recalling $\| t \|_{\cal{M}_D, v}(y) = 1$ this entails $\| \pi^\ast t \|_{\cal{L}^{\otimes D}, v}(g \cdot x) \le \| t \|_{\cal{M}_D, v}(y)$. Taking the supremum over all $g \in \cal{G}(\C_v)$ we find
\begin{align*}
\iota_v(x) &= - \frac{1}{D} \log \sup_{g \in \cal{G}(\C_v)} \frac{\| \pi^\ast t \|_{\cal{L}^{\otimes D}, v}(g \cdot x)}{\| \pi^\ast t \|_{\cal{L}^{\otimes D}, v}(x)} \ge - \frac{1}{D} \log \frac{\| t \|_{\cal{M}_D, v}(y)}{ \| \pi^\ast t \|_{\cal{L}^{\otimes D}, v}(x)},
\end{align*}
which is what we wanted to prove.
\end{proof}

\begin{rem} For a non-archimedean place $v$, it follows from the proof that in the preceding Proposition we have equality if the reduction $\tilde{x}$ of the point $x$ at the place $v$ is semi-stable, \em{i.e.} it is a semi-stable $\ol{\F}_v$-point of the scheme $\cal{X} \times_{\o_K} \ol{\F}_v$ under the action of $\cal{G}(\ol{\F}_v)$ (where $\ol{\F}_v$ is the residue field of $\C_v$).
\end{rem}

\end{npar}

\begin{npar}{Fundamental formula} Summing up the previous considerations we obtain : 

\begin{theorem}[Fundamental Formula] \label{thm:ComparisonFormula} Let $P \in \cal{X}(K)$ be a semi-stable point. Then for almost all places $v \in \Vv_K$ the instability measure $\iota_v(P)$ is zero and we have the inequality : 
$$ h_{\ol{\cal{L}}}(P) + \frac{1}{[K : \Q]} \sum_{v \in \Vv_K} \iota_v(P) \ge \frac{1}{D} h_{\ol{\cal{M}}_D}(\pi(P)).$$
\end{theorem}
%As mentioned before one can prove that this is actually an equality (that is why we name this Theorem as ``Fundamental Formula'' --- see \cite[Scholie III.2.2]{maculan_these}) but we will not need this. 
In practice we use Theorem \ref{thm:ComparisonFormula} through this immediate Corollary : 
\begin{cor} \label{cor:ComparisonInequality} For every semi-stable point $P \in \cal{X}^\ss(K)$ we have
$$ h_{\ol{\cal{L}}}(P) + \frac{1}{[K : \Q]} \sum_{v \in \Vv_K} \iota_v(P) \ge  h_{\min}\left( \quotss{(\cal{X}, \ol{\cal{L}})}{\cal{G}} \right).$$
\end{cor}

\end{npar}

\subsection{Lower bound of the height on the quotient} \label{sec:LowerBoundHeightQuotient}

\begin{npar}{Statement of the lower bound} Let $\ol{\cal{E}} = (\ol{\cal{E}}_1, \dots, \ol{\cal{E}}_n)$ be a $n$-tuple of $\o_K$-hermitian vector bundles of positive ranks. Suppose we are given a $\o_K$-hermitian vector bundle $\ol{\cal{F}}$ and a representation, that is a morphism of $\o_K$-group schemes,
$$ \rho : \GLs(\cal{E})\df \GLs(\cal{E}_1) \times_{\o_K} \cdots \times_{\o_K} \GLs(\cal{E}_n) \too \GLs(\cal{F}),$$
 which is \em{unitary}, \em{i.e.}, for every embedding $\sigma : K \to \C$, the action of the compact subgroup $$\U(\cal{E})_\sigma \df \U(\| \cdot \|_{\cal{E}_1, \sigma}) \times \cdots \times \U(\| \cdot \|_{\cal{E}_n, \sigma}) \subset \GLs(d)_\sigma(\C)$$ respects the hermitian norm $\| \cdot \|_{\cal{F}, \sigma}$.

\begin{theorem} \label{thm:LBHeightQuotient} With the notation introduced above, suppose that we are given an $n$-tuple of integers $b = (b_1, \dots, b_n)$ and a homomorphism of hermitian vector bundles
$$ \varpi : \ol{\cal{E}}_1^{\otimes b_1} \otimes \cdots \otimes \ol{\cal{E}}_n^{\otimes b_n} \too \ol{\cal{F}}$$
generically surjective and $\GLs(\cal{E})$-equivariant. Then, 
$$ h_{\mini}\left( \quotss{(\P(\cal{F}^\vee), \O_{\ol{\cal{F}}^\vee}(1))}{\SLs(\cal{E})} \right) \ge \sum_{i = 1}^n b_i \muar(\ol{\cal{E}}_i) - \frac{1}{2} \sum_{i = 1}^n |b_i| \log \rk \cal{E}_i,$$
where $\O_{\ol{\cal{F}}^\vee}(1)$ is equipped with the natural Fubini-Study metric given by $\ol{\cal{F}}$ and $\SLs(\cal{E})$ is the $\o_K$-reductive group $\SLs(\cal{E}_1) \times_{\o_K} \cdots \times_{\o_K} \SLs(\cal{E}_n)$.
\end{theorem}

\begin{rem}
This statement is more general than \cite[Theorem 4.2]{chen_ss} in the following sense: with our notation Chen proves that for every semi-stable $K$-point $P$ of $\P(\cal{F}^\vee)$ we have
$$ h_{\O_{\ol{\cal{F}}^\vee}(1)}(P) \ge \sum_{i = 1}^n b_i \muar(\ol{\cal{E}}_i) - \frac{1}{2} \sum_{i = 1}^n |b_i| \log \rk \cal{E}_i.$$
Chen's result is deduced from Theorem \ref{thm:LBHeightQuotient} thanks to the inequality given by Corollary \ref{cor:ComparisonInequality}
$$  h_{\O_{\ol{\cal{F}}^\vee}(1)}(P) \ge h_{\mini}\left( \quotss{(\P(\cal{F}^\vee), \O_{\ol{\cal{F}}^\vee}(1))}{\SLs(\cal{E})} \right).$$
\end{rem}

\begin{rem} In the proof of Theorem \ref{thm:LBHeightQuotient} we can limit ourselves to consider the case where the integer $b_i$ are non-negative. Indeed, if the integers $b_i$ are not necessarily non-negative, we can consider, for every $i = 1, \dots, n$,
$$ \ol{\cal{E}}_i' = \begin{cases}
\ol{\cal{E}}_i & \textup{if } b_i \ge 0 \\
\ol{\cal{E}}_i^\vee & \textup{otherwise}.
\end{cases}
$$
Let us set $\GLs(\cal{E}') \df \GLs(\cal{E}_1') \times \cdots \times \GLs(\cal{E}_n')$. If $\varpi : \ol{\cal{E}}_1^{\otimes b_1} \otimes \cdots \otimes \ol{\cal{E}}_n^{\otimes b_n} \to \ol{\cal{F}}$ is a homomorphism of hermitian vector bundles as in the statement of Theorem \ref{thm:LBHeightQuotient}, it induces a generically surjective and $\GLs(\cal{E}')$ homomorphism of hermitian vector bundles
$$ \varpi' : {\ol{\cal{E}}_1'}^{\otimes |b_1|} \otimes \cdots \otimes {\ol{\cal{E}}_n'}^{\otimes |b_n|} \too \ol{\cal{F}}.$$

The quotients of $\P(\cal{F}^\vee)$ by $\SLs(\cal{E})$ and $\SLs(\cal{E}') \df \SLs(\cal{E}_1') \times \cdots \times \SLs(\cal{E}_n')$ are canonically identified and the metrics induced on the polarisation $\cal{M}_D$ are clearly the same. In particular we have
$$ h_{\mini}\left( \quotss{(\P(\cal{F}^\vee), \O_{\ol{\cal{F}}^\vee}(1))}{\SLs(\cal{E})} \right) = h_{\mini}\left( \quotss{(\P(\cal{F}^\vee), \O_{\ol{\cal{F}}^\vee}(1))}{\SLs(\cal{E}')} \right).$$
\end{rem}

The remainder of this section is devoted to the proof of Theorem \ref{thm:LBHeightQuotient} when the integers $b_1, \dots, b_N$ are non-negative.
\end{npar}

\subsubsection*{Invariant theory for a product of linear groups}

\begin{np} Let $k$ be a field. Let $n \ge 1$ be a positive integer and $E = (E_1, \dots, E_n)$ a $n$-tuple of non-zero $k$-vector spaces of finite dimension. We define
\begin{align*}
\GLs(E) &\df \GLs(E_1) \times_k \cdots \times_k \GLs(E_n), \\
\SLs(E) &\df \SLs(E_1) \times_k \cdots \times_k \SLs(E_n).
\end{align*}

\begin{deff} Let $F$ be a non-zero $k$-vector space of finite dimension. A representation, \em{i.e.} a morphism of $k$-group schemes, $\rho : \GLs(E) \to \GLs(F)$ is said to be \em{homogeneous of weight $b = (b_1, \dots, b_n) \in \Z^n$} if, for every $k$-scheme $S$ and all $S$-points $t_1,\dots,t_n \in \Gm(S)$, we have
$$ \rho(t_1 \cdot \id, \dots, t_n \cdot \id) = t_1^{b_1} \cdots t_n^{b_n} \cdot \id_{F}.$$
\end{deff}

\begin{prop} \label{prop:BasicPropertiesHomogeneousRepresentations}
Let $\rho : \GLs(E) \to \GLs(F)$ be a homogeneous representation of weight $b = (b_1, \dots, b_n)$ and suppose that the subspace of $\SLs(E)$-invariant elements of $F$ is non-trivial. Then : 
\begin{enumerate}[(1)]
\item for every $i = 1, \dots, n$ the dimension $e_i$ of $E_i$ divides the integer $b_i$;
\item for every $k$-scheme $S$, any $S$-point $(g_1, \dots, g_n)$ of $\GLs(E)$ and any $\SLs(E)$-invariant element $w$ of $F$ we have : 
\begin{equation} \label{eq:HomRepresentationInvariantElement} \rho(g_1, \dots, g_n) \cdot w = \det(g_1)^{b_1 / e_1} \cdots \det(g_n)^{b_n / e_n} \cdot w. \end{equation}
\end{enumerate}
\end{prop}
\begin{proof} This follows from the fact that characters of the general linear group are powers of the determinant.
\end{proof}
\end{np}

\begin{np} For every non-negative integer $N$ let us denote by $\frak{S}_N$ the group of permutations on $N$ elements (if $N = 0$, then $\frak{S}_0 = \{ \id_{\emptyset}\}$). If $E$ is a $k$-vector space the group $\frak{S}_N$ acts on the $N$-th tensor product $E^{\otimes N}$ permuting factors. Explicitly, if $\sigma \in \frak{S}_N$ is a permutation and $x_1, \dots, x_N$ are elements of $E$ we have
$$ \sigma \ast (x_1 \otimes \cdots \otimes x_N) = x_{\sigma(1)} \otimes \cdots \otimes x_{\sigma(N)}.$$ 

\begin{deff} The preceding action defines a homomorphism of non-commutative $k$-algebras $\frak{S}_{|N|} \to \End_k(E^{\otimes N})$ that we denote by $\eta_{E, N}$.
\end{deff}
\end{np}

\begin{np} Let $b = (b_1, \dots, b_n)$ be a $n$-tuple of non-negative integers. The group $\frak{S}_{b_1} \times \cdots \times \frak{S}_{b_n}$, which we denote by $\frak{S}_{b}$, acts component-wise on the $k$-vector space $E^{\otimes b} \df E_1^{\otimes b_1} \otimes \cdots \otimes E_n^{\otimes b_n}$. The $k$-group scheme $\GLs(E)$ acts by conjugation on the $k$-vector space
$$ \End_k(E^{\otimes b}) = \End_k(E_1^{\otimes b_1}) \otimes_k \cdots \otimes_k \End_k(E_n^{\otimes b_n})$$
and the representation $\GLs(E) \to \GLs(\End(E^{\otimes b}))$ is homogeneous of weight $0 = (0, \dots, 0)$. Proposition \ref{prop:BasicPropertiesHomogeneousRepresentations} (2) entails that the invariant elements of $\End(E^{\otimes D b})$ with respect to the action of $\GLs(E)$ and to the action of $\SLs(E)$ are the same.

\begin{deff} \label{deff:DefinitionEta} The action of $\frak{S}_{b}$ on $E^{\otimes b}$ defines a homomorphism of non-commutative $k$-algebras $\bigotimes_{i = 1}^n k[\frak{S}_{b_i}] \to \End(E^{\otimes b})$ that we denote by $\eta_{E, b}$.
\end{deff}
The image of $\eta_{E, b}$ is contained in the subspace of invariants of $\End(E^{\otimes b})$. The First Main Theorem of Invariant Theory affirms that in characteristic $0$ the converse inclusion holds too (\em{cf.} \cite[Chapter III]{Weyl}, \cite[Theorem 3.1, Corollary]{chen_ss} and \cite[Appendix 1]{Atiyah}):

\begin{theorem}[First Main Theorem of Invariant Theory] Suppose that the characteristic of $k$ is zero. The subspace of $\SLs(E)$-invariant elements of the $k$-vector space $\End(E^{\otimes b})$  is the image of the homomorphism $\eta_{E, b}$.
\end{theorem}

\begin{deff} \label{deff:DefinitionPhi} With the notation introduced above let us suppose that $e_i$ divides $b_i$ for every $i = 1, \dots, n$. We consider the natural homomorphism of $k$-vector spaces 
$$ \Phi_{E_i, b_i} : \End\left(E_i \right)^{\otimes b_i} \otimes \det(E_i)^{\otimes b_i / e_i} \too E_i^{\otimes  b_i}$$ 
defined as the composition of the following homomorphisms:
$$ \xymatrix@R=15pt@C=50pt{
%& \displaystyle k[\frak{S}_{|b_i|}] \otimes \det(E_i)^{\otimes b_i / e_i} \ar^{\eta_{E_i, b_i} \otimes \id}[r] 
 \displaystyle \End\left(E_i \right)^{\otimes b_i} \otimes \det(E_i)^{\otimes b_i / e_i} \ar@{=}[d]
&\displaystyle E_i^{\otimes  b_i} 
 \\
 \displaystyle E_i^{\otimes  b_i} \otimes \left( E_i^{\otimes e_i} \right)^{\vee  \otimes  b_i / e_i}  \otimes \det(E_i)^{\otimes  b_i / e_i} \ar^{\id \otimes \det \otimes \id}[r]
& \displaystyle   E_i^{\otimes  b_i} \otimes \left( \det(E_i)^{\vee}  \otimes \det(E_i) \right)^{\otimes  b_i / e_i} \ar@{=}[u]
}$$
Furthermore we set $\Phi_{E, b} \df \Phi_{E_1, b_1} \otimes \cdots \otimes \Phi_{E_n, b_n}$.
\end{deff}

\begin{cor} \label{cor:FirstThmInvariantTheory+HomRepresentations} Suppose that the characteristic of $k$ is zero. Let $F$ be a non-zero $k$-vector space of finite dimension and $\rho : \GLs(E) \to \GLs(F)$ be a representation. Let $b = (b_1, \dots, b_n)$ be a $n$-tuple of non-negative integers and 
$$ \phi : \bigotimes_{i = 1}^n E_i^{\otimes b_i} \too F $$
be a surjective and $\GLs(E)$-equivariant homomorphism  of $k$-vector spaces. The representation $\rho$ is homogeneous of weight $b = (b_1, \dots, b_n)$ and if the subspace of $\SLs(E)$-invariant elements of $F$ is non-zero we have:
\begin{enumerate}[(1)]
\item For every $i = 1, \dots, n$ the dimension $e_i$ of $E_i$ divides the integer $b_i$.
\item The subspace of $\SLs(E)$-invariants of $F$  is the image of the homomorphism 
$$\phi \circ \Phi_{E, b} \circ (\eta_{E, b} \otimes \id) : \bigotimes_{i = 1}^n k[\frak{S}_{b_i}] \otimes \bigotimes_{i = 1}^n \det(E_i)^{\otimes b_i / e_i} \too F.$$

\end{enumerate}
\end{cor}

This is just a combination of the First Main Theorem of Invariant Theory with following:
\begin{rem}
In characteristic $0$ a linear algebraic group is reductive if and only if for every linear representation $E$ of $G$ there exists a \em{unique} $G$-equivariant projection $R_E : E \to E^G$ (the so-called Reynolds operator). The uniqueness entails the functoriality of the projection on the invariants: for every $G$-equivariant linear homomorphism $\psi : E \to F$ between linear representations of $G$ we have $R_F \circ \psi = \psi \circ R_E$. In particular, if $\psi$ is surjective the induced homomorphism $\phi : E^G \to F^G$ is surjective too. For details, refer to \cite[page 182]{MumfordSuominen} and \cite[Chapter1, \S 1]{git}.
\end{rem}

\end{np}

\subsubsection*{Non-hermitian norms and tensor product} In this paragraph we briefly discuss norms on tensor products which are not hermitian. We refer the interested reader to \cite{GrothendieckProduitTensoriels} for the case of two vector spaces and \cite[Normes tensorielles, page 33]{gaudron} for the present setting.

Let $N \ge 1$ be a positive integer and for every $i = 1, \dots, N$ let $V_i$ be a finite-dimensional complex vector space endowed with a norm $\| \cdot \|_{V_i}$. Let $\| \cdot \|_{V_i^\vee}$ be the operator norm on $V_i^\vee$.

\begin{deff} The $\epsilon$-norm (resp. $\pi$-norm) on the tensor product $V \df V_1 \otimes_\C \cdots \otimes_\C \otimes V_N$ is the norm defined for every element $v \in V$ as
\begin{align*} \| v \|_{V, \epsilon} &\df \sup_{\substack{\phi_i \in V_i^\vee - \{0 \} \\ i = 1, \dots,N }} \frac{|\phi_1 \otimes \cdots \otimes \phi_N (v)|}{\| \phi_1\|_{V_1^\vee} \cdots \| \phi_N\|_{V_N^\vee}}, \\
\Bigg( \textup{resp. }  \| v \|_{V, \pi} \ &\ = \inf \left\{ \sum_{\alpha = 1}^R \| v_{\alpha 1}\|_{V_1} \cdots \| v_{\alpha N}\|_{V_N} : v = \sum_{\alpha = 1}^R v_{\alpha 1} \otimes \cdots \otimes v_{\alpha N}   \right\} \Bigg).
\end{align*}
We denote by $V_1 \otimes_\epsilon \cdots \otimes_\epsilon V_N$ (resp. $V_1 \otimes_\pi \cdots \otimes_\pi V_N$) the vector space $V$ equipped with the norm $\| \cdot \|_{V, \epsilon}$ (resp. $\| \cdot \|_{V, \epsilon}$). If all the norms $\| \cdot \|_{V_i}$ are all hermitians we denote by $V_1 \otimes_2 \cdots \otimes_2 V_N$ the vector space $V$ with the natural hermitian norm on the tensor product.
\end{deff}

\begin{prop} \label{prop:PropertiesOfEpsilonAndPiNorms} With the notation introduced above, the following properties are satisfied:
\begin{enumerate}[(1)]
\item The $\epsilon$-norm $\| \cdot \|_\epsilon$ (resp. the $\pi$-norm $\| \cdot \|_\pi$) is the smallest (resp. the biggest) amongst the norms $\| \cdot \|$ on $V$ such that for every $i = 1, \dots, N$ and every $v_i \in V_i$ we have
\begin{align*} 
\| v_1 \otimes \cdots \otimes v_N\| &\ge \| v_1\|_{V_1^{\vee\vee}} \cdots \| v_N\|_{V_N^{\vee\vee}}, \\
\left(\textup{resp. } \| v_1 \otimes \cdots \otimes v_N\| \right. & \left. \le \| v_1\|_{V_1} \cdots \| v_N\|_{V_N} \right), 
\end{align*}
and for every $i = 1, \dots,n$ and every $\phi_i \in V_i^\vee$ we have
\begin{align*} 
\| \phi_1 \otimes \cdots \otimes \phi_N\|^\vee &\le \| \phi_1\|_{V_1^{\vee}} \cdots \| \phi_N\|_{V_N^{\vee}}, \\
\left(\textup{resp. } \| \phi_1 \otimes \cdots \otimes \phi_N\|^\vee \right. & \left. \ge \| \phi_1\|_{V_1^\vee} \cdots \| \phi_N\|_{V_N^\vee} \right), 
\end{align*}
where $\| \cdot \|^\vee$ denotes the operator norm induced by $\| \cdot \|$ on $V^\vee$. 

\item For every $v \in V$ we have $\| v \|_{V, \epsilon} \le \| v \|_{V,\pi}$.
\item \textup{(Duality)} The natural isomorphism $V^\vee \iso V_1^\vee \otimes_\C \cdots \otimes_\C V_N^\vee$ induces the following isometries:
\begin{align*} 
(V_1 \otimes_\epsilon \cdots \otimes_\epsilon V_N)^\vee &\xrightarrow{\hspace{2.2pt}\sim\hspace{2.2pt}} V_1^\vee \otimes_\pi \cdots \otimes_\pi V_N^\vee, \\
(V_1 \otimes_\pi \cdots \otimes_\pi V_N)^\vee &\xrightarrow{\hspace{2.2pt}\sim\hspace{2.2pt}} V_1^\vee \otimes_\epsilon \cdots \otimes_\epsilon V_N^\vee.
 \end{align*}
%Let $\| \cdot \|_{V, \epsilon}^\vee$ (resp. $\| \cdot \|_{V, \pi}^\vee$) be the operator norm on $V^\vee$ associated to the norm $\| \cdot \|_{V, \epsilon}$ (resp. $\| \cdot \|_{V, \pi}$) and let $\| \cdot \|_{V^\vee, \epsilon}$ (resp. $\| \cdot \|_{V^\vee, \pi}$) be the $\epsilon$-norm (resp. $\pi$-norm) on the tensor product $V^\vee = V^\vee_1 \otimes \cdots \otimes V^\vee_N$ associated to the norms $\| \cdot \|_{V^\vee_1}, \dots, \| \cdot \|_{V^\vee_1}$. Then we have:
%$$ \| \cdot \|_{V, \epsilon}^\vee = \| \cdot \|_{V^\vee, \pi}, \quad \| \cdot \|_{V, \pi}^\vee = \| \cdot \|_{V^\vee, \epsilon}$$
\item \textup{(Fonctoriality)} For every $i = 1, \dots, N$ let $W_i$ be a finite-dimensional complex vector space equipped with a norm $\| \cdot \|_{W_i}$ and let $\phi_i : V_i \to W_i$ be a linear map decreasing the norms. Then, the induced maps
\begin{align*} 
\phi_1 \otimes \cdots \otimes \phi_N :V_1 \otimes_\epsilon \cdots \otimes_\epsilon V_N &\too W_1 \otimes_\epsilon \cdots \otimes_\epsilon W_N, \\
 \phi_1 \otimes \cdots \otimes \phi_N : V_1 \otimes_\pi \cdots \otimes_\pi V_N &\too W_1 \otimes_\pi \cdots \otimes_\pi W_N
 \end{align*}
decrease the norms.
\item Let $L$ be a normed vector space of dimension $1$. Then we have
\begin{align*} 
\left( V_1 \otimes_\epsilon \cdots \otimes_\epsilon V_N \right) \otimes_\epsilon L &= 
V_1 \otimes_\epsilon \cdots \otimes_\epsilon V_N \otimes_\epsilon L, \\
 \left( V_1 \otimes_\pi \cdots \otimes_\pi V_N \right) \otimes_\pi L &= 
V_1 \otimes_\pi \cdots \otimes_\pi V_N \otimes_\pi L.
\end{align*}
\end{enumerate}
\end{prop}

\begin{proof}[Sketch of the proof] (1), (4) and (5) are elementary considerations on the definitions of the norms. (2) Indeed, by bi-duality, for all $i = 1, \dots, N$ and all $v_i \in V_i$ the very definition of the $\epsilon$-norms entails
$$ \| v_1 \otimes \cdots \otimes v_N\|_\epsilon = \prod_{i = 1}^N \| v_i \|_{V_i^{\vee \vee}} \le \prod_{i = 1}^N \| v_i \|_{V_i},$$
and one concludes thanks to (1). (3) follows from (1) and (2) by duality.
%\cite[{\S 1.1, Th\'eor\`eme 2}]{GrothendieckProduitTensoriels} (2) \cite[{\S 1.2, Th\'eor\`eme 3}]{GrothendieckProduitTensoriels} (3) and (4) are clear by definition of the norms.
\end{proof}

\begin{prop} \label{prop:EpsilonNormAsAnOperatorNorm} Let $V$ and $W$ be finite-dimensional normed vector spaces. The operator norm on $\Hom_\C(V,W)$ coincides with the $\epsilon$-norm on $V^\vee \otimes_\C W$ through the canonical isomorphism
$$ V^\vee \otimes_\C W \xrightarrow{\hspace{2.2pt}\sim\hspace{2.2pt}} \Hom_\C(V, W) $$
\end{prop}

\begin{proof} This is Theorem \cite[{\S 1.1, Th\'eor\`eme 1}]{GrothendieckProduitTensoriels} for $E = V$, $F = \C$ and $G = W$.
\end{proof}

\begin{rem} \label{rem:EndomorphismsNormedVectorLine} Let $L$ be a normed complex vector line. It follows from the preceding Proposition that through the natural isomorphism $L \otimes L^{\vee} \iso \C$ the $\epsilon$-norm on $L \otimes L^{\vee}$ induces the natural absolute value on $\C$.
\end{rem}

\begin{prop} \label{prop:PiNormAndDeterminant} Let $W$ be an hermitian vector space and let $r \ge 1$ be a positive integer. Let us endow the exterior powers $\bigwedge^r W$ with the hermitian norm defined in \ref{par:ConvetionHermitianNorms}. Then the canonical map $\det: W^{\otimes_\pi r} \to \bigwedge^r W$ decreases the norms.
\end{prop}

\begin{proof}  For every element $w \in W^{\otimes r}$ and every writing  $w = \sum_{\alpha = 1}^R w_{\alpha 1} \otimes  \cdots \otimes w_{\alpha r}$
the Hadamard inequality \eqref{eq:HadamardInequality} yields
\begin{align*} 
\langle \det w, \det w\rangle_{\det E_i}
&= \sum_{\alpha, \beta = 1}^R \langle w_{\alpha 1} \wedge  \cdots \wedge w_{\alpha r}, w_{\beta 1} \wedge  \cdots \wedge w_{\beta r} \rangle_{\det E_i}  \\
&\le  \sum_{\alpha, \beta = 1}^R \|w_{\alpha 1}\|_{E_i}  \cdots \| w_{\alpha r} \|_{E_i} \|w_{\beta 1}\|_{E_i}  \cdots \| w_{\beta r} \|_{E_i} \\
&= \left( \sum_{\alpha = 1}^R \|w_{\alpha 1}\|_{E_i} \cdots \| w_{\alpha r} \|_{E_i}  \right)^2,
\end{align*}
which concludes the proof.
\end{proof}

\subsubsection*{Application to the lower bound of the height on the quotient}

\begin{np} Let us go back to the proof of Theorem \ref{thm:LBHeightQuotient} in the case when the integers $b_i$ are non-negative. Let us denote by $\cal{Y}$ the quotient of semi-stable points of $\P(\cal{F})$ by $\SLs(\cal{E})$ and, for every sufficiently divisible $D$, by $\ol{\cal{M}}_D$ the hermitian invertible sheaf on $\cal{Y}$ induced by $\O_{\ol{\cal{F}}}(D)$. Let us fix $D$ such that $\cal{M}_D$ is very ample.
\end{np}

\begin{npar}{Application of the First Main Theorem of Invariant Theory} Since the characteristic of $K$ is zero and the homomorphism $\varpi$ decreases the norms, one reduces to the case $\ol{\cal{F}} = \ol{\cal{E}}^{\otimes b} \df \ol{\cal{E}}_1^{\otimes b_1} \otimes_{\o_K} \cdots \otimes_{\o_K} \ol{\cal{E}}_n^{\otimes b_n}$. For every $i = 1, \dots, n$ let us denote by $E_i$ the $K$-vector space $\cal{E}_i \otimes_{\o_K} K$ and by $E^{\otimes b}$ the $K$-vector space $E_1^{\otimes b_1} \otimes_K \cdots \otimes_K E_n^{\otimes b_n}$. 

Remark that subspace of $\SLs(E)$-invariant elements of $\Sym^D F$ is non-zero because $\cal{M}_D$ is very ample. Therefore for every $i = 1, \dots, n$ the integer $e_i \df \dim_K \cal{E}_i$ divides $D b_i$. Let us consider the maps $\eta \df \eta_{E, Db}$ and $\Phi \df \Phi_{E, D b}$ (see Definitions \ref{deff:DefinitionEta} and \ref{deff:DefinitionPhi}) and the natural surjection:
$$ \phi: E^{\otimes D b} \too \Sym^D (E^{\otimes b}),$$
where $E^{\otimes Db} \df E_1^{\otimes D b_1} \otimes_K \cdots \otimes_K E_n^{\otimes D b_n}$.
%$$\phi \circ \Phi(E,  D b)   : \bigotimes_{i = 1}^n k[\frak{S}_{D|b_i|}] \otimes \det(E_i)^{\otimes D b_i / e_i} \too \Sym^D F.$$

\begin{lem} \label{lem:LowerBoundingHeightWithSizeOfPermutations}For every $i = 1, \dots, n$ let $\delta_i \in \det(E_i)$ be non-zero. 
\begin{enumerate}[(1)]
\item A set of generators of the $\SLs(E)$-invariant elements of $\Sym^D E^{\otimes b} = \Gamma(\P(E^{\otimes b}), \O(D))$ is given by the image through $\phi \circ \Phi$ of the elements
$$f_\sigma\df \eta(\sigma) \otimes \left( \delta_1^{\otimes D b_1 / e_1} \otimes \cdots \otimes \delta_1^{\otimes D b_n / e_n} \right)$$
where $\sigma = (\sigma_1, \dots, \sigma_n)$ ranges in $\frak{S}_{D b} \df \frak{S}_{D b_1} \times \cdots \times \frak{S}_{D b_n}$. 
\item Through the identification $\Gamma(\cal{Y}, \cal{M}_D) \otimes_{\o_K} K \iso \Gamma(\P(E^{\otimes b}), \O(D))^{\SLs(E)}$ we have:
\begin{multline*}
h_{\min}( \quotss{(\P(\cal{E}^{\otimes b}), \O_{\ol{\cal{F}}}(1))}{\SLs(\cal{E})}) \\ \ge - \frac{1}{D} \sup_{\sigma \in \frak{S}_{D|b|}} \left\{ \sum_{v \in \Vv_K} \log \sup_{\cal{Y}(\C_v)} \left\| (\phi \circ \Phi)(f_\sigma) \right\|_{\cal{M}_D,v} \right\}.
\end{multline*}
\end{enumerate}
\end{lem}

\begin{proof} (1) This is Corollary \ref{cor:FirstThmInvariantTheory+HomRepresentations} (applied to the representation $\Sym^D E^{\otimes b} $ and the $\SLs(V)$-equivariant surjection $\phi$).

(2) Let $Q \in \cal{Y}(\ol{\Q})$ be a point defined on a finite extension $K'$ of $K$. Since $\cal{M}_D$ is very ample, according to (1) there exists $\sigma \in \frak{S}_{D|b|}$ such that the $\SLs(E)$-invariant polynomial  $\phi \circ \Phi(f_\sigma)$ --- seen as a global section of $\cal{M}_D$ --- does not vanish at $Q$. By definition of the height we have:
\begin{align*}
h_{\cal{M}_D}(Q) &= \sum_{v \in \Vv_K} - \log \| \phi \circ \Phi(f_\sigma)\|_{\cal{M_D}, v}(Q) \ge \sum_{v \in \Vv_K} - \log \sup_{y \in \cal{Y}(\C_v)} \| \phi \circ \Phi(f_\sigma)\|_{\cal{M_D}, v}(y),
\end{align*}
from which the conclusion of the Lemma directly follows.
\end{proof}

\end{npar}

\begin{npar}{Size of the invariants} Consider the $\o_K$-module $\cal{E}^{\otimes D b} = \cal{E}_1^{\otimes D b_1} \otimes \cdots \otimes \cal{E}_n^{\otimes D b_n} $ and denote by $\ol{\cal{E}}^{\otimes_\epsilon D b}$ (resp. $\ol{\cal{E}}^{\otimes_2 D b}$)
the $\o_K$-module $\cal{E}^{\otimes Db}$ endowed for every embedding $\gamma : K \to \C$ with the $\epsilon$-norm on the normed vector space
$$ (\ol{\cal{E}}_{1, \gamma}^{\otimes_\epsilon e_1})^{\otimes_\epsilon D b_1 / e_1} \otimes_\epsilon \cdots \otimes_\epsilon (\ol{\cal{E}}_{n, \gamma}^{\otimes_\epsilon e_n})^{\otimes_\epsilon D b_n / e_n}$$
(resp. with the natural hermitian norm on tensor product) that we denote by $\| \cdot \|_\epsilon$ (resp. $\| \cdot \|_2$). We consider the $\o_K$-module $\End_{\o_K}(\cal{E}^{\otimes Db})$ endowed for every embedding $\gamma : K \to \C$ with the operator norm $\| \cdot \|_{\epsilon, 2, \gamma}$ on
$$ \End_{\epsilon, 2}( \ol{\cal{E}}_\gamma^{\otimes D b}) \df \Hom (\ol{\cal{E}}_\gamma^{\otimes_\epsilon D b}, \ol{\cal{E}}^{\otimes_2 D b}_\gamma).$$
We denote the resulting normed $\o_K$-module by $\End_{\epsilon, 2}( \ol{\cal{E}}^{\otimes D b})$. For every $\gamma : K \to \C$ let us moreover endow the complex vector space $\Sym^D ( \cal{E}^{\otimes b}) \otimes_\gamma \C$ with sup-norm on polynomials (see paragraph \ref{par:ConvetionHermitianNorms}).

\begin{lem} \label{lem:PhiIsNormsDecreasing} With the notation introduced above, the map $\Phi_{E, Db} \circ \phi$ defines an homomorphism of $\o_K$-modules
$$ \phi \circ \Phi : \End_{\epsilon, 2}( \ol{\cal{E}}^{\otimes D b})
\otimes_\epsilon 
\sideset{}{_\epsilon}\bigotimes_{i = 1}^n (\det \ol{\cal{E}}_i)^{\otimes_\epsilon D b_i / e_i}
\too \Sym^D ( \cal{E}^{\otimes b}),
$$
which decreases the norms.
\end{lem}

\begin{proof} The fact that the homomorphism $\Phi_{E, Db} \circ \phi$ is defined at the level of $\o_K$-module is clear. Remark that $\phi \circ \Phi$ is defined as a composition of the following natural maps:
\begin{enumerate}[(1)]
\item $\ol{\cal{E}}^{\otimes_2 D b} \too \Sym^D (\ol{\cal{E}}^{\otimes_2 b})$;
\item $\displaystyle
\ol{\cal{E}}^{\otimes_2 D b} \otimes_\epsilon \left( \sideset{}{_\epsilon}\bigotimes_{i = 1}^n  (\det \ol{\cal{E}}_i)^{\otimes_\epsilon D b_i / e_i}  \right)^\vee
\otimes_\epsilon 
\sideset{}{_\epsilon}\bigotimes_{i = 1}^n  (\det \ol{\cal{E}}_i)^{\otimes_\epsilon D b_i / e_i}   \xrightarrow{\hspace{2.2pt}\sim\hspace{2.2pt}} \ol{\cal{E}}^{\otimes_2 D b}$;
\item $\det : \left( \ol{\cal{E}}_i^{\otimes_\epsilon e_i} \right)^\vee \too \det \ol{\cal{E}}_i^\vee$;
\item $\displaystyle \End_{\epsilon, 2}( \ol{\cal{E}}^{\otimes D b}) \too \ol{\cal{E}}^{\otimes_2 D b} \otimes_\epsilon \left(  \sideset{}{_\epsilon} \bigotimes_{i = 1}^n   (\det \ol{\cal{E}}_i)^{\otimes_\epsilon D b_i / e_i} \right)^\vee $.
\end{enumerate}
We claim that for every $\gamma : K \to \C$ each of these maps reduces the norms. Indeed, for (1) it is a reformulation of the fact that the hermitian norm on polynomials defined in paragraph \ref{par:ConvetionHermitianNorms} is bigger than the sup norm; for (2) it follows from Propositions \ref{prop:PropertiesOfEpsilonAndPiNorms} (5) and Remark \ref{rem:EndomorphismsNormedVectorLine}; for (3) it is Proposition \ref{prop:PiNormAndDeterminant} and the isometric isomorphism $(\ol{\cal{E}}_i^\vee)^{\otimes_\pi e_i} \iso (\ol{\cal{E}}_i^{\otimes_\epsilon e_i})^\vee$ given by Proposition \ref{prop:PropertiesOfEpsilonAndPiNorms} (3); for
(4) it follows from (3) and Proposition \ref{prop:EpsilonNormAsAnOperatorNorm}.
\end{proof}

\begin{lem} \label{lem:SizePermutations} Let $\sigma= (\sigma_1, \dots, \sigma_n) \in \frak{S}_{D b}$. For every $\gamma : K \to \C$ we have:
$$ \| \eta(\sigma)\|_{\epsilon, 2, \gamma} \le \sqrt{e_1^{D b_1 } \cdots e_N^{D b_N }}.$$
\end{lem}

\begin{proof} For every $i = 1, \dots, n$ let $v_{i1}, \dots, v_{i e_i}$ be an orthonormal basis of $\ol{\cal{E}}_i$. Consider the set $\cal{R}$ of indices $ R = (r_{i,j} : 1 \le i \le n, 1 \le j \le D |b_i|)$ with integral entries satisfying $1 \le r_{i,j} \le e_i$ for every $i, j$. For every $R \in \cal{R}$ let us set
$$ v_R \df \bigotimes_{i = 1}^n \bigotimes_{j = 1}^{D b_i / e_i} \bigotimes_{\alpha = 1}^{e_i} v_{i r_{i, j e_i + \alpha}} = \bigotimes_{i = 1}^n \bigotimes_{j = 1}^{D b_i / e_i} v_{i r_{i, j e_i + 1}} \otimes \cdots \otimes v_{i r_{i, j e_i + e_i}}.$$
The vectors $v_R$ for $R \in \cal{R}$ form an orthonormal basis of $\ol{\cal{E}}^{\otimes D b}$. For every 
 $T \in \ol{\cal{E}}^{\otimes D b }$ let us write $T = \sum_{R \in \cal{R}} T_R v_R$. 
With this notation we have:
$$ \| T\|_{2}^2 = \sum_{R \in \cal{R}} |T_R|^2, \qquad \| T\|_{\epsilon} \ge \max_{R \in \cal{R}} |T_R|.$$
For every $R \in \cal{R}$ let us write $\sigma(R) = (r_{i , \sigma_i(j)})_{i, j}$. By definition of $\eta(\sigma)$ for every $T$ we have $ \eta(\sigma)(T) = \sum_{R \in \cal{R}} T_R v_{\sigma(R)}$. Therefore we get $ \| \eta(\sigma)(T)\|_2 = \| T \|_2$ and 
\begin{align*} 
\sup_{T \neq 0} \frac{\| \eta(\sigma)(T)\|_2^2}{\| T\|_\epsilon^2} 
= \sup_{T \neq 0} \frac{\|T\|_2^2}{\| T\|_\epsilon^2} \le \sup_{T \neq 0} \frac{\displaystyle \sum_{R \in \cal{R}} |T_R|^2}{\displaystyle  \max_{R \in \cal{R}} \left\{ |T_R| ^2 \right\}} = \# \cal{R} =e_1^{D b_1 } \cdots e_n^{D b_n}
\end{align*}
(note that the last supremum is attained for $T= \sum_{R \in \cal{R}} v_R$).
\end{proof}
\end{npar}

\begin{npar}{End of the proof of Theorem \ref{thm:LBHeightQuotient}} For every $i = 1, \dots, n$ let $\delta_i \in \det(E_i)$ be non-zero. For every $\sigma = (\sigma_1, \dots, \sigma_n) \in \frak{S}_{Db}$ let us consider
$$f_\sigma\df \eta(\sigma) \otimes \left( \delta_1^{\otimes D b_1 / e_1} \otimes \cdots \otimes \delta_1^{\otimes D b_n / e_n} \right) \in \End(E^{\otimes Db}) \otimes \bigotimes_{i = 1}^n \det(E_i)^{\otimes D b_i / e_i}$$
Since the elements $\eta$ are integral and the map $\phi \circ \Phi$ is defined at the level of $\o_K$-modules we have, for every non-archimedean place $v$,
\begin{align*}
\sup_{y \in \cal{Y}(\C_v)} \| \phi \circ \Phi(f_\sigma)\|_{\cal{M}_D, v} (y) \le \prod_{i = 1}^n \| \delta_i \|^{D b_i / e_i}_{\det E_i, v}.
\end{align*}
On the other hand according to Lemmata \ref{lem:PhiIsNormsDecreasing} and \ref{lem:SizePermutations} for every embedding $\gamma : K \to \C$ we have:
\begin{align*}
\sup_{y \in \cal{Y}(\C)} \| \phi \circ \Phi(f_\sigma)\|_{\cal{M}_D, \gamma} (y) &\le \| \phi \circ \Phi(f_\sigma)\|_{\sup, \gamma}
\le \| \eta(\sigma) \|_{\epsilon, 2} \cdot \prod_{i = 1}^n \| \delta_i \|^{D b_i / e_i}_{\det E_i, \gamma} \\
&\le \sqrt{e_1^{D b_1} \cdots e_N^{D b_N }} \cdot \prod_{i = 1}^n \| \delta_i \|^{D b_i / e_i}_{\det E_i, \gamma}.
\end{align*}
According to Lemma \ref{lem:LowerBoundingHeightWithSizeOfPermutations} we get
\begin{align*} 
[K : \Q] h_{\min}( \quotss{(\P(\cal{E}^{\otimes b}), \O_{\ol{\cal{F}}}(1))}{\SLs(\cal{E})}) \hspace{-70pt}& \\ &\ge - \sum_{i = 1}^n  \left( \frac{ b_i}{e_i} \left( \sum_{v \in \Vv_K}  \log \| \delta_i\|_{\det E_i, v} \right) \right) -  \log \sqrt{e_1^{b_1} \cdots e_N^{b_N}} \\
&\ge \frac{ b_i}{e_i} \degar \ol{\cal{E}_i} - \frac{1}{2} \sum_{i = 1}^n b_i \log \rk \cal{E}_i,
\end{align*}
and one concludes recalling $\muar(\cal{E}_i) = \degar(\cal{E}_i) / e_i$. \qed
\end{npar}

\section{From the Fundamental Formula to the Main Theorem} \label{sec:FundamentalFormulaToMainTheorem}

\subsection{Interlude on the index}

Let $K$ be a field of characteristic $0$.

\begin{npar}{Index} Let $n \ge 1$ be a positive integer and $\P = (\P^1)^n$ be the product of $n$ copies of the projective line over $K$. For every $i = 1, \dots, n$ let $\pr_i : \P \to \P^1$ be the projection onto the $i$-th factor.

Let $z = (z_1, \dots, z_n)$ be a $K$-point of $\P$ and $b = (b_1, \dots, b_n)$ be a $n$-tuple of positive real numbers. For every $i = 1, \dots, n$ let $t_i$ be a local parameter around $z_i \in \P^1(K)$.

\begin{deff} Let $f \in \O_{\P, z}$ be a regular function on $\P$ defined on an open neighbourhood of $z$. The function $f$ develops into power series 
$$ f = \sum_{\ell = (\ell_1, \dots, \ell_n) \in \N^n} f_\ell t_1^{\ell_1} \cdots t_n^{\ell_n},$$
with $f_\ell \in K$. If $f$ is non-zero, then we define the \em{index of $f$ at $z$ with respect to the weight $b$} as the real number
$$ \ind_b(f, z) \df \min \left\{b_1 \ell_1 + \cdots + b_n \ell_n : f_\ell \neq 0 \right\};$$
if $f = 0$ we set $\ind_b(0, z) \df + \infty$.
\end{deff}

If $b = (b_1, \dots, b_n)$ is a $n$-tuple of positive real numbers we write $1 / b$ to denote the $n$-tuple $(1 / b_1, \dots, 1 / b_n)$. Let the index with the respect the weight $1 / b$ be denoted by $\ind_{1 / b}$. The notion of index can be naturally extended to meromorphic sections $s$ of an invertible sheaf $L$ on $\P$: it suffices to choose a trivialising section $s_0$ of $L$ around $z$ and set
$$ \ind_b(s, z) \df \ind_b(s/s_0, z).$$
\end{npar}

\begin{npar}{Higher dimensional Dyson's Lemma}
The main result concerning the index is the Higher Dimensional Dyson's Lemma: the version stated here is due to Nakamaye \cite{nakamaye_id}. The original version of Esnault-Viehweg \cite{esnault-viewheg} (which has a slightly bigger error term) would work as well.

Let $r = (r_1, \dots, r_n)$ be a $n$-tuple of positive integers. We consider the following invertible sheaf on the projective scheme $\P$ : 
$$ \O_\P(r) \df \pr_1^\ast \O_{\P^1}(r_1) \otimes \cdots \otimes \pr_n^\ast \O_{\P^1}(r_n).$$

\begin{theorem}[Higher dimensional Dyson's Lemma] \label{thm:DysonsLemma} Let $z^{(0)}, \dots, z^{(q)}$ be $K$-points of $\P$ and $t^{(0)}, \dots, t^{(q)}$ be non-negative real numbers. Suppose that
\begin{itemize}
\item for every $i = 1, \dots, n$ and any $\sigma \neq \tau$ we have $\pr_i(z^{(\sigma)}) \neq \pr_i(z^{(\tau)})$;
\item there exists a non-zero global section $f \in \Gamma(\P, \O_\P(r))$ such that for every $\sigma = 0, \dots, q$ we have
$$ \ind_{1/r} (f, z^{(\sigma)}) \ge t^{(\sigma)}.$$
\end{itemize}
Then the following inequality is satisfied:
$$ \sum_{\sigma = 0}^{q} \vol \triangleup_n(t^{(\sigma)}) \le 1 + \epsilon_{q, r},$$
where
$$ \epsilon_{q, r} \df \prod_{i = 1}^{n - 1} \left( 1 +  \max_{i + 1 \le j \le n} \left\{ \frac{r_j}{r_i} \right\} \max \{ q - 1, 0\}\right) - 1.$$
\end{theorem}
\end{npar}

\begin{npar}{Index at a single point} Let $z = (z_1, \dots, z_n)$ be a $K$-point of $\P$, let $r = (r_1, \dots, r_n)$ be a $n$-tuple of positive integers and let $t \ge 0$ be a non-negative real number. 

\begin{deff} Let $Z_{q,r}(z, t)$ be the subscheme of $\P$ defined by the ideal sheaf of regular sections $f$ such that $\ind_{1/r}(f, z) \ge t$. We consider the following linear subspace of $\Gamma(\P, \O_\P(r))$:
\begin{align*}
K_{r}(z, t) &\df \ker \left(\Gamma(\P, \O_\P(r)) \to \Gamma(Z_{r}(z, t), \O_\P(r)) \right) \\
&\ = \left\{ f \in \Gamma(\P, \O_\P(r)) : \ind_{1/r}(f, z) \ge t \right\}.
\end{align*}
\end{deff}

\begin{prop} \label{prop:DimensionKernelSinglePoint} Keeping the notation introduced above, for every $i = 1, \dots, n$ let $T_{i0}, T_{i1}$ be a basis of $K^{2\vee}$ such that $T_{i1}$ vanishes at $z_i$.
\begin{enumerate}[(1)]
\item The monomials $T_z(\ell) = \bigotimes_{i = 1}^n T_{i0}^{r_i - \ell_i} T_{i1}^{\ell_i}$ for $\ell \in \triangledown_r^\Z(t)$ form a basis of the $K$-vector space $K_{r}(z, t)$.
\item We have $ \dim_K K_{r}(z, t) = \# \triangledown_r^\Z(t)$. In particular,
$$ \lim_{\alpha \to \infty} \frac{\dim_K K_{\alpha r}(z, t)}{\alpha^n (r_1 \cdots r_n)} = \vol \triangledown_n(t).$$
\item We have $ \dim_K \Gamma(Z_{r}(z, t), \O_\P(r)) = \# \triangleup_r^\Z(t)$. In particular,
$$ \lim_{\alpha \to \infty} \frac{\dim_K \Gamma(Z_{\alpha r}(z, t), \O_\P(\alpha r))}{\alpha^n (r_1 \cdots r_n)} = \vol \triangleup_n(t).$$
\end{enumerate}
\end{prop}
\begin{proof} Left to the reader as an easy exercise. \end{proof}
\end{npar}

\begin{npar}{Index at multiple points} Let $r = (r_1, \dots, r_n)$ be a $n$-tuple of positive integers and let $t \ge 0$ be a non-negative real number and let $N \ge 1$ be a positive integer.

For every $\sigma = 1, \dots, q$ let $z^{(\sigma)} = (z_1^{(\sigma)}, \dots, z_n^{(\sigma)})$ be a $K$-point of $\P$. Let us suppose that for every $\sigma \neq \tau$ and for every $i = 1, \dots, n$ we have $z_i^{(\sigma)} \neq z_i^{(\tau)}$. 

\begin{deff} Let us consider the $q$-tuple $\bs{z} = (z^{(1)}, \dots, z^{(q)})$. We consider the closed subscheme of $\P$,
$$ Z_{q,r}(\bs{z}, t) \df \bigcup_{\sigma = 1}^q Z_r(z^{(\sigma)}, t).$$
We consider the following linear subspace of $\Gamma(\P, \O_\P(r))$:
\begin{align*} 
K_{q,r}(\bs{z}, t) &\df \ker(\Gamma(\P, \O_\P(r)) \to \Gamma(Z_{q,r}(\bs{z}, t), \O_\P(r))) \\
&\ = \left\{ f \in \Gamma(\P, \O_\P(r)) : \ind_{1/r}(f, z^{(\sigma)}) \ge t \textup{ for all } \sigma = 1, \dots, q  \right\}.
\end{align*}
\end{deff}
Recall that $u_{q,r}(t)$ is defined as the unique real number belonging to $[0, n]$ and such that
$$ \vol \triangleup_{n}(u_{q, r}(t)) = \min \left\{ \max\left\{1 + \epsilon_{q, r} - q \vol \triangleup_n(t), 0 \right\}, 1 \right\}.$$
\begin{prop} \label{Prop:DimensionKernelAtMultiplePoints}Keeping the notation introduced above, we have:
\begin{enumerate}[(1)]
\item $\displaystyle \Gamma(Z_{q,r}(\bs{z}, t), \O_\P(r) ) = \bigoplus_{\sigma = 1}^q \Gamma(Z_r(z^{(\sigma)}, t), \O_\P(r) )$. 
\item $ \displaystyle \dim_K K_{q,r}(\bs{z}, t) \ge \prod_{i = 1}^n (r_i + 1) - q \# \triangleup_r^\Z(t)$. In particular,
\begin{equation} \label{eq:LBDimensionKernelAlgebraicPoint} \liminf_{\alpha \to \infty} \frac{\dim_K K_{q, \alpha r}(\bs{z}, t)}{\alpha^n (r_1 \cdots r_n)} \ge 1 - q \vol \triangleup_n(t). \end{equation}
\item Suppose $u_{q,r}(t) < n$. Let $z^{(0)} \in \P^1(K)$ be a point such that for every $i = 1, \dots, n$ and every $\sigma =1, \dots, q$ we have $\pr_i(z^{(0)}) \neq \pr_i( z^{(\sigma)})$. For every $t^{(0)} > u_{q,r}(t)$ we have 
$$K_{q,r}(\bs{z}, t) \cap K_r(z^{(0)}, t^{(0)}) = 0. $$
\item $\displaystyle  \limsup_{\alpha \to \infty} \frac{\dim_K K_{q, \alpha r}(\bs{z}, t)}{\alpha^n (r_1 \cdots r_n)} \le \vol \triangleup_n(u_{q,r}(t)).$

\end{enumerate}
\end{prop}

Let us emphasize that (3) and (4) are consequences of the Higher Dimensional Dyson's Lemma.

\begin{proof} (1) This is because the closed subschemes $Z_r(z^{(\sigma)}, t)$'s are pairwise disjoint (see \cite[Lemma 2.8]{esnault-viewheg}).

(2) Using the definition of $K_{q,r}(\bs{z}, t)$ as $\ker(\Gamma(\P, \O_\P(r)) \to \Gamma(Z_{q,r}(\bs{z}, t), \O_\P(r)))$, the preceding point yields
\begin{align*}
\dim_K K_{q,r}(\bs{z}, t) &\ge \dim_K \Gamma(\P, \O_\P(r)) - \dim_K \Gamma(Z_{q,r}(\bs{z}, t), \O_\P(r)) \\
&= \dim_K \Gamma(\P, \O_\P(r)) - \sum_{\sigma = 1}^q \dim_K \Gamma(Z_r(z^{(\sigma)}, t), \O_\P(r)) \\
&= \prod_{i = 1}^n (r_i + 1) -  q \# \triangleup_r^\Z(t),
\end{align*}
where in the last equality we used Proposition \ref{prop:DimensionKernelSinglePoint} (3).

(3) By contradiction let us suppose that there exists a non-zero element $f$ in the intersection $K_{q,r}(\bs{z}, t) \cap K_r(z^{(0)}, t^{(0)})$. The Higher Dimensional Dyson's Lemma entails
$$ \sum_{\sigma =1}^q \vol \triangleup_n(t) + \vol \triangleup_n(t^{(0)}) \le 1+ \epsilon_{q, r},$$
and thus $ \vol \triangleup_n(t^{(0)}) \le \vol \triangleup_n(u_{q,r}(t))$. This yields $t^{(0)} \le u_{q,r}(t)$ which contradicts the hypothesis $t^{(0)} > u_{q,r}(t)$.

(4) Remark that if $\vol \triangleup_{n}(u_{q, r}( t)) = 1$, that is $u_{q,r}(t) = n$, the statement is trivial. Hence we assume $u_{q,r}(t) < n$.  Let $z^{(0)} \in \P^1(K)$ be a point such that for every $i = 1, \dots, n$ and every $\sigma =1, \dots, q$ we have $\pr_i(z^{(0)}) \neq \pr_i( z^{(\sigma)})$. According to (3), for every $t^{(0)} > u_{q,r}(t)$ we have 
$$K_{q,r}(\bs{z}, t) \cap K_r(z^{(0)}, t^{(0)}) = 0. $$
Therefore Grassman's formula of dimensions gives
\begin{align*}
 \dim_K K_{q,r}(\bs{z}, t) &\le \dim_K \Gamma(\P, \O_\P(r)) - \dim_K K_r(z^{(0)}, t^{(0)}) \\
 &= \dim_K \Gamma(\P, \O_\P(r)) - \# \triangledown_r^\Z(t^{(0)}),
 \end{align*}
 where we used Proposition \ref{prop:DimensionKernelSinglePoint} (2) in the last equality. The statement is then obtained by applying this inequality to any positive multiple of $r$ and then letting $t^{(0)}$ tend to $u_{q,r}(t)$.
\end{proof}
\end{npar}

\subsection{Definition of the ``moduli problem''} \label{subsec:DefinitionModuliProblem}

Let $K$ be a number field and let $\Vv_K$ be its set of places. 

\begin{npar}{Linear actions on grassmannians} Let $\cal{E}$ be a flat  $\o_K$-module of finite rank. For every non-negative integer $N$ we consider the grassmannian $\Grass_{N}(\cal{E})$ of subspaces of rank $N$ of $\cal{E}$, \em{i.e.} the $\o_K$-scheme representing the functor
$$ \xymatrix@R=0pt{
\ul{\textup{Gr}}_N(\cal{E}) : \hspace{-25pt} & \hspace{16pt}\big\{  \text{ $\o_K$-schemes }  \big\}  \ar[r] &  \text{~$\bigl\{$ sets~$\bigr\}$ } \hspace{136pt}\\
&  (f : X \to \Spec \o_K) \ar@{|->}[r] & \left\{ \txt{locally free sub-$\O_X$-modules $\cal{F}$\\ of $f^\ast \cal{E}$ of rank $N$ with flat cokernel} \right\}.
}$$
Suppose that an $\o_K$-group scheme $\cal{G}$ acts linearly on the $\o_K$-module $\cal{E}$. Then, for every integer $N \ge 0$, the $\o_K$-group scheme $\cal{G}$ acts naturally on the grassmannian $\Grass_N(\cal{E})$ of subspaces of rank $N$, on the projective space $\P( \bigwedge^N \cal{E})$ and in an equivariant way on the invertible sheaf $\O_{\bigwedge^N \cal{E}}(1)$. Moreover, the Pl\"ucker embedding $ \varpi : \Grass_N(\cal{E}) \to \P \big(\textstyle \bigwedge^N \displaystyle \cal{E}\big)$ is $\cal{G}$-equivariant.
\end{npar}

\begin{npar}{Back to the Main Effective Lower Bound} \label{par:KernelsInGrassmannians} Let $K'$ be a finite extension of $K$ of degree $q \ge 2$. Let $n \ge 1$ be a positive integer. Let $\P = (\P^1_{\o_K})^n$ be the product of $n$ copies of the projective line over $\o_K$. Let $r = (r_1, \dots, r_n)$ be a $n$-tuple of positive integers and let $\O_\P(r)$ be the following invertible sheaf on $\P$,
$$ \O_\P(r) \df \pr_1^\ast \O_{\P^1}(r_1) \otimes \cdots \otimes \pr_n^\ast \O_{\P^1}(r_n).$$

For all $i = 1, \dots, n$ let $x_i$ be a $K$-point of $\P^1_{\o_K}$ and let $a_i$ be a $K'$-point of $\P^1_{\o_K}$ such that $K(a_i) = K'$. Consider the following points of $\P$:
\begin{align*}
x &\df (x_1, \dots, x_n),  \\
a &\df (a_1, \dots, a_n).
\end{align*}
We furthermore set $\bs{a} \df \{ a^{(\sigma)} :  \sigma \in \Hom_{K\textup{-alg}}( K',\ol{\Q}) \}$. Let $t_x, t_{\bs{a}} \ge 0$ be non-negative real numbers and let us consider the following $K$-vector spaces
\begin{align*}
K_r(x, t_x) &\df \left\{ f \in \Gamma(\P_K, \O_\P(r))  : \ind_{1/r}(f, x) \ge t_x \right\}, \\
K_{q,r}(\bs{a}, t_{\bs{a}}) &\df \left\{ f \in \Gamma(\P_K, \O_\P(r))   : \ind_{1/r}(f, a) \ge t_{\bs{a}} \right\}
\end{align*}
where $\P_K$ denotes the generic fiber of $\P$.\footnote{Here the index of the section $f$, which is defined over $K$, at the point $a$, which is defined over $K'$, means the index of the extension of $f$ to $K'$. Alternatively, one may define the $\ol{\Q}$-vector space
$$ \ol{K}_{q,r}(\bs{a}, t_{\bs{a}}) \df \left\{ f \in \Gamma(\P, \O_\P(r)) \otimes_{\o_K} \ol{\Q}   : \ind_{1/r}(f, a^{(\sigma)}) \ge t_{\bs{a}} \textup{ for all } \sigma : K' \to \ol{\Q}  \right\}$$
and notice that it is invariant under Galois action, thus it comes from a $K$-vector space $K_{q, r}(\bs{a}, t_{\bs{a}})$. In any case we have
$$ K_{q, r}(\bs{a}, t_{\bs{a}}) \otimes_K \ol{\Q} = \bigcap_{\sigma : K' \to \ol{\Q}} \ol{K}_r(a^{(\sigma)}, t_{\bs{a}}), $$
where $ \ol{K}_{r}(a^{(\sigma)}, t_{\bs{a}}) \df \left\{ f \in \Gamma(\P, \O_\P(r)) \otimes_{\o_K} \ol{\Q}   : \ind_{1/r}(f, a^{(\sigma)}) \ge t_{\bs{a}}  \right\}$.} Since $f$ is $K$-rational and $a$ is not, imposing index at $a$ automatically imposes the same index condition at all conjugates of $a$: this is the reason why we introduced the bold letter $\bs{a}$.

Let us denote $k_{q, r}(t_{\bs{a}})$ and $k_r(t_x)$ respectively the dimension of the $K$-vector spaces $K_r(x, t_x)$ and $K_{q,r}(\bs{a}, t_{\bs{a}})$. In such a way, these sub-vector spaces of the global sections $\Gamma(\P_K, \O_\P(r))$ define the following $K$-points of grassmannians : 
\begin{align*}
[K_r(x, t_x)] &\in \Grass_{k_r(t_x)}(\Gamma(\P, \O_\P(r))), \\
[K_{q,r}(\bs{a}, t_{\bs{a}})] &\in \Grass_{k_{q, r}(t_{\bs{a}})} (\Gamma(\P, \O_\P(r))).
\end{align*}
The $\o_K$-reductive group $\SLs_{2, \o_K}^n$ acts naturally on the product $\P = \left( \P^1_{\o_K} \right)^n$ and we consider the natural action induced on the grassmannians mentioned above. If we write
\begin{align*}
 \cal{F}_r(t_x) &\df \bigwedge^{k_r(t_x)} \Gamma(\P, \O_\P(r)), \\
\cal{F}_{q,r}(t_{\bs{a}}) &\df \bigwedge^{k_{q, r}(t_{\bs{a}})} \Gamma(\P, \O_\P(r)),
\end{align*}
the Pl\"ucker embeddings, which are equivariant morphisms with respect to the action of $\SLs_{2, \o_K}^n$, are maps
\begin{align*}
 \Grass_{k_r(t_x)} (\Gamma(\P, \O_\P(r)) ) &\too \P( \cal{F}_r(t_x)), \\
 \Grass_{k_{q, r}(t_{\bs{a}})} (\Gamma(\P, \O_\P(r)) ) & \too \P( \cal{F}_{q,r}(t_{\bs{a}})).
\end{align*}
\end{npar}

\begin{npar}{The geometric invariant theory data} \label{par:git_data} We apply the Fundamental Formula to the following situation :
\begin{align*}
P_r &= ([K_r(x, t_x)], [K_{q,r}(\bs{a}, t_{\bs{a}})]), \\
\cal{X}_r &= \Grass_{k_r(t_x)}(\Gamma(\P, \O_\P(r))) \times_{\o_K} \Grass_{k_{q, r}(t_{\bs{a}})}(\Gamma(\P, \O_\P(r)) ), \\
\cal{G} &= \SLs_{2, \o_K}^n, \\
\cal{L}_r &= \textup{polarization given by the Pl\"ucker embeddings of the grassmannians},
\end{align*}
and the closed embedding:
$$ j_r : \cal{X}_r \too \P(\cal{F}_r(t_x)) \times_{\o_K} \P(\cal{F}_{q,r}(t_{\bs{a}})) \too \P(\cal{F}_r(t_x) \otimes_{\o_K} \cal{F}_{q,r}(t_{\bs{a}})).$$
(The first arrow is the Pl\"ucker embedding of the Grassmannians and the second one is the Segre embedding).
For every embedding $\gamma : K \to \C$ the complex vector spaces 
\begin{align*}
\cal{F}_r(t_x) \otimes_\gamma \C &= \bigwedge^{k_{r}(t_x)} \left( \bigotimes_{i = 1}^n \Sym^{r_i} \C^{2 \vee} \right), \\
\cal{F}_{q,r}(t_{\bs{a}}) \otimes_\gamma \C &= \bigwedge^{k_{q, r}(t_{\bs{a}})} \left( \bigotimes_{i = 1}^n \Sym^{r_i} \C^{2 \vee} \right),
\end{align*}
are respectively equipped with the hermitian norms $\| \cdot \|_{\cal{F}_r(t_x), \gamma}$ and $\| \cdot \|_{\cal{F}_{q,r}(t_{\bs{a}}), \gamma}$ obtained by tensor operations (see \ref{par:ConvetionHermitianNorms}). We endow the complex vector space $\cal{F}_r(t_x)_\gamma \otimes_{\C} \cal{F}_{q,r}(t_{\bs{a}})_\gamma$ with the tensor norm associated to these norms. The result hermitian norm is clearly invariant under the action of $\textbf{SU}_2^n$. We denote $\ol{\cal{L}}_r$ for the associated hermitian invertible sheaf on $\cal{X}_r$.
\end{npar}

\subsection{Proof of the Main Theorem}

\begin{np} In this section we show Theorem \ref{thm:ApplyingComparisonFormula} admitting a semi-stability result (Theorem \ref{thm:Semi-stabilityCondition}) that we will prove in section \ref{sec:ProofOfSemiStabilityTheorem} and some intermediate computations (namely Propositions \ref{prop:UBHeightRationalPoint}, \ref{prop:HeightAlgebraicPoint} and \ref{prop:UBInstabilityMeasure}) that we shall prove in sections \ref{sec:UBHeight} and \ref{sec:UBInstabilityMeasure}. 

In order to prove Theorem \ref{thm:ApplyingComparisonFormula},  by an approximation argument, we may suppose that $n$-tuple of positive real numbers $r = (r_1, \dots, r_n)$ in the statement is made of positive rational numbers. Since the Main Effective Lower Bound is homogeneous in $r$, we may therefore assume that $r$ is made of positive integers. 
\end{np}

\begin{npar}{Semi-stability conditions} To prove Theorem \ref{thm:ApplyingComparisonFormula} we apply the Fundamental Formula to the point $P_r$. Therefore we have to show that it is semi-stable. In Section \ref{sec:ProofOfSemiStabilityTheorem} we prove the following:

\begin{theorem} \label{thm:Semi-stabilityCondition} Let $n \ge 1$ be a positive integer and $r = (r_1, \dots, r_n)$ be a $n$-tuple of positive integers. Let $t_x, t_{\bs{a}} \ge 0$ be real numbers with $t_{\bs{a}} < t_{q, n}(0)$. If the inequality
\begin{equation*}
\mu_n(u_{q,r}(t_{\bs{a}})) > \mu_n(t_x) + \epsilon_{q, r},
\end{equation*}
is satisfied then there exists a positive integer $\alpha_0 = \alpha_0(q, n, r, t_{\bs{a}}, t_x)$ such that, for every integer $\alpha \ge \alpha_0$, the $K$-point $P_{\alpha r} \in \cal{X}_{\alpha r}(K)$
is semi-stable under the action of $\SLs_2^n$ with respect to the polarization given by the Pl\"ucker embeddings.
\end{theorem}
\end{npar}

\begin{npar}{Applying the Fundamental Formula} The numerical condition appearing in the previous statement is exactly the condition \eqref{eq:SemiStabilityConditionMainTheorem} in Theorem \ref{thm:ApplyingComparisonFormula}. Hence according to Theorem \ref{thm:Semi-stabilityCondition} there exists a positive integer $\alpha_0 = \alpha_0(q, n, r, t_{\bs{a}}, t_x)$ such that, for every integer $\alpha \ge \alpha_0$, the $K$-point 
$$ P_{\alpha r} = ([K_{\alpha r}(x, t_x)], [K_{q, \alpha r}(\bs{a}, t_{\bs{a}})])$$
is semi-stable. The Fundamental Formula (or, better, Corollary \ref{cor:ComparisonInequality}) applied for every $\alpha \ge \alpha_0$ to the point $P_{\alpha r}$ gives  the following inequality: 
$$ h_{\ol{\cal{L}}_{\alpha r}}(P_{\alpha r}) + \frac{1}{[K : \Q]} \sum_{v \in S} \iota_v(P_{\alpha r}) \ge h_{\min}\left( \quotss{(\cal{X}_{\alpha r}, \ol{\cal{L}}_{\alpha r})}{\cal{G}}\right),$$
where we used that the instability measures are non-positive. Dividing by $\alpha^{n+1} (r_1 \cdots r_n) $ and letting $\alpha$ go to infinity we get
\begin{multline} 
- \frac{1}{[K : \Q]} \sum_{v \in S} \limsup_{\alpha \to \infty} \frac{\iota_v(P_{\alpha r})}{\alpha^{n +1} (r_1 \cdots r_n)} \\ \le \limsup_{\alpha \to \infty} \frac{h_{\ol{\cal{L}}_{\alpha r}} (P_{\alpha r})}{\alpha^{n+1} (r_1 \cdots r_n)}  -  \limsup_{\alpha \to \infty} \frac{h_{\min}\left( \quotss{(\cal{X}_{\alpha r}, \ol{\cal{L}}_{\alpha r})}{\cal{G}}\right)}{\alpha^{n+1} (r_1 \cdots r_n)}. \label{eq:ApplyingComparisonFormulaEq0} \end{multline}

In the following paragraphs we estimate the terms appearing in the preceding inequality. 

The bound of the term involving the height of the point $P_r$ will make appear the height of the points $x_i$'s and $a_i$'s. It is the counterpart of the classical upper bound of the size of the auxiliary polynomial made by means of Siegel's Lemma. Here it will be a direct consequence of basic definitions in Arakelov geometry.

The term where the instability measure occurs is of local nature and will make naturally intervene the distance between the algebraic and the rational point. In the classical framework this corresponds to the Taylor expansion of the auxiliary polynomial around the algebraic point. 

The term involving the lowest height on the quotient will finally play the role of the constant terms.
\end{npar}

\begin{npar}{Upper bound of the height} The Pl\"ucker embeddings give a closed isometric embedding of the product of grassmannians $ \cal{X}_r$  into the product $\P(\cal{F}_{q,r}(t_{\bs{a}})) \times \P(\cal{F}_r(t_x))$. Hence we have : 
$$ h_{\ol{\cal{L}}_{r}} (P_{r})= h_{\ol{\cal{F}}_r(t_x)}([K_r(x, t_x)]) + h_{\ol{\cal{F}}_r(t_{\bs{a}})}([K_{q,r}(\bs{a}, t_{\bs{a}})]). $$
Now some elementary estimates of Arakelov degrees (see Propositions \ref{prop:UBHeightRationalPoint}-\ref{prop:HeightAlgebraicPoint}) give : 
\begin{align*}
h_{\ol{\cal{F}}_r(t_x)}([K_r(x, t_x)]) &\le \sum_{i = 1}^n \sum_{\ell \in \triangledown_r^\Z( t_x)} \ell_i h(x_i) ,\\
h_{\ol{\cal{F}}_r(t_{\bs{a}})}([K_{q,r}(\bs{a}, t_{\bs{a}})]) &\le \left( \prod_{i = 1}^n (r_i + 1) - k_{q, r}(t_{\bs{a}})\right) \left(  q \left( \sum_{i = 1}^n r_i h(a_i)\right) + |r| \log \sqrt{2 q}\right).
\end{align*}
Applying these estimates to every positive integer multiple of $r$ we get: 
\begin{multline} 
\limsup_{\alpha \to \infty} \frac{h_{\ol{\cal{L}}_{\alpha r}} (P_{\alpha r})}{\alpha^{n+1} (r_1 \cdots r_n)} \le \left( \int_{\triangledown_n(t_x)} \zeta_1 \ d \lambda \right) \sum_{i = 1}^n r_i h(x_i) \\
+ q \vol \triangleup_n(t_{\bs{a}}) \left( q \left( \sum_{i = 1}^n r_i h(a_i) \right) +|r|  \log \sqrt{2 q} \right). \label{eq:UBHeightRothIntro}
\end{multline}

\end{npar}

\begin{npar}{Upper bound of the instability measure} \label{par:UBInstabilityMeasureRoth} Let $v$ be a place of $K$. 
If the place $v$ is non-archimedean we find: 
\begin{align*}
- \iota_v(P_r) \ge \max_{\sigma : K' \to \C_v} \Bigg\{ & k_{q, r}(t_{\bs{a}}) t_{\bs{a}}  \min_{i = 1, \dots, n} \left\{ r_i \m_v(a_i^{(\sigma)}, x_i) \right\}  \\
&+ \sum_{i = 1}^n \left( \sum_{\ell \in \triangledown_r^\Z(t_x)} \ell_i  - \frac{k_r(t_x) + k_{q, r}(t_{\bs{a}})}{2}r_i\right) \m_v(a_i^{(\sigma)}, x_i) \Bigg\},
\end{align*}
whereas in the archimedean case the previous lower bound holds when the error term
$$ k_{q, r}(t_{\bs{a}}) \sum_{i = 1}^n \log \sqrt{r_i + 1} +  k_{q, r}(t_{\bs{a}}) |r| \log \sqrt{3} + k_r(t_x) |r| \log 2 $$
is subtracted from the right-hand side of the previous lower bound. These bounds are proved in Section \ref{sec:UBInstabilityMeasure} (see Proposition \ref{prop:UBInstabilityMeasure}). If $v$ is non-archimedean, then applying these estimates to any positive integer multiple of $f$, and using Propositions \ref{Prop:DimensionKernelAtMultiplePoints} (2), \ref{prop:DimensionKernelSinglePoint} (2) and \ref{Prop:DimensionKernelAtMultiplePoints} (4), we obtain:
\begin{align*} 
- \limsup_{\alpha \to \infty} \frac{\iota_v(P_{\alpha r})}{\alpha^{n+1} (r_1 \cdots r_n)} \ge  
\max_{\sigma : K' \to \C_v} \Bigg\{ (1 - q \vol \triangleup_n(t_{\bs{a}})) t_{\bs{a}}    \min_{i = 1, \dots, n}  \left\{ r_i \m_v(a_i^{(\sigma)}, x_i) \right\}  \\ 
+ \frac{\mu_n(t_x) - \vol \triangleup_n(u_{q,r}(t_{\bs{a}}))}{2} \left(  \sum_{i = 1}^n r_i \m_v(a_i^{(\sigma)}, x_i) \right) \Bigg\}
\end{align*}
If $v$ is archimedean the term $ |r|( \vol \triangleup_n(u_{q,r}(t_{\bs{a}})) \log \sqrt{3} + \vol \triangledown_n(t_x) \log 2)$ has to be subtracted from the right-hand side. By Definition \ref{def:DefinitionOfMoreCombinatorics} (2) we have $\mu_n((u_{q,r}(t_{\bs{a}})) \le \vol \triangleup_n(u_{q,r}(t_{\bs{a}}))$, therefore condition \eqref{eq:SemiStabilityConditionMainTheorem} entails $ \mu_n(t_x) - \vol \triangleup_n(u_{q,r}(t_{\bs{a}})) < 0$. We coarsely bound from above the term $\m_v(a_i^{(\sigma)}, x_i)$:
$$ \m_v(a_i^{(\sigma)}, x_i) \le \sum_{w \mid v} \m_w(a_i, x_i),$$
 the sum being taken over the places $w$ of $K'$ over $v$. Therefore taking the sum over the places of $S$ and noticing that by Proposition \ref{prop:ProjectiveLiouvilleInequality} we have,
 $$ \sum_{v \in S} \sum_{w \mid v} \m_w(a_i, x_i) \le \sum_{v \in S} \sum_{w \in \Vv_{K'}} \m_w(a_i, x_i) = [K' : \Q] (h(x) + h(a_i)),$$
 we conclude that the term
$$ - \frac{1}{[K : \Q]} \sum_{v \in S}\limsup_{\alpha \to \infty} \frac{\iota_v(P_{\alpha r})}{\alpha^{n+1} (r_1 \cdots r_n)} , $$
is bounded below by
\begin{multline} \label{eq:UBInstabilityMeasureNARothIntro} 
\frac{1}{[K : \Q]} (1 - q \vol \triangleup_n(t_{\bs{a}})) t_{\bs{a}}   \sum_{v \in S} \left( \max_{\sigma : K' \to \C_v} \left\{   \min_{i = 1, \dots, n}  \left\{ r_i \m_v(a_i^{(\sigma)}, x_i) \right\} \right\} \right) \\ 
+ \frac{\mu_n(t_x) - \vol \triangleup_n(u_{q,r}(t_{\bs{a}}))}{2} \left(  \sum_{i = 1}^n r_i \left(  q h(x_i) + q h(a_i) \right) \right) \\
-  |r|( \vol \triangleup_n(u_{q,r}(t_{\bs{a}})) \log \sqrt{3} + \vol \triangledown_n(t_x) \log 2).
\end{multline}
\end{npar}

\begin{npar}{Lower bound of the height on the quotient} For every $i = 1, \dots, n$  let us set $\ol{\cal{E}}_i = \o_K^2$ and $b_i = - r_i (k_{q, r}(t_{\bs{a}}) +  k_r(t_x))$. Let us apply the lower bound given by Theorem \ref{thm:LBHeightQuotient} to the representation  
$$ \cal{G} = \SLs(\cal{E}_1) \times_{\o_K} \cdots \times_{\o_K} \SLs(\cal{E}_n) \too \GLs(\cal{F}_{q,r}(t_{\bs{a}}) \otimes \cal{F}_r(t_x))$$
and to the natural surjection
$$\varpi : \bigotimes_{i=1}^n \cal{E}_i^{\otimes b_i} =  \bigotimes_{i = 1}^n \left(\o_K^{2 \vee} \right)^{\otimes r_i (k_{q, r}(t_{\bs{a}}) +  k_r(t_x))}  \too \cal{F}_{q,r}(t_{\bs{a}}) \otimes \cal{F}_r(t_x).$$
Since the hermitian vector bundle $\ol{\cal{E}}_i$ is trivial we have $\muar(\ol{\cal{E}}_i) = 0$ for every $i = 1, \dots, n$. Through the closed $\cal{G}$-equivariant  embedding $j_r : \cal{X}_r \to \P(\cal{F}_{q,r}(t_{\bs{a}}) \otimes \cal{F}_r(t_x))$ Theorem \ref{thm:LBHeightQuotient} yields
\begin{align} 
h_{\min}\left(\quotss{(\cal{X}_r, \ol{\cal{L}}_r)}{\cal{G}} \right) 
&\ge h_{\mini}\left( \quotss{(\P(\cal{F}_{q,r}(t_{\bs{a}}) \otimes \cal{F}_r(t_x)), \O(1))}{\cal{G}} \right) \nonumber 
\\
&\ge  - \big(k_{q, r}(t_{\bs{a}}) + k_r(t_x) \big) |r| \log \sqrt{2} - \frac{1}{2} \left( \log k_{q, r}(t_{\bs{a}})! + \log k_r(t_x)! \right) \nonumber,
\end{align}
where the term $- 1/2 \left( \log k_{q, r}(t_{\bs{a}})! + \log k_r(t_x)! \right)$ is due to the ratio between the hermitian norm on the alternating product and the quotient norm with the respect to surjection $\varpi$ (see \ref{par:ConvetionHermitianNorms}). Thanks to Stirling's approximation one sees easily that
$$ \lim_{\alpha \to \infty} \frac{\log k_{q, r}(t_{\bs{a}})!}{\alpha^{n+1} (r_1 \cdots r_n)} = 0$$
and similarly for $k_r(t_x)$. The previous estimates, applied to any positive multiple of $r$, give: 
\begin{multline} - \limsup_{\alpha \to \infty} \frac{h_{\min}\left( \quotss{(\cal{X}_{\alpha r}, \ol{\cal{L}}_{\alpha r})}{\cal{G}}\right)}{\alpha^{n+1} (r_1 \cdots r_n)} \\ 
\le \left( \vol \triangleup_n(u_{q,r}(t_{\bs{a}})) + \vol \triangledown_n(t_x)\right)|r| \log \sqrt{2}. 
\label{eq:LBHeightQuotientRothIntro}
\end{multline}
\end{npar}

To conclude the proof of the Main Theorem it suffices to bound the asymptotic terms in \eqref{eq:ApplyingComparisonFormulaEq0} taking in account the inequalities  \eqref{eq:UBHeightRothIntro}, \eqref{eq:UBInstabilityMeasureNARothIntro} and \eqref{eq:LBHeightQuotientRothIntro}. 

\section{Upper bound of the height} \label{sec:UBHeight}

We go back to the notation introduced in Section \ref{subsec:DefinitionModuliProblem}.

\subsection{Rational point}

\begin{prop} \label{prop:UBHeightRationalPoint}With the notation introduced above, we have
$$ h_{\ol{\cal{F}}_r(t_x)}([K_r(x, t_x)]) \le \sum_{i = 1}^n \sum_{\ell \in \triangledown_n^\Z(r, t_x)} \ell_i h(x_i).$$
\end{prop}

\begin{proof}
Let $T_0, T_1$ be the canonical basis of $K^{2 \vee}$. For every $i = 1, \dots, n$ let $(x_{i0}, x_{i1}) \in K^2$ be a generator of the  line $x_i \in \P^1(K)$. We may suppose that $x_{i0}$ is non-zero. For every $n$-tuple of non-negative integers $\ell \in \quadrato_{r}$ define
$$ T(\ell) \df \bigotimes_{i = 1}^n T_0^{r_i - \ell_i} T_{x_i}^{\ell_i}$$
where $T_{x_i} = x_{i0} T_1 - x_{i1} T_{i0}$. A basis of the $K$-vector space $K_r(x, t_x)$ is given by the elements $T(\ell)$ while $\ell$ ranges in $\triangledown_r^\Z(t_x)$.

Let $v$ be a place of $K$. The Hadamard inequality \eqref{eq:HadamardInequality} gives
$$ \log \left\| \bigwedge_{\ell \in \triangledown_r^\Z(t, x)} T(\ell) \right\|_{\ol{\cal{F}}_r(t_x), v} \le \sum_{\ell \in \triangledown_r^\Z(t_x)} \log \| T(\ell)\|_{\Gamma(\P, \O_\P(r)), v}.$$
For every $n$-tuple of non-negative integers $\ell \in \quadrato_{r}$ the sub-multiplicativity of the norm on symmetric powers gives
\begin{align*}
\log \left\| T(\ell) \right\|_{\Gamma(\P, \O_\P(r)), v} &= \sum_{i = 1}^n \log \| T_0^{r_i - \ell_i} T_{x_i}^{\ell_i} \|_{v} \\ &\le \sum_{i=1}^n (r_i - \ell_i) \log \| T_0\|_{v} + \sum_{i = 1}^n \ell_i \log \| T_{x_i}\|_{v} = \sum_{i = 1}^n \ell_i \log \| x_i \|_{v}.
\end{align*}
Thanks to the Hadamard inequality, we conclude the proof by taking the sum over all places.
\end{proof}

\subsection{Target points} \label{subsec:UBHeightAlgebraicPoint}

\begin{prop} \label{prop:HeightAlgebraicPoint} With the notation introduced above, we have
$$  h_{\ol{\cal{F}}_r(t_{\bs{a}})}([K_{q,r}(\bs{a}, t_{\bs{a}})]) \le \left(\prod_{i = 1}^n (r_i + 1) - k_{q, r}(t_{\bs{a}}) \right) \left( q \sum_{i = 1}^n r_i h(a_i) + |r| \log \sqrt{2 q} \right).$$
\end{prop}

The rest of this section is devoted to the proof of this upper bound.

\begin{np} We begin equipping the $\o_K$-module $\Gamma(\P, \O_\P(r))$ with the natural hermitian metric induced by the identification
$$ \Gamma(\P, \O_\P(r)) = \bigotimes_{i = 1}^n \Sym^{r_i} \left(\o_K^{2 \vee}\right).$$
We denote by $\ol{\Gamma}(\P, \O_\P(r))$ the resulting $\o_K$-hermitian vector bundle. We remark that the $\o_K$-hermitian vector bundle $\ol{\Gamma}(\P, \O_\P(r))$ is \em{not} trivial since the basis of $\Gamma(\P, \O_\P(r))$ given by the elements 
$$ T(\ell) = \bigotimes_{i = 1}^n T_0^{r_i - \ell_i} T_1^{\ell_i}$$
is orthogonal but not orthonormal. Anyway, for every place $v$, the sub-multiplicativity of the norm on symmetric powers gives
$$\log \| T(\ell)\|_{\ol{\Gamma}(\P, \O_\P(r)), v} \le \sum_{i = 1}^n (r_i - \ell_i) \log \| T_0\|_v + \sum_{i = 1}^n \ell_i \log \| T_1\|_v = 0.$$
In particular we have
\begin{equation} \label{eq:SlopeSymmetricPowers}
\muar \big( \ol{\Gamma}(\P, \O_\P(r)) \big)  \ge - \sum_{v \in \Vv_K} \sum_{\ell \in \square_r^\Z} \log \| T(\ell)\|_{\ol{\Gamma}(\P, \O_\P(r)), v} \ge 0.
\end{equation}
\end{np}

\begin{np} We endow the $K$-vector space $K_{q,r}(\bs{a}, t_{\bs{a}})$ with the structure of $\o_K$-hermitian vector bundle induced by the one of $\ol{\Gamma}(\P, \O_\P(r))$. Namely, we consider the $\o_K$-module
$$ \cal{K}_{q,r}(\bs{a}, t_{\bs{a}}) = \Gamma(\P, \O_\P(r)) \cap K_{q,r}(\bs{a}, t_{\bs{a}})$$
equipped with the restriction of the hermitian norms on $\ol{\Gamma}(\P, \O_\P(r))$. Let us then consider
$$ \cal{C} = \Gamma(\P, \O_\P(r)) / \cal{K}_{q,r}(\bs{a}, t_{\bs{a}})$$
and endow it with quotient norms with respect to the surjection $\Gamma(\P, \O_\P(r)) \to \cal{C}$. We denote by $\ol{\cal{C}}$ the $\o_K$-hermitian vector bundle obtained in this way. With these choices and according to \eqref{eq:SlopeSymmetricPowers} we have
\begin{align*}
[K : \Q] h_{\ol{\cal{F}}_r(t_{\bs{a}})}([K_{q,r}(\bs{a}, t_{\bs{a}})]) &= - \degar \ol{\cal{K}}_{q,r}(\bs{a}, t_{\bs{a}})  \\
&= \degar \ol{\cal{C}} - \degar \ol{\Gamma}(\P, \O_\P(r)) \le \degar \ol{\cal{C}}.
\end{align*}
\end{np}

\begin{np} Let us denote by $E$ the $K$-vector space $\Gamma(Z_{q,r}(\bs{a}, t_{\bs{a}}), \O_\P(r))$ and let $\Omega$ be a Galois closure of $K'$ over $K$.  Let us endow the $\Omega$-vector space $E \otimes_K \Omega$ with a structure of $\o_\Omega$-hermitian vector bundle as follows. According to Proposition \ref{Prop:DimensionKernelAtMultiplePoints} we have
\begin{align} \label{eq:HeightAlgebraicPoint1} E \otimes_K \Omega = \bigoplus_{\sigma = 1}^q E^{(\sigma)} \end{align}
where for every embedding $\sigma : K' \to \ol{\Q}$ we wrote $ E^{(\sigma)} \df \Gamma (Z_r(a^{(\sigma)}, t_{\bs{a}}), \O_{\P_\Omega}(r))$\footnote{Here the point $a^{(\sigma)}$ is seen as an $\Omega$-point of $\P^1_\Omega = \P^1_K \times_K \Omega$ and $Z_r(a^{(\sigma)}, t_{\bs{a}})$ denotes the subscheme of $\P^1_\Omega$ of index $t_{\bs{a}}$ on the point $a^{(\sigma)}$.}. 
For every $i = 1, \dots, n$ let $(a_{i0}^{(\sigma)}, a_{i1}^{(\sigma)}) \in \Omega^2$ a generator of the line $a_i^{(\sigma)} \in \P^1(K)$. Since $a^{(\sigma)}$ is not $K$-rational we may assume $a_{i0}^{(\sigma)} \neq 0$. Therefore up to rescaling $(a_{i0}^{(\sigma)}, a_{i1}^{(\sigma)})$ we suppose $a_{i0}^{(\sigma)} = 1$.

For any $\sigma : K' \to \ol{\Q}$ a basis of the $\Omega$-vector space $E^{(\sigma)}$ is given by the elements
$$ T_{a^{(\sigma)}}(\ell) = \bigotimes_{i = 1}^n T_0^{r_i - \ell_i} T_{a_i^{(\sigma)}}^{\ell_i}$$
where  $T_{a_i^{(\sigma)}} = T_1 - a_{i1}^{(\sigma)} T_0$ and $\ell = (\ell_1, \dots, \ell_n)$ ranges in the elements of $\triangleup_r^\Z(t_{\bs{a}})$. We consider the $\o_\Omega$-submodule $\cal{E}^{(\sigma)} \subset E^{(\sigma)}$ generated by the elements $T(\ell)$'s and we equip it with the hermitian norm having the elements $T(\ell)$'s as an orthonormal basis. We denote by $\ol{\cal{E}}^{(\sigma)}$ the associated $\o_\Omega$-hermitian vector bundle.

Finally, according with \eqref{eq:HeightAlgebraicPoint1}, we endow $K$-vector space $E \otimes_K \Omega$ with the structure of $\o_\Omega$-hermitian vector bundle given by the orthogonal direct sum of the $\o_\Omega$-hermitian vector bundles $\ol{\cal{E}}^{(\sigma)}$'s. We denote by $\ol{\cal{E}}_\Omega$ the so-obtained hermitian vector bundle.
\end{np}

\begin{np} The evaluation homomorphism $\eta : \Gamma(\P_K, \O_\P(r)) \to E = \Gamma(Z_{q,r}(\bs{a}, t_{\bs{a}}), \O_\P(r))$ factors through an injection $\epsilon : \cal{C} \otimes_{\o_K} K \to E$. Applying the slope inequality (Proposition \ref{prop:SlopeInequality}), one gets
\begin{equation} \label{eq:HeightAlgebraicPoint0} 
 h_{\ol{\cal{F}}_r(t_{\bs{a}})}([K_{q,r}(\bs{a}, t_{\bs{a}})]) 
\le \frac{ \degar \ol{\cal{C}}}{[K : \Q]} \le \frac{\rk \cal{C}}{[\Omega : \Q]} \left( \muar_{\max}(\ol{\cal{E}}_\Omega) + \sum_{v \in \Vv_\Omega} \log \| \epsilon \|_{\sup, v}
\right) \end{equation}
where for every place $v \in \Vv_\Omega$ we denoted by $\| \epsilon\|_{\sup, v}$ the $v$-adic operator norm of the injection $\epsilon$, 
$$ \| \epsilon \|_{\sup, v} \df \sup_{0 \neq f \in \cal{C} \otimes \Omega} \frac{\| \epsilon(f) \|_{\cal{E}, v}}{\| f \|_{\cal{C}, v}} .$$
%= \sup_{0 \neq f \in \Gamma(\P_K, \O_\P(r))} \frac{\| \epsilon( f) \|_{\cal{E}, v}}{\| f \|_{\Gamma(\P, \O_\P(r)), v}}
Clearly this coincides with the operator norm $\| \eta\|_{\sup, v}$ of the evaluation morphism $\eta$. Let us also remark that, by definition, the $\o_\Omega$-hermitian vector bundle $\ol{\cal{E}}$ is trivial hence $\muar_{\max}(\ol{\cal{E}}_\Omega) = 0$.
\end{np}

\begin{np} We are left with bounding the $v$-adic size of the evaluation homomorphism $\eta$. For every embedding $\sigma = 1, \dots, q$ let us consider the composition $\eta^{(\sigma)} : \Gamma(\P_\Omega, \O_\P(r)) \to E_\sigma$ of the homomorphism $\eta$ with the canonical projection $E \otimes_K \Omega \to E^{(\sigma)}$. Let us also denote by $\| \eta^{(\sigma)} \|_{\sup, v}$ the operator norm of $\eta^{(\sigma)}$. With this notation we have : 

\begin{itemize}
\item $v$ non-archimedean: $\displaystyle \| \eta \|_{\sup, v} = \max_{\sigma = 1, \dots, q} \| \eta^{(\sigma)} \|_{\sup, v}$;
\item $v$ archimedean: $\displaystyle  \| \eta \|_{\sup, v} \le \sqrt{q}  \max_{\sigma = 1, \dots, q} \left\{ \| \eta^{(\sigma)} \|_{\sup, v} \right\} $.
\end{itemize}
For every $\sigma = 1, \dots, q$ and any $i = 1, \dots, n$ let us consider the automorphism $\phi_{i}^{(\sigma)}$ of the $\Omega$-vector space $\Omega^{2 \vee}$ defined by
$$ \phi_{i}^{(\sigma)} : \begin{cases}
T_0 \mapsto T_0 \\
T_1 \mapsto T_{a_i^{(\sigma)}} = T_1 - a_{i1}^{(\sigma)} T_0.\\
\end{cases} $$
We consider the linear automorphism $\phi^{(\sigma)}_r = \Sym^{r_1} \phi_1^{(\sigma)} \otimes \cdots \otimes \Sym^{r_n} \phi_n^{(\sigma)}$ on the $\Omega$-vector space 
$$\Gamma(\P_K, \O_\P(r)) \otimes_K \Omega = \Sym^{r_1} (\Omega^{2 \vee}) \otimes_\Omega \cdots \otimes_\Omega \Sym^{r_n} (\Omega^{2 \vee}),$$ 
where for any $i = 1, \dots, n$ the linear automorphism $\phi^{(\sigma)}_i$ its acting on the $i$-th factor through its natural action on symmetric powers. 

With this notation the homomorphism $\eta^{(\sigma)} \circ \phi^{(\sigma)}_r : \Gamma(\P_\Omega, \O_\P(r)) \to E^{(\sigma)}$ coincides with the evaluation morphism at the closed subscheme $Z_r((1 : 0), t_{\bs{a}})$, \em{i.e.} it is described as follows
$$ T(\ell) = \bigotimes_{i = 1}^n T_0^{r_i - \ell_i} T_1^{\ell_i} \mapsto 
\begin{cases}
T_{a^{(\sigma)}}(\ell) & \text{if $\ell \in \triangleup_r^\Z(t_{\bs{a}})$} \\
0 & \text{otherwise}.
\end{cases}
$$
By definition the elements $T_{a^{(\sigma)}}(\ell)$'s form an orthonormal basis of the trivial $\o_\Omega$-hermitian vector bundle $\ol{\cal{E}}^{(\sigma)}$. Thus we have $ \| \eta^{(\sigma)} \circ \phi^{(\sigma)}_r \|_{\sup, v} \le 1$ and we deduce 
$$ \| \eta^{(\sigma)}\|_{\sup, v}  \le \| (\phi^{(\sigma)}_r)^{-1} \|_{\sup, v}. $$

Recalling that for an endomorphism $\psi$ of a $\o_K$-hermitian vector bundle $\ol{\cal{V}}$ the sup-norm of $\psi$ is smaller than its norm as an element of $\ol{\cal{V}}^\vee \otimes \ol{\cal{V}}$, we have the following inequalities:
\begin{align*}
\|(\phi^{(\sigma)}_r)^{-1}\|_{\sup, v}
&\le \log \left\|(\phi^{(\sigma)}_r)^{-1} \right\|_{\End(\ol{\Gamma}(\P, \O_\P(r))), v}\\
&\le \sum_{i = 1}^n r_i \log \| (\phi^{(\sigma)}_i)^{-1} \|_{ \End(\o_K^{2 \vee}), v}.
\end{align*}
Now one has to treat separately the archimedean and the non-archimedean cases. By definition of $\phi_i^{(\sigma)}$ we have $(\phi^{(\sigma)}_i)^{-1}(T_0) = (1, 0)$ and $(\phi^{(\sigma)}_i)^{-1}(T_1) = ( a_{i1}^{(\sigma)}, 1)$, thus
\begin{itemize}
\item if $v$ is non-archimedean:
\begin{align*} 
\log \| (\phi^{(\sigma)}_i)^{-1} \|_{ \End(\o_\Omega^{2 \vee}), v} &= \log \max \left\{ \| (\phi^{(\sigma)}_i)^{-1}(T_0) \|_{v}, \| (\phi^{(\sigma)}_i)^{-1}(T_1) \|_{v} \right\}\\
&= \log \| (a_{i0}^{(\sigma)}, a_{i1}^{(\sigma)}) \|_v.
\end{align*}
\item if is $v$ archimedean:
\begin{align*} 
\log \| (\phi^{(\sigma)}_i)^{-1} \|_{ \End(\o_\Omega^{2 \vee}), v} &= \log  \sqrt{\| (\phi^{(\sigma)}_i)^{-1}(T_0) \|_{v}^2 +\| (\phi^{(\sigma)}_i)^{-1}(T_1) \|^2_{v} }\\
&\le \log \| (a_{i0}^{(\sigma)}, a_{i1}^{(\sigma)}) \|_v + \log \sqrt{2}.
\end{align*}
\end{itemize}
Taking the sum over all the places of $K$ we finally get the following chains of inequalities
\begin{align*}
\sum_{v \in \Vv_\Omega} \log \| \eta\|_{\sup, v} &\le \left( \sum_{v \in \Vv_\Omega} \max_{\sigma : K' \to \ol{\Q}} \left\{ \sum_{i = 1}^n r_i \log \| (a_{i0}^{(\sigma)}, a_{i1}^{(\sigma)}) \|_v \right\} \right) + |r| [\Omega : \Q] \log \sqrt{2q}  \\
&\le \left( \sum_{v \in \Vv_\Omega} \sum_{\sigma : K' \to \ol{\Q}} \sum_{i = 1}^n r_i \log \| (a_{i0}^{(\sigma)}, a_{i1}^{(\sigma)}) \|_v \right) + |r| [\Omega : \Q] \log \sqrt{2q}  \\
&= [\Omega : \Q] \left( \left( \sum_{\sigma : K' \to \ol{\Q}} \sum_{i = 1}^n r_i h(a_i^{(\sigma)}) \right) + |r|  \log \sqrt{2q} \right).
\end{align*}
Hence dividing by $[\Omega : \Q]$ and writing $h(a_i^{(\sigma)}) = h(a_i)$ for every $\sigma : K' \to \ol{\Q}$, according to \eqref{eq:HeightAlgebraicPoint0} we get
\begin{align*} 
h_{\ol{\cal{F}}_r(t_{\bs{a}})}([K_{q,r}(\bs{a}, t_{\bs{a}})]) \le \rk \cal{C} \left( q \left( \sum_{i = 1}^n r_i h(a_i) \right) + |r| \log \sqrt{2q} \right)
\end{align*}
We conclude by remarking that we have $\rk \cal{C} = \rk \Gamma(\P, \O_\P(r)) - \rk \cal{K}_{q,r}(\bs{a}, t_{\bs{a}})$ and recalling that we have $\rk \Gamma(\P, \O_\P(r)) = \prod_{i = 1}^n (r_i + 1)$ and that, by the very definition of $k_{q, r}(t_{\bs{a}})$, we have $k_{q, r}(t_{\bs{a}}) = \rk \cal{K}_{q,r}(\bs{a}, t_{\bs{a}})$. \qed
\end{np}

\section{Upper bound of the instability measure} \label{sec:UBInstabilityMeasure}

\subsection{Notations and first reductions} \label{subsec:NotationsUpperBoundInstabilityMeasure}

Let $v$ be a place of $K$ and let us go back to the notations defined in Section \ref{subsec:DefinitionModuliProblem}.

\begin{prop} \label{prop:UBInstabilityMeasure}
With the notation introduced above if the place $v$ is non-archimedean we have: 
\begin{align*}
\iota_v(P_r) \le - \max_{\sigma : K ' \to \C_v} \Bigg\{ & k_{q, r}(t_{\bs{a}}) \left( t_{\bs{a}} \min_{i = 1, \dots, n} \left\{ r_i \m_v(a_i^{(\sigma)}, x_i) \right\} \right) \\
&+ \sum_{i = 1}^n \left( \left( \sum_{\ell \in \triangledown_r^\Z(t_x)} \ell_i  - \frac{k_r(t_x) + k_{q, r}(t_{\bs{a}})}{2}r_i\right) \m_v(a_i^{(\sigma)}, x_i) \right) \Bigg\},
\end{align*}
whereas, in the $v$ archimedean case, the previous inequality holds with
$$ k_{q, r}(t_{\bs{a}}) \sum_{i = 1}^n \log \sqrt{r_i + 1} +  k_{q, r}(t_{\bs{a}}) |r| \log \sqrt{3} + k_r(t_x) |r| \log 2 $$
added on the right-hand side.
\end{prop}

Throughout this section we fix for every $i = 1, \dots, n$ a generator $(x_{i0}, x_{i1}) \in K_v^2$ of the line $x_i \in \P^1(K_v)$ such that $ \| (x_{i0}, x_{i1})\|_v = 1$. For every $i = 1, \dots,n$ and every $\sigma : K' \to \C_v$ we fix a generator $(a_{i0}^{(\sigma)}, a_{i1}^{(\sigma)}) \in \C_v^2$ of the line $a_i^{(\sigma)} \in \P^1(\C_v)$ such that  $\| (a_{i0}^{(\sigma)}, a_{i1}^{(\sigma)})\|_v = 1$. We finally fix a square root $\theta_i^{(\sigma)} \in \C_v$ of 
$$ \frac{1}{a_{i0}^{(\sigma)} x_{i1} - a_{i1}^{(\sigma)} x_{i0}}. $$

\begin{npar}{Elements of $\SLs_2$ measuring distances}

\begin{deff} For every $i = 1, \dots, n$ and every $\sigma = 1, \dots, q$ let $g^{(\sigma)}_i$ be the linear automorphism of $\C_v^2$ which is given by the following matrix with respect to the canonical basis $e_0, e_1$ of $\C_v^2$:
$$
g_i^{(\sigma)} \df
\left( \begin{pmatrix} 
a_{i0}^{(\sigma)} & a_{i1}^{(\sigma)} \\
x_{i0} & x_{i1}
\end{pmatrix}^\top\right)^{-1} = \frac{1}{a_{i0}^{(\sigma)} x_{i1} - a_{i1}^{(\sigma)} x_{i0}}
{\begin{pmatrix} 
x_{i1} & - x_{i0} \\
- a_{i1}^{(\sigma)} &  a_{i0}^{(\sigma)}
\end{pmatrix}} \in \GLs_2(\C_v).$$ 
Let us consider the $n$-tuple $g^{(\sigma)} \df (g_1^{(\sigma)}, \dots, g_n^{(\sigma)}) \in \GLs_2(\C_v)^n$.
\end{deff}

\begin{prop} \label{prop:PropertiesElementOfTheGroup} With the notation introduced above, for every $i = 1, \dots, n$ and every $\sigma = 1, \dots, q$ the following properties are satisfied:
\begin{enumerate}[(1)]
\item We have
$$ \det g_i^{(\sigma)} = \frac{1}{a_{i0}^{(\sigma)} x_{i1} - a_{i1}^{(\sigma)} x_{i0}}.$$
In particular $|\det g_i^{(\sigma)}|_v =  \d_v(a_i^{(\sigma)}, x_i)^{-1}$.
\item For every non-zero vector $(y_0, y_1) \in \C_v^2$ such that $\| (y_0, y_1)\|_v = 1$ we have
$$
 \| g_i^{(\sigma)} \ast (y_0 T_1 - y_1 T_0)\|_v = 
\begin{cases}
\vspace{7pt}\max \left\{\d_v(a_i^{(\sigma)}, [y]), \d_v(x_i, [y]) \right\} &\textup{$v$ is non-archimedean} \\
\sqrt{\d_v(a^{(\sigma)}_i, [y])^2 + \d_v(x_i, [y])^2} & \textup{$v$ archimedean}
\end{cases} $$
where $\ast$ denotes the dual action, $T_0, T_1$ the canonical basis of $\C_v^{2\vee}$ and $[y] \in \P^1(\C_v)$ denotes the line generated by $(y_0, y_1)$.
\end{enumerate}
\end{prop}

\begin{proof} (1) Recalling that the points $(x_{i0}, x_{i1})$ and $(a_{i0}^{(\sigma)}, a_{i1}^{(\sigma)})$ are of norm $1$, this is clear.
(2) The automorphism $g^{(\sigma)}_i$ of $\C_v^2$ acts on the dual vector space $\C_v^{2 \vee}$ through the transposed inverse automorphism, whose matrix (with respect to the canonical basis $T_0$, $T_1$) is
$$
\begin{pmatrix} 
a_{i0}^{(\sigma)} & a_{i1}^{(\sigma)} \\
x_{i0} & x_{i1}
\end{pmatrix}.$$
The remainder is an elementary computation.\end{proof}
\end{npar}

\begin{npar}{Proof of Proposition \ref{prop:UBInstabilityMeasure}} In this paragraph we prove Proposition \ref{prop:UBInstabilityMeasure} admitting two independent computations (Proposition \ref{prop:InstabilityMeasureRationalPoint} and \ref{prop:InstabilityMeasureAlgebraicPoint}) that we prove in the following sections.

\begin{deff}
For every $h = (h_1, \dots, h_n) \in \GLs_2(\C_v)$ let us define
\begin{itemize}
\item $\displaystyle \iota_v(h, [K_r(x, t_x)]) \df \log \frac{\| h \ast w_x \|_{\cal{F}_r(t_x), v}}{\| w_x \|_{\cal{F}_r(t_x), v}}$,
\item $\displaystyle  \iota_v(h, [K_{q,r}(\bs{a}, t_{\bs{a}})]) \df \log \frac{\| h \ast w_a \|_{\cal{F}_{q,r}(t_{\bs{a}}), v}}{\| w_a \|_{\cal{F}_{q,r}(t_{\bs{a}}), v}} $,
\end{itemize}
where $w_x \in \cal{F}_r(x, t_x) \otimes K$ (resp. $w_{\bs{a}} \in \cal{F}_r(\bs{a}, t_{\bs{a}}) \otimes K$) is a non-zero representative of Pl\"ucker embedding of the point $[K_r(x, t_x)]$ (resp. $[K_{q,r}(\bs{a}, t_{\bs{a}})]$).
\end{deff}

\begin{deff} For every $i = 1, \dots, n$ and every $\sigma : K' \to \C_v$ let us consider the linear automorphism $\tilde{g}_i^{(\sigma)} \df g_i^{(\sigma)} / \theta_i^{(\sigma)}$, which is of determinant $1$. We set $ \tilde{g}^{(\sigma)} \df (\tilde{g}_1^{(\sigma)}, \dots,\tilde{g}_n^{(\sigma)})$.
\end{deff}

Employing this notation the instability measure $\iota_v(P_r)$ can be written as
\begin{align*} 
\iota_v(P_r) &= \inf_{h \in \SLs_2(\C_v)^n} \left\{ \iota_v(h, [K_r(x, t_x)]) + \iota_v(h, [K_{q,r}(\bs{a}, t_{\bs{a}})]) \right\} \\
&\le \min_{\sigma : K' \to \C_v} \left\{ \iota_v(\tilde{g}^{(\sigma)}, [K_r(x, t_x)]) + \iota_v(\tilde{g}^{(\sigma)}, [K_{q,r}(\bs{a}, t_{\bs{a}})]) \right\}.
\end{align*}
The representations $\cal{F}_r(t_x)$ and $\cal{F}_r(t_{\bs{a}})$ of $\GLs_{2, \o_K}^n$ are respectively homogeneous of weights $k_r(t_x)r$ and $k_r(t_{\bs{a}})r$. By Proposition \ref{prop:PropertiesElementOfTheGroup} (1) we have $\log |\theta_i^{(\sigma)}| = \m_v(a_i^{(\sigma)}, x_i)/2$ and this yields
\begin{align*}
\iota_v(\tilde{g}^\sigma, [K_r(x, t_x)]) &= \iota_v(g^{(\sigma)}, [K_r(x, t_x)]) + \frac{k_r(t_x)}{2} \left( \sum_{i = 1}^n r_i \m_v(a_i^{(\sigma)}, x_i)\right), \\
\iota_v(\tilde{g}^\sigma, [K_{q,r}(\bs{a}, t_{\bs{a}})])
&= \iota_v(g^{(\sigma)}, [K_{q,r}(\bs{a}, t_{\bs{a}})]) + \frac{k_{q, r}(t_{\bs{a}})}{2} \left( \sum_{i = 1}^n r_i \m_v(a_i^{(\sigma)}, x_i)\right).
\end{align*}
One concludes the proof of Proposition \ref{prop:UBInstabilityMeasure} applying the following:

\begin{prop} \label{prop:InstabilityMeasureRationalPoint} With the notation introduced above, if $v$ is non-archimedean we have
$$ \iota_v(g^{(\sigma)}, [K_r(x, t_x)] ) \le - \sum_{i = 1}^n \sum_{\ell \in \triangledown_r^\Z(t_x)} \ell_i  \m_v(a_i^{(\sigma)}, x_i) $$
whereas, if $v$ is archimedean, the preceding inequality holds with $k_r(t_x) |r| \log 2$ be added on the right-hand side.
\end{prop}

\begin{prop} \label{prop:InstabilityMeasureAlgebraicPoint} With the notation introduced above, if $v$ is non-archimedean we have
\begin{align*} \iota_v(g^{(\sigma)}, [K_{q,r}(\bs{a}, t_{\bs{a}})] ) \le - k_{q, r}(t_{\bs{a}}) t_{\bs{a}} \min_{i = 1, \dots, n} \left\{ r_i \m_v(a_i^{(\sigma)}, x_i) \right\} ,
\end{align*}
whereas, if $v$ is archimedean, the preceding inequality holds with
$$ k_{q, r}(t_{\bs{a}}) \left( \frac{1}{2}  \sum_{i = 1}^n \log (r_i + 1) +  |r| \log \sqrt{3} \right) $$
added on the right-hand side.
\end{prop}

%\end{proof}
\end{npar}

\subsection{Taylor expansion at the single point : proof of Proposition \ref{prop:InstabilityMeasureRationalPoint}}

Let us keep the notations introduced in Section \ref{subsec:NotationsUpperBoundInstabilityMeasure}. In this section we prove Proposition \ref{prop:InstabilityMeasureRationalPoint}.  

\begin{np} Let $i \in \{ 1, \dots, n\}$ and let us consider the linear form
$$ T_{i1} \df - x_{i1} T_0 + x_{i0} T_{1} \in K_v^{2\vee}.$$
Since the point $(x_{i0}, x_{i1}) \in K_v^2$ is of norm $1$ the linear form $T_{i1}$ is of norm $1$. If $v$ is non-archimedean let $T_{i0} \in \o_v^{2 \vee}$ be a linear form such that $T_{i0}$, $T_{i1}$ is a basis of the $\o_v$-module $\o_v^{2 \vee}$. If $v$ is archimedean let $T_{i0} \in K_v^{2 \vee}$ be a linear form such that $T_{i0}$, $T_{i1}$ is an orthonormal basis of $K_v^{2 \vee}$.

Since the linear form $T_{i1}$ vanishes at $x_i$ for every $i = 1, \dots, n$, Proposition \ref{prop:DimensionKernelSinglePoint} (1) implies that a basis of the $K$-vector space $K_r(x, t_x)$ is given by the monomials
$$ T(\ell) = \bigotimes_{i = 1}^n T_{i0}^{r_i - \ell_i} T_{i1}^{\ell_i}$$
where $\ell = (\ell_1, \dots, \ell_n)$ ranges in $\triangledown_r^\Z(t_x)$. Therefore the following vector of $\cal{F}_r(x, t_x) \otimes_{\o_K} K_v$, 
$$ w \df \bigwedge_{\ell \in \triangledown_r^\Z(t_x)} T(\ell)$$
is a non-zero representative of the Pl\"ucker embedding of $[K_r(x, t_x)]$. If $v$ is non-archimedean the elements $T(\ell)$'s form a basis of the $\o_v$-module $\left( K_r(x, t_x) \otimes K_v \right) \cap \left( \Gamma(\P, \O_\P(r)) \otimes \o_v\right)$.  Hence
$$ \log \left\| w \right\|_{\cal{F}_r(t_x), v} = 0.$$
If $v$ is archimedean the elements $T(\ell)$'s are orthogonal but they are not of norm $1$ and we have
\begin{align*}
 \log \| w \|_{\cal{F}_r(t_x), v} 
 = \sum_{\ell \in \triangledown_r^\Z(t_x)}  \log \| T(\ell) \|_{\Gamma(\P, \O_\P(r)), v} 
&= - \frac{1}{2}  \sum_{\ell \in \triangledown_r^\Z(t_x)}  \sum_{i = 1}^n \log \binom{r_i}{\ell_i}.
\end{align*}
Bounding the binomial $\binom{r_i}{\ell_i}$ by $2^{r_i}$ we obtain
\begin{align*}  \log \| w \|_{\cal{F}_r(t_x), v}&= - \frac{1}{2} \sum_{\ell \in \triangledown_r^\Z(t_x)}  \sum_{i = 1}^n \log \binom{r_i}{\ell_i}  \ge - \frac{1}{2} \sum_{\ell \in \triangledown_r^\Z(t_x)}  \sum_{i = 1}^n r_i \log 2 \\ &= - k_r(t_x) |r| \log \sqrt{2}.
\end{align*}
\end{np}

\begin{np} \label{par:UBSizeTransformedMonomials}  For any $\ell \in \triangledown_r^\Z(t)$ the sub-multiplicativity of the norm on symmetric powers yields
$$ \log \| g^{(\sigma)} \cdot T(\ell)\|_{\Gamma(\P, \O_\P(r)), v} \le \sum_{i = 1}^n \left( (r_i - \ell_i) \log \|g^{(\sigma)}_i \cdot T_{i0}\|_{v} +  \ell_i \log \| g^{(\sigma)}_i \cdot T_{i1}\|_v \right). $$
Therefore applying the Hadamard inequality we get
\begin{align*}
\log \| g^{(\sigma)} \cdot w \|_{\cal{F}_r(t_x), v}
&\le \sum_{\ell \in\triangledown_r^\Z(t_x)} \log \| g^{(\sigma)} \cdot T(\ell) \|_{\Gamma(\P, \O_\P(r)), v} \\
&\le \sum_{\ell \in\triangledown_r^\Z(t_x)} \sum_{i = 1}^n \left( (r_i - \ell_i) \log \| g^{(\sigma)} \cdot T_{i0} \|_v + \ell_i \log \| g^{(\sigma)} \cdot T_{i1} \|_v \right).
\end{align*}
For every $i = 1, \dots, n$ Proposition \ref{prop:PropertiesElementOfTheGroup} (2) entails
\begin{itemize}
\item $\displaystyle \| g_i^{(\sigma)} \cdot T_{i0} \|_v  \le
\begin{cases}
\vspace{3pt} 1 & \textup{$v$ non-archimedean}\\
\sqrt{2} &\textup{$v$ archimedean}
\end{cases}$

\item $\| g_i^{(\sigma)} \cdot T_{i1} \|_v = \d_v(a_i^{(\sigma)}, x_i)$.
\end{itemize}
Summarising if $v$ is non-archimedean we obtain:
$$ \iota_v(g^{(\sigma)}, [K_r(x, t_x)]) = \log \| g^{(\sigma)} \cdot w \|_{\cal{F}_r(t_x), v} \le - \sum_{i = 1}^n \left( \sum_{\ell \in\triangledown_r^\Z(t_x)} \ell_i \right) \m_v(a_i^{(\sigma)}, x_i).$$
On the other hand if $v$ is archimedean we have:
\begin{align*}
\iota_v(g^{(\sigma)}, [K_r(x, t_x)]) &\le \log \| g^{(\sigma)} \cdot w \|_{\cal{F}_r(t_x), v} + k_r(t_x) |r| \log \sqrt{2} \\
&\le - \sum_{i = 1}^n \left( \sum_{\ell \in\triangledown_r^\Z(t_x)} \ell_i \right) \m_v(a_i^{(\sigma)}, x_i) + k_r(t_x) |r| \log 2,
\end{align*}
which concludes the proof. \qed
\end{np}

\subsection{Taylor expansion at the algebraic points : proof of Proposition \ref{prop:InstabilityMeasureAlgebraicPoint}}

Let us keep the notations introduced in Section \ref{subsec:NotationsUpperBoundInstabilityMeasure}. In this section we prove Proposition \ref{prop:InstabilityMeasureAlgebraicPoint}.

\begin{np} If $v$ is non-archimedean let $\o_v$ be the ring of $K_v$ and let $f_1, \dots, f_{k_{q, r}(t_{\bs{a}})}$ be a basis of the $\o_v$-module:
$$\left( K_{q,r}(\bs{a}, t_{\bs{a}}) \otimes_K K_v \right) \cap \left( \Gamma(\P, \O_\P(r) \otimes_{\o_K} \o_v \right).$$
If $v$ is  archimedean let $f_1, \dots, f_{k_{q, r}(t_{\bs{a}})}$ be an orthonormal basis of $K_{q,r}(\bs{a}, t_{\bs{a}}) \otimes K_v$. With these notations the vector of $\cal{F}_{q, r}(\bs{a}, t_{\bs{a}}) \otimes_{\o_K} K_v$,
$$ w \df \bigwedge_{\alpha = 1}^{k_{q, r}(t_{\bs{a}})} f_\alpha$$
is a non-zero representative of the Pl\"ucker embedding of $[K_{q,r}(\bs{a}, t_{\bs{a}})]$. In order to simplify notation let us denote by $\| \cdot \|_v$ the induced norm on the $\C_v$-vector space $\Gamma(\P, \O_\P(r)) \otimes_{\o_K} \C_v$. With this notation Hadamard's inequality \eqref{eq:HadamardInequality} entails
$$
\iota_v(g^{(\sigma)}, [K_{q,r}(\bs{a}, t_{\bs{a}})]) 
= \log \left\|g^{(\sigma)} \cdot w \right\|_{\cal{F}_{q,r}(t_{\bs{a}}), v} \le \sum_{\alpha = 1}^{k_{q, r}(t_{\bs{a}})} \log \| g^{(\sigma)} \cdot f_\alpha\|_v.
$$
and we are thus left with proving the following lemma:
\begin{lem} Let $f$ be a non-zero element of $K_{q,r}(\bs{a}, t_{\bs{a}})$. With the notation introduced above, if $v$ is non-archimedean we have
$$ \log \frac{\| g^{(\sigma)} \cdot f \|_v}{\| f\|_v} \le t_{\bs{a}} \max_{i = 1, \dots,n} \left\{r_i \log \d_v(a_i^{(\sigma)}, x_i) \right\}$$
whereas, if $v$ is archimedean, the preceding inequality holds with
$$ \frac{1}{2}  \sum_{i = 1}^n \log (r_i + 1) +  |r| \log \sqrt{3} $$
added on the right-hand side.
\end{lem}
\end{np}

\begin{np} For every $i = 1, \dots, n$ let us consider the linear form:
$$ T_{i1} \df - a_{i1}^{(\sigma)} T_0 + a_{i0}^{(\sigma)} T_1 \in \C_v^{2 \vee}.$$
Since the point $(a_{i0}^{(\sigma)}, a_{i1}^{(\sigma)}) \in \C_v^2$ is of norm $1$ the linear form $T_{i1}$ is of norm $1$. If $v$ is non-archimedean let $\ol{\o}_v$ be the ring of integers of $\C_v$ and let $T_{i0} \in \ol{\o}_v^{2 \vee}$ be a linear form such that $T_{i0}$, $T_{i1}$ is a basis of the $\ol{\o}_v$-module $\ol{\o}_v^{2 \vee}$. If $v$ is archimedean let $T_{i0} \in \C_v^{2 \vee}$ be a linear form such that $T_{i0}$, $T_{i1}$ is an orthonormal basis of $\C_v^{2 \vee}$. For every $n$-tuple of integers $\ell = (\ell_1, \dots, \ell_n) \in \square_r$ let us define
$$ T(\ell) = \bigotimes_{i = 1}^n T_{i0}^{r_i - \ell_i} T_{i1}^{\ell_i}.$$
The monomials $T(\ell)$'s form a basis of the $\C_v$-vector space $\Gamma(\P, \O_\P(r)) \otimes_K \C_v$. If $v$ is non-archimedean the elements $T(\ell)$'s form a basis of the $\ol{\o}_v$-module $\Gamma(\P, \O_\P(r)) \otimes_{\o_K} \ol{\o}_v$. If $v$ is archimedean the monomials $T(\ell)$'s are orthogonal and for every $\ell \in \square_r$ we have
$$ \| T(\ell) \|_v = \binom{r}{\ell}^{-1/2} \df \prod_{i = 1}^n \binom{r_i}{\ell_i}^{-1 / 2}.$$
\end{np}

\begin{np} A computation similar to the one in paragraph \ref{par:UBSizeTransformedMonomials} yields if $v$ is non-archimedean
$$ \log \| g^{(\sigma)} \cdot T(\ell)\|_v \le \sum_{i = 1}^n \ell_i \log \d_v(a_i^{(\sigma)}, x_i) = - \sum_{i = 1}^n \ell_i \m_v(a_i^{(\sigma)}, x_i),$$
whereas, if $v$ is archimedean, the preceding inequality holds when $k_{q, r}(t_{\bs{a}}) |r| \log \sqrt{2}$ is added to the right-hand side.
\end{np}

\begin{np} Let us write $f = \sum_{\ell} f_{\ell} T(\ell)$ with $f_\ell \in \C_v$. If $v$ is non-archimedean we have 
$$ \| f\|_v = \max \left\{ |f_\ell|_v :  \ell \in \square_r^\Z  \right\}.$$
If $v$ is archimedean we have
$$ \| f\|_v^2 =  \sum_{\ell \in \square_r^\Z} |f_\ell|_v^2 \binom{r}{\ell}^{-1}.$$
Since the real numbers $\m_v(a_i^{(\sigma)}, x_i)$ are non-negative for every $\ell \in \triangledown_r^\Z(t_{\bs{a}})$ we have:
\begin{align*} 
\sum_{i = 1}^n \ell_i \m_v(a_i^{(\sigma)}, x_i)
&\ge \left(\sum_{i = 1}^n  \frac{\ell_i}{r_i} \right) \min_{i = 1, \dots, n} \left\{ r_i \m_v(a_i^{(\sigma)}, x_i) \right\} \\
&\ge t_{\bs{a}} \min_{i = 1, \dots, n} \left\{ r_i \m_v(a_i^{(\sigma)}, x_i) \right\}.
\end{align*}
By definition the global section $f$ satisfies $\ind_{1/r}(f, a^{(\sigma)}) \ge t_{\bs{a}}$, \em{i.e.} we have $f_\ell = 0$  for every $\ell \in \triangleup_r^\Z(t_{\bs{a}})$. In the non-archimedean case this yields:
\begin{align*} 
\log \| g^{(\sigma)} \cdot f \|_v
&\le \max_{\ell \in \triangledown_r^\Z(t_{\bs{a}})} \left\{ \log |f_\ell|_v + \log \| g^{(\sigma)} \ast  T(\ell)  \|_v \right\}\\
&\le-  t_{\bs{a}} \min_{i = 1, \dots, n} \left\{ r_i \m_v(a_i^{(\sigma)}, x_i) \right\} + \log \| f \|_v,
\end{align*}
which actually concludes the proof in the non-archimedean case. 
\end{np}

\begin{np}
Let us suppose henceforth that $v$ is archimedean. Proposition \ref{prop:PropertiesElementOfTheGroup} (2) and the triangle inequality give:
\begin{align} 
 \| g^{(\sigma)} \cdot f \|_{v}
& \le \sum_{\ell \in \triangledown_r^\Z(t_{\bs{a}})} |f_\ell|_v  \|g^{(\sigma)} \ast T(\ell) \|_v \nonumber \\
&\le \max_{i = 1, \dots, n} \left\{ \d_v(a_i^{(\sigma)}, x_i)^{r_i} \right\}^{t_{\bs{a}}} \left( \sum_{\ell \in \triangledown_r^\Z(t_{\bs{a}})} |f_\ell|_v \prod_{i =1}^n \sqrt{2}^{r_i - \ell_i}  \right). \label{eq:InstabilityMeasureAlgebraicPoint1}
\end{align}
Comparing $\ell^1$ and $\ell^2$ norms on $\Gamma(\P, \O_\P(r))$ thanks to Jensen's inequality we have:
$$ \sum_{\ell \in \triangledown_r^\Z(t_{\bs{a}})} |f_\ell|_v \prod_{i =1}^n \sqrt{2}^{r_i - \ell_i} \le \sqrt{ \left( \prod_{i = 1}^n (r_i + 1) \right) \sum_{\ell \in \triangledown_r^\Z(t_{\bs{a}})} |f_\ell|^2_v \prod_{i =1}^n 2^{r_i - \ell_i}  },$$
where we used $\rk \Gamma(\P, \O_\P(r)) = \prod_{i = 1}^n (r_i + 1)$. The right term can be compared with the norm of $f$:
\begin{align*}
\sum_{\ell \in \triangledown_r^\Z(t_{\bs{a}})} |f_\ell|_v^2 \prod_{i =1}^n 2^{r_i - \ell_i} 
&\le \max_{\ell \in \triangledown_r^\Z(t_{\bs{a}})} \left\{ \binom{r}{\ell} \prod_{i =1}^n 2^{r_i - \ell_i}  \right\} \left( \sum_{\ell \in \triangledown_r^\Z(t_{\bs{a}})} |f_\ell|_v^2 \binom{r}{\ell}^{-1} \right)  \\
&= \max_{\ell \in \triangledown_r^\Z(t_{\bs{a}})} \left\{ \binom{r}{\ell} \prod_{i =1}^n 2^{r_i - \ell_i}  \right\} \| f\|^2_v.
\end{align*}
Using $\sum_{a=0}^b \binom{b}{a} 2^{b-a} = 3^b$ we find
$$ \max_{\ell \in \triangledown_r^\Z(t_{\bs{a}})} \left\{ \binom{r}{\ell} \prod_{i =1}^n 2^{r_i - \ell_i}  \right\} \le 3^{|r|}$$
According to \eqref{eq:InstabilityMeasureAlgebraicPoint1} this concludes the proof.
\qed
\end{np}

\section{Semi-stability} \label{sec:Semi-stability}

\subsection{Basic facts about the semi-stability of subspaces}

\begin{npar}{Instability coefficient} Let $K$ be a field and let $G$ a $K$-reductive group acting on a proper $K$-scheme $X$ equipped with a $G$-equivariant invertible sheaf $L$. Let $x$ be a $K$-point of $X$. Let $\lambda : \Gm \to G$ be a one-parameter subgroup of $G$ (which means that $\lambda$ is a morphism of algebraic groups) and consider the morphism $\lambda_x : \Gm \to X$ given by 
$$ \lambda_x(\tau) \df \lambda(\tau) \cdot x.$$
By properness of $X$, the morphism $\lambda_x$ extends in a unique way to a morphism $\ol{\lambda}_x : \A^1 \to X$. We denote by $x_0$ the $K$-point $\ol{\lambda}_x(0)$. Since it is a fixed point under the action of $\Gm$, then $\Gm$ acts on the $K$-vector space $x_0^\ast L$ through a character $$ \tau \mapsto \tau^{- \mu_L(\lambda, x)}$$
with $\mu_L(\lambda, x) \in \Z$. We call it the \em{instability coefficient of $x$} with respect to the one-parameter subgroup $\lambda$ and the invertible sheaf $L$.\footnote{We follow here the convention adopted in \cite[Definition 2.2]{git}.}

\begin{theorem}[Hilbert-Mumford criterion] \label{thm:HilbertMumford} Let us suppose that $K$ is perfect and $L$ is ample. With the notation introduced above, the point $x$ is semi-stable if and only if
$$ \mu_L(\lambda, x) \ge 0$$
for every one-parameter subgroup $\lambda : \Gm \to G$.
\end{theorem}

When $K$ is algebraically closed, this theorem has been proved by Mumford \cite[Theorem 2.1]{git}. The general case has been proved independently by Kempf \cite[Theorem 4.2]{kempf} and Rousseau \cite{rousseau} \footnote{In order to understand that \cite[Theorem 4.2]{kempf} translates into Theorem \ref{thm:HilbertMumford} it is useful to consult the dictionary between Kempf's and Mumford's notations given in the table in \cite[Appendix to Chapter 2, section B]{git}.}. 
\end{npar}

\begin{npar}{Instability coefficient of linear subspaces} \label{par:InstabilityCoefficientLinearSubspaces} Let $V$ be a finite dimension $K$-vector space and $r$ be a non-negative integer. We consider the grassmannian of $r$-dimensional subspaces $\Grass_r(V)$ and its Pl\"ucker embedding $\varpi : \Grass_r(V) \to \P( \bigwedge^r V )$. 

Suppose that a $K$-reductive group $G$ acts linearly on $V$. Then it acts on the grassmannian $\Grass_r(V)$, on the projective space $\P( \bigwedge^r V )$ and in a equivariant way on the invertible sheaf $\O(1)$ on $\P( \bigwedge^r V )$. Since the Pl\"ucker embedding $\varpi$ is $G$-equivariant with respect to this action, the ample invertible sheaf $\varpi^\ast \O(1)$ on $\Grass_r(V)$ is naturally endowed with a $G$-equivariant action.

\begin{deff} Let $\lambda : \Gm \to G$ be a one-parameter subgroup.

\begin{enumerate}[(1)] 

\item Let $W \subset V$ be a linear subspace of dimension $r$. We set
$$ \mu(\lambda, [W]) \df \mu_{\varpi^\ast \O(1)}(\lambda, [W])$$
omitting the polarisation $\varpi^\ast \O(1)$.

\item For every integer $p \in \Z$ consider the subspace $ V_{\lambda, p} \df \{ v \in V : \lambda(\tau) \ast v = \tau^p v\}$. We define $p_{\lambda, \min}$ (resp. $p_{\lambda, \max}$) denotes the smallest (resp. the biggest) integer $p$ such that $V_{\lambda, p}$ is non-zero.

\item For every integer $p \in \Z$ we set $V[p] \df \bigoplus_{q \ge p} V[q]$.
\end{enumerate}
\end{deff}

Remark that, since the action of a torus is diagonalisable, we have $V = \bigoplus_{p \in \Z} V_{\lambda, p}$. In particular, 
$$V[p] = 
\begin{cases}
0 & \textup{if } p > p_{\lambda, \max} \\
V & \textup{if } p < p_{\lambda, \min}. 
\end{cases}
$$

\begin{prop} \label{prop:BasicPropertiesInstabilityCoefficientLinearSubspaces} Let $W \subset V$ be a linear subspace of dimension $r$. For every integer $p$ we define $W[p] \df W \cap V[p]$.
\begin{enumerate}[(1)]
\item The subspaces $W[p]$ form a decreasing filtration of $W$ and we have
\begin{align*} 
\mu(\lambda, [W]) &= \sum_{p \in \Z} p \left( \dim_K W[p] - \dim_K W[p + 1] \right)\\
&= - p_{\lambda, \min} \dim_K W - \sum_{p = p_{\lambda, \min} + 1}^{p_{\lambda, \max}} \dim_K W[p].
\end{align*}
\item Let $w_1, \dots, w_r$ be a basis of $W$. For every $i = 1, \dots, r$ let $\mu(\lambda, [w_i])$ be the instability coefficient of the point $[w_i] \in \P(V)$. Then the vector $w_i$ writes as
$$ \lambda(\tau) \ast w_i = \tau^{- \mu(\lambda, [w_i])} w_{i, \min} +  \textup{terms of higher order in } \tau,$$
with $w_{i, \min} \in V$. If the elements $w_{1, \min}, \dots, w_{r, \min} \in V$ are linearly independent, then
$$ \mu(\lambda, [W]) = \sum_{i = 1}^r \mu(\lambda, [w_i]). $$
\item With the notations of (2), there exists a basis $w_1, \dots, w_r$ of $W$ such that their components of minimal weight $w_{1, \min}, \dots, w_{r, \min} \in V$ are linearly independent.
\end{enumerate}

\begin{proof} This is a reformulation of the computations in \cite[Chapter 4, \S 4]{git}. See also \cite[\S 2, Lemma 2]{Totaro}. 
\end{proof}

\end{prop}

\begin{prop} \label{prop:PropertiesInstabilityCoefficientLinearSubspaces} Let $W_1$, $W_2$ be subvector spaces of $V$. 
We have the following properties:
\begin{enumerate}[(1)]
\item \label{eq:InclusionFormulaInstabilityCoefficient} \textup{(Inclusion formula)} If $W_1$ is contained in $W_2$ we have:
\begin{equation}  \mu(\lambda, [W_1]) \ge \mu(\lambda, [W_2]) - p_{\lambda, \min}(\dim_K W_1 - \dim_K W_2). \end{equation}
\item \label{eq:GrassmannFormulaInstabilityCoefficient} \textup{(Grassmann formula)}
\begin{equation}  \mu(\lambda, [W_1]) + \mu(\lambda, [W_2]) \ge \mu(\lambda, [W_1+ W_2]) + \mu(\lambda, [W_1\cap W_2]). \end{equation}
\end{enumerate}
\end{prop}

\begin{proof} (1) Clear. (2) In fact for every integer $p$ the usual Grassmann formula for linear subspaces gives
$$ \dim_K W_1[p] + \dim_K W_2[p] = \dim_{K} (W_1[p] + W_2[p]) + \dim_{K} (W_1[p] \cap W_2[p])$$
and we conclude noticing that $W_1[p] + W_2[p] \subset (W_1 + W_2)[p]$.
\end{proof}
\end{npar}
\subsection{Asymptotic semi-stability : proof of Theorem \ref{thm:Semi-stabilityCondition}} \label{sec:ProofOfSemiStabilityTheorem}

\begin{np} We go back to the notation introduced in Section \ref{subsec:DefinitionModuliProblem}. The construction of invariant elements is compatible with flat base change \cite[\S 2 Lemma 2]{seshadri77}. It follows that the semi-stability of the points $P_{\alpha r}$ is only a matter of the generic fiber of $\cal{X}_{\alpha r}$. From now on we silently work over $K$ (for instance $\P$ will denote the projective scheme $(\P^1_K)^n$).

\end{np}

\begin{npar}{Computation of the instability coefficients} \label{par:ComputationInstabilityCoefficients} For every $n$-tuple of positive integers $r = (r_1, \dots, r_n)$, every non-negative real number $t \ge 0$ and every $i \in \{1, \dots, n\}$ let us set:
$$ \mu_{r, i}^\Z(t) \df \sum_{\ell \in \triangledown_r^\Z(t)} 2 \ell_i - r_i$$
Arguments similar to those in Lemma \ref{lem:PropertiesFunctionMu} show that $\mu_{r, i}^\Z(t)$ is non-negative. 

\begin{deff} Let $\lambda: \Gm \to \SLs_{2, K}^n$ be a one-parameter subgroup. 
\begin{enumerate}[(1)]
\item  For every $i = 1, \dots, n$ there exists a basis $T_{i0}, T_{i1}$ of $K^{2 \vee}$ and a non-negative integer $m_{\lambda, i} \ge 0$ such that 
$$ \lambda(\tau) \ast T_{i 0} = \tau^{m_{\lambda, i}} T_{i0}, \quad \lambda(\tau) \ast T_{i1} = \tau^{-m_{\lambda, i}} T_{i1}$$ for every $\tau \in \Gm(K)$. The $n$-tuple of non-negative integers $m_\lambda = (m_{\lambda, 1}, \dots, m_{\lambda, n})$ is called the \em{weight of $\lambda$} and the bases $T_{i0}, T_{i1}$ (for $i = 1, \dots, n$) is called an \em{adapted basis for $\lambda$}. 

Note that $m_{\lambda, i}$ does not depend on the choice of an adapted basis and if $m_{\lambda, i}$ is non-zero then the lines $\{ T_{i0} = 0 \}$ and $\{ T_{i1} = 0 \}$ are determined by $\lambda$.

\item With the notations introduced here above, for every $\ol{\Q}$-point $y$ of $\P$ we set
$$  \chi_{\lambda, i}(y) \df 
\begin{cases}
1 & \textup{if $T_{i0}$ vanishes at $y$} \\
0 & \textup{otherwise.}
\end{cases}$$
We denote by $y_\lambda$ the unique $K$-point of $\P$ such that $\chi_{\lambda, i}(y_\lambda) = 1$ for all $i = 1, \dots, n$ and we call it the \em{instability point of $\lambda$} (with respect to the chosen adapted bases).
\end{enumerate}
\end{deff}

\begin{prop}[Instability coefficient at the single point] \label{prop:InstabilityCoefficientAtSinglePoint} Let $\lambda : \Gm \to \SLs_2^n$ be a one-parameter subgroup. With the notation introduced above, for every $i \in \{1, \dots, n\}$ we have
$$ 
\mu(\lambda, [K_r(x, t_x)]) = \sum_{i = 1}^n (-1)^{\chi_{\lambda, i}(x)} m_{\lambda, i}  \mu_{r, i}^\Z(t_x).
$$
\end{prop}

\begin{prop}[Instability coefficient at the algebraic point] \label{prop:InstabilityCoefficientAtTargetPoints} Let $\delta$ be a positive real number. Under the assumptions of Theorem \ref{thm:Semi-stabilityCondition} there exist a positive real number $\delta_0$ and a positive integer $\alpha_0$ (the two of them possibly depending on $n$, $q$, $r$, $t_{\bs{a}}$ and $t_x$) satisfying the following properties: for every one-parameter subgroup $\lambda : \Gm \to \SLs_2^n$, every integer $\alpha \ge \alpha_0$ and every real number $0 < \rho < \rho_0$ we have
$$\mu(\lambda, [K_{\alpha r}(a, t_{\bs{a}})]) \ge \sum_{i = 1}^n m_{\lambda, i}\left[\mu_{\alpha r, i}^\Z(u_{q,r}(t_{\bs{a}}) + \rho) - \alpha^{n+1} r_i (r_1 \cdots r_n)(\epsilon_{q, r} + \delta) \right].$$
\end{prop}

Before showing Propositions \ref{prop:InstabilityCoefficientAtSinglePoint} and \ref{prop:InstabilityCoefficientAtTargetPoints} (whose proofs will be respectively expounded in paragraphs \ref{par:InstabilityCoefficientRationalPoint} and \ref{par:InstabilityCoefficientAlgebraicPoint}), we will show how to deduce Theorem \ref{thm:Semi-stabilityCondition}. 
\end{npar}

\begin{npar}{Proof of Theorem  \ref{thm:Semi-stabilityCondition}} \label{par:ProofOfSemiStabilityTheorem} According to the Hilbert-Mumford criterion (Theorem \ref{thm:HilbertMumford}) it suffices to show that there exists $\alpha_0$ such that for every $\alpha \ge \alpha_0$ and every one-parameter subgroup $\lambda : \Gm \to \SLs_2^n$ we have
$$ \mu(\lambda, P_{\alpha r}) = \mu(\lambda, [K_{\alpha r}(x, t_x)]) + \mu(\lambda, [K_{\alpha r}( a, t_{\bs{a}})]) \ge 0.$$

Let $\delta$ be a positive real number. Let $\alpha_0$, $\delta_0$ and $\rho_0$ given by Proposition \ref{prop:InstabilityCoefficientAtTargetPoints}. Up to increasing $\alpha_0$ and decreasing $\delta$ and $\rho_0$ we may assume that for every $i = 1, \dots, n$, every $\alpha \ge \alpha_0$ and every $0 < \rho < \rho_0$ we have:
\begin{equation} 
\tag{SS$'$} \label{eq:EndProofOfSemiStabilityTechnicalCondition1} \mu_{\alpha r, i}^\Z(u_{q, r}(t_{\bs{a}}) + \rho) >  \mu_{\alpha r, i}^\Z(t_x) +  \alpha^{n+1} r_i (r_1 \cdots r_n) (\epsilon_{q, r} + \delta), \\
 \end{equation}
Fix $\alpha \ge \alpha_0$ and $0 < \rho < \rho_0$. Propositions \ref{prop:InstabilityCoefficientAtSinglePoint} and \ref{prop:InstabilityCoefficientAtTargetPoints} yield that $\mu(\lambda, P_{\alpha r})$ is non-negative if:
$$ \sum_{i = 1}^n m_{\lambda, i} \left[  \mu_{\alpha r, i}^\Z(u_{q, r}(t_{\bs{a}}) + \rho)  - \alpha^{n+1} r_i (r_1 \cdots r_n) (\epsilon_{q, r} + \delta) + (-1)^{\chi_{\lambda, i}(x)} \mu_{\alpha r, i}^\Z(t_x) \right] \ge 0.$$
Since the integers $m_{\lambda, i}$ are supposed to be non-negative, this is satisfied according to \eqref{eq:EndProofOfSemiStabilityTechnicalCondition1}. This concludes the proof.
\qed
\end{npar}

\begin{npar}{Proof of Proposition \ref{prop:InstabilityCoefficientAtSinglePoint}} \label{par:InstabilityCoefficientRationalPoint} Let us consider adapted bases $T_{i0}, T_{i1}$ for $\lambda$ ($i = 1, \dots, n$). We suppose $\chi_{\lambda, i}(x) = 0$ for every $i = 1, \dots, n$ and we let the reader adapt the argument in the other cases.

Under this assumption there exists $\xi_i \in K$ such that $T_{i1} - \xi_i T_{i0}$ vanishes at $x_i$. According to Proposition \ref{prop:DimensionKernelSinglePoint} (1), a basis of the $K$-vector space $K_r(x, t_x)$ is given by polynomials of the form
$$ T(\ell) \df \bigotimes_{i = 1}^n T_{i0}^{r_i - \ell_i} (T_{i1} - \xi_i T_{i0})^{\ell_i}.$$
where $\ell = (\ell_1, \dots, \ell_n) \in \triangledown_r(t_x)$. The action of the one-parameter subgroup $\lambda$ is given by 
$$
\lambda_i(\tau) \ast T(\ell) = \bigotimes_{i = 1}^n \tau^{m_{\lambda, i}(r_i - 2 \ell_i)} \left(T_{i0}^{r_i - \ell_i} (T_{i1} - \tau^{2 m_{\lambda, i}} \xi_i T_{i0})^{\ell_i} \right)
$$
Since the integers $m_{\lambda, i}$ are supposed to be non-negative, the component of $T(\ell)$ of minimal weight is the polynomial multiplied by $\sum_{i = 1}^n m_{\lambda, i}(r_i - 2 \ell_i)$, that is
$$ T(\ell)_{\min} = \bigotimes_{i = 1}^n T_{i0}^{r_i - \ell_i} T_{i1}^{\ell_i}. $$
Since the elements $T(\ell)_{\min}$ for $\ell \in \triangledown_r^\Z(t_x)$ are linearly independent, Proposition \ref{prop:BasicPropertiesInstabilityCoefficientLinearSubspaces} (2) yields
$$ \mu(\lambda, [K_r(x, t_x)]) = \sum_{i = 1}^n \sum_{\ell \in \triangledown_r^\Z(t_x)} m_{\lambda, i}(2 \ell_i - r_i) = \sum_{i =1}^n m_{\lambda, i} \mu_{r, i}^\Z(t_x),$$
which concludes the proof. \qed
\end{npar}

\begin{npar}{Proof of Proposition \ref{prop:InstabilityCoefficientAtTargetPoints}} \label{par:InstabilityCoefficientAlgebraicPoint}

\begin{lem} \label{lem:InstabilityCoefficientAlgebraicPointGrassmannFormula} Let $\lambda : \Gm \to \SLs_2^n$ be a one-parameter subgroup. Let us suppose moreover $u_{q,r}(t_{\bs{a}}) \neq 0, n$\footnote{If $u_{q,r}(t_{\bs{a}}) \in \{0, n\}$ then Condition \eqref{eq:SemiStabilityConditionMainTheorem} in Theorem \ref{thm:Semi-stabilityCondition} is not satisfied because $\mu_n(0) = \mu_n(n) = 0$.} and therefore $\vol \triangleup_n(u_{q, r}(t_{\bs{a}})) = 1 + \epsilon_{q, r} - q \vol \triangleup_n(t_{\bs{a}})$ by Definition \ref{def:DefinitionOfMoreCombinatorics} (4).  Let $a^{(0)}$ be a $K$-rational point of $\P^1$. For every $\rho > 0$ we have:
\begin{multline*}
\mu(\lambda, [K_{q,r}(\bs{a}, t_{\bs{a}})]) + \mu(\lambda, [K_r(a^{(0)}, u_{q,r}(t_{\bs{a}}) + \rho)]) \\ \ge  (k_{q, r}(t_{\bs{a}}) + k_r(u_{q,r}(t_{\bs{a}}) + \rho) - \dim \Gamma(\P, \O_\P(r)) ) \left( \sum_{i = 1}^n m_{\lambda, i}r_i \right). \end{multline*}
\end{lem}

\begin{proof} Since the point $a^{(0)}$ is $K$-rational we have $\pr_i(a^{(0)}) \neq \pr_i(a^{(\sigma)})$ for every $i = 1, \dots, n$ and every $\sigma : K' \to \ol{\Q}$. According to Proposition \ref{Prop:DimensionKernelAtMultiplePoints} (3) the intersection $$K_{q,r}(\bs{a}, t_{\bs{a}}) \cap K_r(a^{(0)}, u_{q,r}(t_{\bs{a}}) + \rho)$$ is zero (we may apply it over $\ol{\Q}$ to deduce it over $K$). Grassmann's formula for instability coefficients (Proposition \ref{prop:PropertiesInstabilityCoefficientLinearSubspaces} \ref{eq:GrassmannFormulaInstabilityCoefficient}) applied to the subspaces $K_{q,r}(\bs{a}, t_{\bs{a}})$ and $K_r(a^{(0)}, u_{q,r}(t_{\bs{a}}) + \rho)$ yields:
\begin{multline*} \mu(\lambda, [K_{q,r}(\bs{a}, t_{\bs{a}})]) + \mu(\lambda, [K_r(a^{(0)}, u_{q,r}(t_{\bs{a}}) + \rho)]) \\ \ge \mu(\lambda, [K_{q,r}(\bs{a}, t_{\bs{a}}) + K_r(a^{(0)}, u_{q,r}(t_{\bs{a}}) + \rho)])
\end{multline*}

With the notations introduced in paragraph \ref{par:InstabilityCoefficientLinearSubspaces}, the smallest integer $b$ such that the vector space $\Gamma(\P, \O_\P(r))_b$ is non-zero is $ -  \sum_{i = 1}^n m_{\lambda, i}r_i$ (it occurs only for the monomial $T_{10}^{r_1} \otimes \cdots \otimes T_{n0}^{r_n}$). The inclusion formula (Proposition \ref{prop:PropertiesInstabilityCoefficientLinearSubspaces} \ref{eq:InclusionFormulaInstabilityCoefficient}) applied to the inclusion
$$K_{q,r}(\bs{a}, t_{\bs{a}}) + K_r(y, u_{q,r}(t_{\bs{a}}) + \rho) \subset \Gamma(\P, \O_\P(r)) $$ gives the result.
\end{proof}

\begin{lem} \label{lem:TechnicalConditionsOnAlpha} Let $\delta$ be a positive real number. Under the assumption of Theorem \ref{thm:Semi-stabilityCondition} there exist a positive real number $\rho_0$ and a positive integer $\alpha_0$ (the two of them possibly depending on $n$, $d$, $r$, $t_{\bs{a}}$ and $t_x$) such that, for every integer $\alpha \ge \alpha_0$ and every real number $0  < \rho < \rho_0$ we have:
\begin{enumerate}[(1)]
\item$\displaystyle \left| \frac{\dim_K \Gamma(\P, \O_\P(\alpha r))}{\alpha^n (r_1 \cdots r_n)} - 1\right| < \frac{\delta}{3};$

\item \vspace{7pt} $ \displaystyle \left| \frac{\# \triangledown^\Z_{\alpha r}(u_{q, r}(t_{\bs{a}}) + \rho)}{\alpha^n (r_1 \cdots r_n)} - (q \vol \triangleup_n(t_{\bs{a}}) - \epsilon_{q, r}) \right|  < \frac{\delta}{3}$;

\item \vspace{7pt}$ \displaystyle \frac{k_{q, \alpha r}( t_{\bs{a}})}{\alpha^n (r_1 \cdots r_n)} > (1 - q \vol \triangleup_n(t_{\bs{a}})) - \frac{\delta}{3}.$
\end{enumerate}
\end{lem}

\begin{proof} (1) is clear; (2) follows from the definition of $u_{q,r}(t_{\bs{a}})$; (3) follows from Proposition \ref{Prop:DimensionKernelAtMultiplePoints} (2).
\end{proof}

\begin{proof}[Proof of Proposition \ref{prop:InstabilityCoefficientAtTargetPoints}] Let $\delta$ be a positive real number and let $\alpha_0$ and $\rho_0$ given by Lemma \ref{lem:TechnicalConditionsOnAlpha}. Let us take an integer $\alpha \ge \alpha_0$ and a real number $0 < \rho < \rho_0$.

Let $a^{(0)} \in \P(K)$ be the unique point such that $\chi_{\lambda, i}(a^{(0)}) = 1$ for every $i = 1, \dots, n$. Applying Proposition \ref{prop:InstabilityCoefficientAtSinglePoint} to the point $a^{(0)}$ we get : 
$$ \mu(\lambda, [K_{\alpha r}(a^{(0)}, u_{q, r}(t_{\bs{a}}) + \rho)]) =  - \sum_{i = 1}^n m_{\lambda, i} \mu_{\alpha r, i}^\Z(u_{q, r}(t_{\bs{a}}) + \rho). $$
According to Lemma \ref{lem:TechnicalConditionsOnAlpha} and Proposition \ref{prop:DimensionKernelSinglePoint} (2) we have:
$$k_{q, \alpha r}( t_{\bs{a}}) + k_{\alpha r}(u_{q, r}(t_{\bs{a}}) + \rho) - \dim \Gamma(\P, \O_\P(\alpha r))  \ge - \alpha^{n} (r_1 \cdots r_n) (\epsilon_{q, r} + \delta) ,$$
Therefore Lemma \ref{lem:InstabilityCoefficientAlgebraicPointGrassmannFormula} yields:
$$
\mu(\lambda, [K_{\alpha r}(a, t_{\bs{a}})]) 
\ge \sum_{i = 1}^n m_{\lambda, i} \left[ \mu_{\alpha r, i}^\Z(u_{q, r}(t_{\bs{a}}) + \rho)) - \alpha^{n + 1} r_i (r_1 \cdots r_n) (\epsilon_{q, r} + \delta) \right],
$$
which concludes the proof.
\end{proof}

\end{npar}

\subsection{Semi-stability in the two-dimensional case}

In this section we are going to prove semi-stability in the case $n = 2$ in a slightly different way. In this section let us write
$$ \epsilon_{q, r} = (q - 1) \frac{\min\{ r_1, r_2\}}{\max\{r_1, r_2\}}.$$
When $r_1 \ge r_2$ this coincides with previous definition.

\begin{np} The semi-stability statement we prove is the following one:

\begin{theorem} \label{thm:Semi-stabilityConditionN=2} Let $r = (r_1, r_2)$ be a couple of positive integers such that $r_1 \ge r_2$. Let $t_x, t_{\bs{a}} \ge 0$ be non-negative real numbers such that $ 0 \le 1 - q \vol \triangleup_2(t_{\bs{a}}) + \epsilon_{q+1, r} \le 1/2$. If the inequality
\begin{equation} \label{eq:Semi-stabilityConditionN=2}
\mu_2(t_x) < \left(1 - q \vol \triangleup_2(t_{\bs{a}}) \right) \left(1 - 2 \sqrt{2 \left( 1 - q \vol \triangleup_2(t_{\bs{a}}) + \epsilon_{q+1, r} \right)} \right) 
\end{equation}
is satisfied then there exists a positive integer $\alpha_0 = \alpha_0(q, r, t_{\bs{a}}, t_x)$ such that, for every integer $\alpha \ge \alpha_0$, the $K$-point $P_{\alpha r} \in \cal{X}_{\alpha r} (K)$ is semi-stable under the action of $\SLs_2^2$ with respect to the polarization given by the Pl\"ucker embeddings.
\end{theorem}

In particular, given $0 < \delta < 1$ we can apply it with
\begin{itemize}
\item $\displaystyle t_{\bs{a}} = t_{q, 2}(\delta) = \sqrt{\frac{2}{q} (1 - \delta)}$;
\item  $t_x$ that tends to the unique real number $w \in [1, 2]$ such that $$ \mu_2(w) = \delta \left(1 - 2 \sqrt{2 \left( \delta + \epsilon_{q+1, r} \right)} \right).$$
\end{itemize}
This is enough to derive the Main Effective Lower Bound in the case $n = 2$.

\begin{prop}[Instability coefficient at the algebraic point] \label{prop:InstabilityCoefficientAtTargetPointsN=2} Let $\delta$ be a positive real number. Under the assumptions of Theorem \ref{thm:Semi-stabilityConditionN=2} there exists a positive integer $\alpha_0$ (possibly depending on $q$, $r$, $t_{\bs{a}}$ and $t_x$) satisfying the following properties: for every one-parameter subgroup $\lambda : \Gm \to \SLs_{2,K}^2$ and every integer $\alpha \ge \alpha_0$:
\begin{multline*}
\mu(\lambda, [K_{q, \alpha r}(\bs{a}, t_{\bs{a}})]) \ge \alpha^3 r_1 r_2 \langle m_\lambda, r \rangle \left(1 - q \vol \triangleup_2(t_{\bs{a}}) - \delta \right) \times \\ \times
\left(1 - 2 \sqrt{2 \left( 1 - q \vol \triangleup_2(t_{\bs{a}}) + \epsilon_{q+1, r} \right)} \right) ,
\end{multline*}
where $\langle -, - \rangle$ denotes the standard scalar product on $\R^2$.
\end{prop}

We leave to the reader adapting the argument given in paragraph \ref{par:ProofOfSemiStabilityTheorem} in order to obtain Theorem \ref{thm:Semi-stabilityConditionN=2} from Proposition \ref{prop:InstabilityCoefficientAtTargetPointsN=2}. The main ingredient in the proof of Proposition \ref{prop:InstabilityCoefficientAtTargetPointsN=2} is the following:

\begin{prop} \label{prop:InstabilityCoefficientOfASinglePolynomial} Let $f \in K_r(\bs{a}, t_{\bs{a}})$ be a non-zero section. If the following condition is satisfied,
$$ 1 - q \vol \triangleup_2(t_{\bs{a}}) + \epsilon_{q+1, r} \le 1/2, $$
then, for every one parameter subgroup $\lambda : \Gm \to \SLs^2_{2, K}$ we have
$$ \mu(\lambda, [f]) \ge  \langle m_\lambda, r \rangle   \left(1 - 2 \sqrt{2 \left( 1 - q \vol \triangleup_2(t_{\bs{a}}) + \epsilon_{q+1, r} \right)} \right)$$
the instability coefficient of $f$ being taken as a point of $\P(\Gamma(\P, \O_\P(r)))$ and with respect to the invertible sheaf $\O(1)$.
\end{prop}

Let us show how to deduce Proposition \ref{prop:InstabilityCoefficientAtTargetPointsN=2}.

\begin{proof}[Proof of Proposition \ref{prop:InstabilityCoefficientAtTargetPointsN=2}] Let $\delta$ be a positive real number. Analogously to the proof of Lemma \ref{lem:TechnicalConditionsOnAlpha} we may find a positive integer $\alpha_0$ such that for every integer $\alpha \ge \alpha_0$ we have:
\begin{equation} \label{eq:ApproximationOfVolumesForN=2} \frac{k_{q, \alpha r}( t_{\bs{a}})}{\alpha^2 r_1 r_2} > (1 - q \vol \triangleup_2(t_{\bs{a}})) - \delta. \end{equation}
Let $f_1, \dots, f_{k_{q, \alpha r}( t_{\bs{a}}) }$ be a basis of $K_{q, \alpha r}(\bs{a}, t_{\bs{a}})$ such that their components of minimal weight (with respect to $\lambda$)
$f_{1, \min}, \dots, f_{k_{q, \alpha r}( t_{\bs{a}}), \min}$ are linearly independent (such a basis exists according to Proposition \ref{prop:BasicPropertiesInstabilityCoefficientLinearSubspaces}). Then Proposition \ref{prop:BasicPropertiesInstabilityCoefficientLinearSubspaces} (2) entails
\begin{align*} 
\mu(\lambda, [K_{q, \alpha r}(\bs{a}, t_{\bs{a}})]) &= \sum_{\ell = 1}^{k_{q, \alpha r}(t_{\bs{a}})} \mu(\lambda, [f_\ell]) \\
&\ge \alpha k_{q, \alpha r}(t_{\bs{a}}) \langle m_\lambda, r \rangle  \left(1 - 2 \sqrt{2 \left( 1 - q \vol \triangleup_2(t_{\bs{a}}) + \epsilon_{q+1, r} \right)} \right),
\end{align*}
where the second inequality follows from Proposition \ref{prop:InstabilityCoefficientOfASinglePolynomial}. We conclude using \eqref{eq:ApproximationOfVolumesForN=2}.
\end{proof}

The result we are going to prove in what follows is actually the following version of Dyson's Lemma:

\begin{theorem}[{\em{cf}. Corollary \ref{cor:DysonsLemmaWtihTwoWeights}}] \label{thm:DysonsLemmaTwoWeights} Let $f \in \Gamma(\P, \O(r))$ be a non-zero global section and let $b = (b_1, b_2)$ be a couple of non-negative real numbers. 

Let $y$ be a $\ol{\Q}$-point of $\P^1$. Let $q \ge 1$ be an integer and  for every $\sigma = 1, \dots, q$ let $z^{(\sigma)}$ be a $\ol{\Q}$-point of $\P$. For every $i = 1, 2$ let us suppose:
\begin{enumerate}[(1)]
\item $\pr_i(z^{(\sigma)}) \neq \pr_i(z^{(\tau)})$ for every $\sigma \neq \tau$;
\item $\pr_i(z^{(\sigma)}) \neq \pr_i(y)$ for every $\sigma = 1, \dots, q$.
\end{enumerate}
For every $\sigma = 1, \dots, q$ let us suppose $t_\sigma \le 1$ and 
\begin{equation} \label{eq:CorAlmostDysonsLemmaInstabilityCoefficientHypothesis} 
1 - \sum_{\sigma = 1}^q \vol \triangleup_2(\ind_{1/r}(f, z^{(\sigma)})) + \epsilon_{q + 1, r} < \frac{1}{2}. \end{equation}
Then, $\ind_{b}(f, y) < \max \left\{ b_i r_i \right\}$ and consequently
$$
\vol \triangleup_2 \left( \frac{\ind_{b}(f, y)}{\max \left\{ b_i r_i \right\}} \right) \le 
  1 - \sum_{\sigma = 1}^q \vol \triangleup_2(\ind_{1/r}(f, z^{(\sigma)})) + \epsilon_{q + 1, r}.
$$
\end{theorem}

Before showing how to deduce Proposition \ref{prop:InstabilityCoefficientOfASinglePolynomial} let us remark that a similar version of Dyson's Lemma with two different weights is indicated to hold in \cite[page 489]{esnault-viewheg}. The hypothesis \eqref{eq:CorAlmostDysonsLemmaInstabilityCoefficientHypothesis} makes quantitative the assertion of Esnault-Viehweg that $\ind_b(f, y)$ should be ``very small'' (see \em{loc. cit.}).

In order to prove Proposition \ref{prop:InstabilityCoefficientOfASinglePolynomial} let us link the instability coefficient and index through the following easy fact (whose proof is left to the reader):

\begin{prop} \label{prop:ComparingInstabilityCoefficientAndIndex}Let $\lambda : \Gm \to \SLs_{2, K}^2$ be a one parameter subgroup and let us fix an adapted basis for $\lambda$. Then, for every non-zero element $f \in \Gamma(\P, \O_\P(r))$ we have
$$ \mu(\lambda, [f]) = \langle m_\lambda, r \rangle - 2 \ind_{m_\lambda}(f, y_\lambda),$$
where $y_\lambda$ is the instability point of $\lambda$ (with respect to the chosen adapted bases).
\end{prop}

\begin{proof}[Proof of Proposition \ref{prop:InstabilityCoefficientOfASinglePolynomial}] It suffices to apply Theorem \ref{thm:DysonsLemmaTwoWeights} to the points $z^{(\sigma)} = a^{(\sigma)}$ for $\sigma : K' \to \ol{\Q}$, the weight $b = m_\lambda$ and the point $y = y_\lambda$ (the instability point of $\lambda$) and to remark that, since $t_{\bs{a}} \le 1$ (otherwise \eqref{eq:Semi-stabilityConditionN=2} is not satisfied), we have $\vol \triangleup_2(t_{\bs{a}}) = t_{\bs{a}}^2 / 2$.
\end{proof}

The rest of this section is therefore devoted to prove Theorem \ref{thm:DysonsLemmaTwoWeights}.

In view of Proposition \ref{prop:ComparingInstabilityCoefficientAndIndex}, one would like to use this interpretation of the instability measure to apply the usual Dyson's Lemma --- \em{i.e.} Theorem \ref{thm:DysonsLemma} when $n = 2$ \footnote{This is case originally treated by Dyson \cite{dyson}, whose proof has then been revisited by several authors (see \cite{bombieri}, \cite{ViolaDyson} and \cite{VojtaDysonCurves}).} ---  in order to derive the semi-stability of the point. Unfortunately, the usual Dyson's Lemma can be applied when the weight of the index is the same at all points: here instead we have to apply it to weight $1 / r$ at the points $z^{(\sigma)}$ for $\sigma = 1, \dots, q$ and the weight $m_\lambda$ at the point $y_\lambda$.

The key point in the proof of Theorem \ref{thm:DysonsLemmaTwoWeights} is that in general the index of a polynomial taken with respect to two different weights are not comparable. Anyway this is the case when the polynomial is a product of polynomials in separate variables:

\begin{prop} \label{prop:ComparingTheIndexWithRespectTwoWeights} Let $b = (b_1, b_2)$ and $c = (c_1, c_2)$ be couples of non-negative real numbers. Let us suppose that is made of positive real numbers. 

For all $i = 1, 2$ let $f_i \in \Gamma(\P^1, \O_{\P^1}(r_i))$ be a non-zero section. Then, for all $\ol{\Q}$-point $z$ of $\P^1$ we have
$$ \ind_{b}(f_1 \otimes f_2, z) \le \max \{ b_i / c_i\} \ind_c(f_1 \otimes f_2, z).$$
\end{prop}

In a nutshell, in the proof of Theorem \ref{thm:DysonsLemmaTwoWeights} we will use the Wronskian to be led back to the latter case.

%In this section for every non-negative integer $r$ let us $\Sym^r K^{2 \vee}$ for the vector space $\Sym^r K^{2 \vee}$.

\end{np}

\begin{npar}{Homogeneous Wronskian} In this paragraph we introduce the wronskian as an invariant under of $\SLs_{2, K}$. We follow the presentation given in \cite[2.8]{Abdesselam}. Let $r, \rho$ be non-negative integers such that $\rho \le r + 1$.

\begin{deff} Let $f_1, \dots, f_\rho \in \Sym^{r} K^{2 \vee} $  and let $T_0, T_1$ be the canonical basis of $K^{2 \vee}$.  The \em{homogeneous Wronskian} of the polynomials $f_1, \dots, f_\rho$ is defined as:
$$ \Wr(f_1, \dots, f_\rho) \df \left( \frac{(r - \rho + 1)!} {r!}\right)^\rho \cdot \det \left( \frac{\partial^{\rho- 1}f_\ell}{\partial T_0^{\rho - j} \partial T_1^{j - 1}} \right)_{j, \ell = 1, \dots, \rho}. $$
It is an element of $\Sym^{\rho(r - \rho + 1)} K^{2 \vee}$, \em{i.e.} it is a homogeneous polynomial of degree $\rho(r - \rho + 1)$ in the variables $T_0, T_1$ (indeed, each entry is a homogeneous polynomial of degree $r - ( \rho -1)$).
\end{deff}

The reader may consult \cite[2.9]{Abdesselam} for the relation with the classical notion of Wronskian. It follows from Wronski's criterion of linear independence \cite[Proposition 6.3.10]{bombieri-gubler} that $f_1, \dots, f_\rho$ are linearly independent if and only if 
 $ \Wr(f_1 , \dots,  f_\rho)$ does not vanish.
 
 The Wronskian is an alternating multi-linear map on $\Sym^{r} K^{2 \vee}$ and therefore it can be extended to a linear map
 $$ \Wr : \bigwedge^\rho \Sym^{r} K^{2 \vee} \too \Sym^{\rho(r - \rho + 1)} K^{2 \vee} .$$
 
\begin{prop}[{\cite[2.3, 2.5 and 2.8]{Abdesselam}}] \label{prop:WronskianIsInvariant} The following properties are satisfied:

\begin{enumerate}[(1)]

\item If $T_0', T_1'$ is a basis and  $ \Wr' : \bigwedge^\rho \Sym^{r} K^{2 \vee} \to \Sym^{\rho(r - \rho + 1)} K^{2 \vee}  $ is the Wronskian map taken with respect to the latter basis, then as linear maps we have
$$ \Wr' = \det (T_0', T_1')^{\rho(r - \rho + 1)} \Wr,$$
where $(T_0', T_1')$ is the linear map sending $T_i$ on $T_i'$ for $i = 0, 1$.

\item The linear map $\Wr$ is equivariant under the natural action of $\SLs_{2, K}$ on the vector spaces $\bigwedge^\rho \Sym^{r} K^{2 \vee}$ and $\Sym^{\rho(r - \rho + 1)} K^{2 \vee} $.

\end{enumerate}
\end{prop}

%Let us relate this definition with its ``inhomogeneous'' version. For every $\ell = 1, \dots, \rho$ let set $\tilde{f}_\ell(T_1 / T_0) = f_\ell(T_0, T_1) / T_0^r$. Let us denote by $x$ the variable of the polynomial $\tilde{f}_\ell$. The \em{inhomogeneous Wronskian} of the polynomials $\tilde{f}_1, \dots, \tilde{f}_\rho$ is the determinant:
%$$ \tilde{\Wr}(\tilde{f}_1, \dots, \tilde{f}_\rho) \df \det \left( \frac{\partial^{j - 1} \tilde{f}_\ell}{\partial x^{j - 1}}\right)_{j , \ell = 1, \dots, \rho}.$$

%\begin{prop}[{\cite[2.9]{Abdesselam}}] \label{prop:InhomogeneousWronskianOneVariable} With the notation introduced above, we have:
%$$ \Wr(f_1, \dots, f_\rho)(T_0, T_1) = \left( \prod_{i = 1}^{\rho - 2} (r - i)^{i - \rho + 1}\right) T_0^{\rho(r - \rho + 1) } \tilde{\Wr}(\tilde{f}_1, \dots, \tilde{f}_\rho) (T_1/ T_0 ). $$
%\end{prop} 
 
 \end{npar}

\begin{npar}{Tensorial rank} Let $V_1, V_2$ be finite-dimensional $K$-vector spaces.

\begin{deff} The \em{tensorial rank} of a non-zero vector $v \in V_1 \otimes_K V_2$ is the minimal integer $\rho \ge 0 $ such that $v$ can be written in the form $v_{11} \otimes v_{21} + \cdots + v_{1\rho} \otimes v_{2\rho}$ with $v_{i \ell} \in V_i$  for $i = 1, 2$  and $\ell = 1, \dots, \rho$. We denote it by $\rk(v)$.
\end{deff}

The tensorial rank is invariant under homotheties and under $\GLs(V_1) \times \GLs(V_2)$ (acting naturally component-wise). The tensorial rank of $v$ coincides with the rank of the linear map $V_1^\vee \to V_2$ associated to $v$ through the canonical isomorphism  $V_1 \otimes_K V_2 \iso \Hom_K(V_1^\vee, V_2)$ (analogously it is the rank of the dual map $V_2^\vee \to V_1$).

Let us also remark that if $v_{i1}, \dots, v_{i\rho} \in V_i$  for $i = 1, 2$ are families of vectors such that $v = \sum_{\ell = 1}^{\rk(v)} v_{1\ell} \otimes v_{2 \ell}$ then
the vectors $v_{i 1}, \dots, v_{i \rk(v)}$ are linearly independent for all $i = 1, 2$.

%\begin{deff} For every integer $\rho \ge 0$ let $\P(V \otimes_K W)_\rho$ be the subset of $\P(V \otimes W)$ of points having tensorial rank $\rho$. It is locally closed and stable under the action of $\GLs(V) \times \GLs(W)$.
%\end{deff}

\end{npar}

\begin{npar}{Splitting polynomials through the Wronskian} Let $r = (r_1, r_2)$ be a couple of positive integers and for $i = 1, 2$ let us set $V_i \df \Sym^{r_i} K^{2 \vee} $. With this notation we have
$$ \Gamma(\P, \O_\P(r)) \iso V_1 \otimes V_2.$$
For every non-zero $f \in \Gamma(\P, \O_\P(r))$ we can consider its tensorial rank $\rk(f)$ with respect to this decomposition. In the notation of \cite[page 266]{bombieri} we have $s_2(f) = \rk(f) + 1$.

For every $i = 1, 2$ let us fix a basis $T_{i0}$, $T_{i1}$ of $K^{2\vee}$.

\begin{deff} Let $ f \in \Gamma(\P, \O_\P(r))$ be a non-zero section and let $\rho = \rk(f)$ be its tensorial rank. For every couple of positive integers $\ell = (\ell_1, \ell_2)$ such that $\ell_i \le \rho$ for $i = 1, 2$, let us set:
$$ \partial^{2(\rho - 1)}_\ell f \df \frac{\partial^{2(\rho - 1)} f}{\partial T_{10}^{\rho - \ell_1} \partial T_{11}^{\ell_1 - 1} \partial T_{20}^{\rho - \ell_2} \partial T_{21}^{\ell_2 - 1}},$$
which is a global section of $\O_\P(r_1 - (\rho - 1), r_2 - (\rho - 1))$. The \em{homogeneous Wronskian} $\Wr(f)$ is the determinant
$$ \Wr(f) \df \left[ \prod_{i = 1}^2 \left( \frac{(r_i - \rho + 1)!} {r_i!}\right)^\rho \right] \cdot \det \left( \partial^{2(\rho - 1)}_\ell f \right)_{\ell_1, \ell_2 = 1, \dots, \rho},$$
seen as a global section of $\O_\P(r_1 - (\rho - 1), r_2 - (\rho - 1))^{\otimes \rho}$.
\end{deff}

Let $ f \in \Gamma(\P, \O_\P(r))$ be a non-zero section and let $\rho = \rk(f)$ be its tensorial rank. Let us write $ f =  \sum_{\ell = 1}^\rho f_{1 \ell} \otimes f_{2 \ell}$ with $f_{i\ell} \in \Gamma(\P^1, \O(r_i))$ for all $i = 1, 2$ and all $\ell = 1, \dots, \rho$. For $i = 1, 2$ let us consider the homogeneous Wronskian
$$ \Wr_i(f) \df \Wr(f_{i1}, \dots, f_{i\rho}),$$
computed with respect to the basis $T_{i0}, T_{i1}$. An elementary computation shows:

\begin{prop} \label{prop:WronskianIsProduct} With the notations introduced above, we have  $$ \Wr(f) = \Wr_1(f) \otimes \Wr_2(f).$$
\end{prop}

\end{npar}

\begin{npar}{Index of the Wronskian} In this section we link the index of Wronskian with the one of the polynomial we started from.

\begin{prop} \label{prop:IndexOfTheWronskian} Let $f \in \Gamma(\P, \O_\P(r))$ be a non-zero section and let $b = (b_1, b_2)$ be a couple of non-negative real numbers. Then, for every $\ol{\Q}$-point $z$ of $\P$ we have:
$$
\ind_{b}(\Wr(f), z) \ge \max \left\{ b_i r_i \right\} \left( (\rk(f) - 1) \left( 2 - \frac{\rk(f) - 1}{r_2}\right) \vol \triangleup_2(t) - \rk(f) \epsilon_{2, r} \right) ,
$$
where $t \df \min \{1,  \inf_b(f, z) / \max \{ b_i r_i \} \}$.
\end{prop}

\begin{proof} Since the Wronskian does not depend (up to a non-zero scalar factor) on the chosen basis (Proposition \ref{prop:WronskianIsInvariant}), then for $i = 1, 2$ we might chose a basis $T_{i0}$, $T_{i1}$ of $K^{2\vee}$ such that $ T_{i0}(\pr_i(z)) \neq 0 $ and $T_{i1}(\pr_i(z)) = 0$. 

Let $t \df \ind_{b}(f, z)$ be the index of $f$ at $z$. Let $\rho = \rk(f)$ be the tensorial rank of $f$ and, up to permuting the coordinates let us suppose $r_1 \le r_2$. Since deriving with respect to $T_{10}$ and $T_{20}$ does not affect the index on $z$, for every $\ell_1, \ell_2 = 1, \dots, \rho$ we have
\begin{align*} 
\ind_b\left( \partial^{2(\rho - 1)}_{(\ell_1, \ell_2)} f, z \right) 
&\ge \max \left\{ 0,  t - \langle b, (\ell_1 - 1, \ell_2 - 1) \rangle  \right\} \\
&\ge \max \left\{ 0,  t - \frac{\ell_2 - 1}{r_2} \max \left\{ b_i r_i \right\}  \right\} - \epsilon_{2, r} \max \left\{ b_i r_i \right\}
\end{align*}
where we wrote, recalling convention $1 / r = (1/r_1, 1/ r_2)$ and using $\ell_1  - 1\le \rho - 1 \le r_2$, 
$$ \langle b, (\ell_1 - 1, \ell_2 - 1) \rangle \le \max \left\{ b_i r_i \right\} \langle 1 / r, (\ell_1 - 1, \ell_2 - 1) \rangle \le \max \left\{ b_i r_i \right\} \left( \epsilon_{2, r} + \frac{\ell_2 - 1}{r_2}\right).$$
Let us denote by $\frak{S}_\rho$ the permutation group on $\{ 1, \dots, \rho\}$. Since the index is a valuation we have
\begin{align*}
\ind_{b}(\Wr(f), z) &\ge \min_{\pi \in \frak{S}_{\rho }} \left\{ \sum_{\ell = 1}^{\rho} \ind_b\left( \partial^{2(\rho - 1)}_{(\pi(\ell), \ell)} f, z \right)  \right\} \\
&\ge \min_{\pi \in \frak{S}_{\rho}} \left\{ \sum_{\ell = 1}^{\rho}  \max \left\{ 0,  t - \frac{\ell - 1}{r_2} \max \left\{ b_i r_i \right\}  \right\} \right\} - \rho \epsilon_{2, r} \max \left\{ b_i r_i \right\}  
\end{align*}
Writing $t' \df t / \max \left\{ b_i r_i \right\} $ and $u = \min \{ (\rho - 1) / r_2, t' \} $, we obtain
$$ \sum_{\ell = 1}^{\rho}  \max \left\{ 0,  t - \frac{\ell - 1}{r_2} \max \left\{ b_i r_i \right\}  \right\} =  \max \{ b_i r_i \}  \left( \sum_{\ell = 0}^{r_2 u} \left(  t' - \frac{\ell}{r_2} \right) \right).$$
We finally have
$$ \sum_{\ell = 0}^{r_2 u} \left(  t' - \frac{\ell}{r_2} \right) = (r_2 u + 1) \left( t' - \frac{u}{2}\right) \ge r_2 u \left( t' - \frac{u}{2}\right) ,$$
and we conclude by:

\begin{lem} \label{lem:TheTwoCasesInDysonsLemma} With the notations introduced above, let $\tilde{t} \df \min \{ t' , 1 \}$. Then,
$$ u \left( t' - \frac{u}{2} \right) \ge \frac{\rho - 1}{r_2} \left( 2 - \frac{\rho - 1}{r_2} \right) \vol \triangleup_2(\tilde{t}).$$
\end{lem}

\begin{proof}[{Proof of Lemma \ref{lem:TheTwoCasesInDysonsLemma}}] Two cases have to be considered:
\begin{enumerate}
\item $u = t'$;
\item $u = (\rho - 1) / r_2$. 
\end{enumerate}
The first case is trivial since we have by hypothesis $u (t' - u / 2) = t'^2 / 2$. Therefore it suffices to remark that we have
$$ \frac{\rho - 1}{r_2} \left( 2 - \frac{\rho - 1}{r_2} \right) \le 1 $$
because $\rho - 1 \le r_2$. For the second case let us set $\tilde{t} \df \min \{ t', 1 \}$. Clearly we have
$$ u \left( t' - \frac{u}{2} \right) \ge u \left( \tilde{t} - \frac{u}{2} \right).$$
Now it suffices to remark that function
$$
\frac{\xi \left( \tilde{t} - \xi / 2\right)}{\xi \left( 2 - \xi \right)} = \frac{\tilde{t} - \xi/2}{ 2 - \xi},
$$
is decreasing for $\xi \in [0, 1]$ because $\tilde{t} \le 1$. By assumption we have $u \le \tilde{t}$. Therefore, applying this consideration with $\xi = u = (\rho - 1)/r_2$, we find
$$  \frac{u \left( \tilde{t} - u /2 \right)}{u (2 - u)} \ge \frac{\tilde{t}^2/ 2}{ \tilde{t} (2 - \tilde{t})} \ge \tilde{t},$$
where in the last inequality we used again the inequality $\tilde{t}(2 - \tilde{t}) \le 1$. This terminates the proof of the lemma.
\end{proof}

This concludes the proof of the Proposition.
\end{proof}
%\vspace{5pt} $\displaystyle \mu(\lambda, [f]) \ge \frac{ \mu(\lambda, [\Wr(f)])}{\rk(f)}$;

\begin{prop} \label{prop:AlmostDysonsLemmaInstabilityCoefficient} Let $f \in \Gamma(\P, \O_\P(r))$ be a non-zero global section and let $b = (b_1, b_2)$ be a couple of non-negative real numbers. 

Let $q \ge 1$ be an integer and  for every $\sigma = 0, \dots, q$ let $z^{(\sigma)}$ be a $\ol{\Q}$-point of $\P$. For every $i = 1, 2$ let us suppose $\pr_i(z^{(\sigma)}) \neq \pr_i(z^{(\tau)})$ for every $\sigma \neq \tau$. Then,
$$ \sum_{\sigma = 0}^q \vol \triangleup_2 (t^{(\sigma)}) \le 1 + \epsilon_{q + 1, r},$$
where 
$$t^{(\sigma)} = 
\begin{cases}
\min \left\{1, \ind_{1/r}(f, z^{(\sigma)}) \right\} & \text{if } \sigma = 1, \dots, q\\
\min \left\{1,  \ind_b(f, z^{(0)}) / \max \left\{ b_i r_i \right\} \right\} & \text{if } \sigma = 0.\\ 
\end{cases}$$

\end{prop}

\begin{proof} Let us suppose $\rk(f) > 1$. The proof is done bounding from above and from below the index $\ind_b(\Wr(f), z^{(0)})$. 

\

\em{Upper bound.} Let us go back to the notation in Proposition \ref{prop:ComparingTheIndexWithRespectTwoWeights}. Since we have $\Wr(f) = \Wr_1(f) \otimes \Wr_2(f)$, Proposition \ref{prop:ComparingTheIndexWithRespectTwoWeights} applied to the weight $c = 1/r$ gives
$$ \ind_b(\Wr(f), z^{(0)}) \le \max\{ b_i r_i \} \ind_{1/r}(\Wr(f),z^{(0)}).$$
We are therefore led back to estimate $\ind_{1/r}(\Wr(f), z^{(0)})$. Let us set $\rho \df \rk(f)$. Using the definition of the index and the fact that $\Wr_i(f)$ is a section of $\O(\rho(r_i - \rho + 1))$ on $\P^1$, we have:
\begin{align*} 
\ind_{1/r}(\Wr(f), z^{(0)}) &= \sum_{i = 1}^2 \frac{1}{r_i} \mult(\Wr_i(f), \pr_i(z^{(0)})) \\
&\le \sum_{i = 1}^2 \frac{1}{r_i} \left( \rho(r_i - \rho + 1) - \sum_{\sigma = 1}^q \mult(\Wr_i(f), \pr_i(z^{(\sigma)}) ) \right) \\
&= \rho \left( 2 - \sum_{i = 1}^2 \frac{\rho - 1}{r_i} \right) - \sum_{\sigma = 1}^q \ind_{1/r}(\Wr(f), z^{(\sigma)}),
\end{align*}
where we used that the projection of the points $z^{(\sigma)}$ are pairwise distinct.  For every $\sigma = 1, \dots, q$, Proposition \ref{prop:IndexOfTheWronskian} (applied to $z = z^{(\sigma)}$ and $b = 1/r$) entails:
$$\ind_{1/r}(\Wr(f), z^{(\sigma)}) \ge (\rho - 1) \left( 2 - \frac{\rho - 1}{r_2}\right) \vol \triangleup_2(t^{(\sigma)}), $$
where $t^{(\sigma)} = \min \left\{ 1,  \ind_{1 / r}(f, z^{(\sigma)}) \right\}$. Summing up, the index $\ind_b(\Wr(f), z^{(0)})$ is bounded above by 
$$ \max \left\{ b_i r_i \right\}  \left[\rho \left( 2 - \sum_{i = 1}^2 \frac{\rho - 1}{r_i} +  \epsilon_{q +1, r} \right)  - (\rho - 1) \left( 2 - \frac{\rho - 1}{r_2}\right)\left( \sum_{\sigma = 1}^q  \vol \triangleup_2(t^{(\sigma)}) \right)  \right],$$
where we noted $q \epsilon_{2, r} = \epsilon_{q + 1, r}$.

\

\em{Lower bound.}  Proposition \ref{prop:IndexOfTheWronskian} applied to the point $z = z^{(0)}$ and to the weight $b$ gives:
$$
\ind_{b}(\Wr(f), y) \ge \max \left\{ b_i r_i \right\} (\rho - 1) \left( 2 - \frac{\rho - 1}{r_2}\right) \vol \triangleup_2(t^{(0)}),
$$
where $t^{(0)} \df \min \left\{ 1, \inf_b(f, z^{(0)}) / \max \{ b_i r_i \} \right\}$.

\ 

Combining the lower bound and the upper bound of $\ind_b(f, z^{(0)})$ we find:
$$ (\rho - 1) \left( 2 - \frac{\rho - 1}{r_2}\right)\left( \sum_{\sigma = 0}^q  \vol \triangleup_2(t^{(\sigma)}) \right) \le \rho \left( 2 - \frac{\rho - 1}{r_2}\right) + \rho \epsilon_{q+1, r},$$
where in the right-hand side we neglected the term $- (\rho -1)/r_1$.  Dividing by  $(\rho - 1 ) \left( 2 - \frac{\rho - 1}{r_2} \right)$ we obtain 
\begin{align*} 
\sum_{\sigma = 0}^q  \vol \triangleup_2(t^{(\sigma)}) &\le \frac{\rho}{\rho -1} + \frac{\rho}{\rho -1} \left( 2 + \frac{\rho -1}{r_2} \right)^{-1} \epsilon_{q+1, r} \\
&\le 1 + \epsilon_{q + 1, r} + \frac{1}{\rho - 1} \left( 1 + \epsilon_{q +1, r}\right),
\end{align*}
where in the second inequality we used $2 - (\rho - 1) / r_2 \ge 1$. Taking powers of $f$ and multiplying by suitable linear polynomials, one can show that $\rho$ can be taken arbitrarily large (even though it could be small compared to $r_2$) --- see \cite[{Lemma 2 and II.4}]{bombieri} for more details. This concludes the proof in the case $\rk(f) >1$.

The same type of argument shows that we can suppose $\rk(f) >1$.
\end{proof}

\begin{cor} \label{cor:DysonsLemmaWtihTwoWeights} Under the assumptions of Proposition \ref{prop:AlmostDysonsLemmaInstabilityCoefficient}, let us moreover suppose that for every $\sigma = 1, \dots, q$ we have $\ind_{1/r}(f, z^{(\sigma)}) \le 1$ and 
\begin{equation} \label{eq:CorAlmostDysonsLemmaInstabilityCoefficientHypothesisProof} 
1 - \sum_{\sigma = 1}^q \vol \triangleup_2(\ind_{1/r}(f, z^{(\sigma)})) + \epsilon_{q + 1, r} < \frac{1}{2}. \end{equation}
Then, $\ind_{b}(f, z^{(0)}) < \max \left\{ b_i r_i \right\}$ and consequently
$$
\vol \triangleup_2 \left( \frac{\ind_{b}(f, z^{(0)})}{\max \left\{ b_i r_i \right\}} \right) \le 
  1 - \sum_{\sigma = 1}^q \vol \triangleup_2(\ind_{1/r}(f, z^{(\sigma)})) + \epsilon_{q + 1, r}.
$$
\end{cor}

\begin{proof} Indeed, if $\ind_{b}(f, y) \ge \max \left\{ b_i r_i \right\}$, Proposition \ref{prop:AlmostDysonsLemmaInstabilityCoefficient} entails
$$ \frac{1}{2} \le 1 - \sum_{\sigma = 1}^q \vol \triangleup_2(\ind_{1/r}(f, z^{(\sigma)})) + \epsilon_{q + 1, r}, $$
which contradicts the hypothesis \eqref{eq:CorAlmostDysonsLemmaInstabilityCoefficientHypothesisProof}.
\end{proof}
\end{npar}

\begin{npar}{Wronskian as a covariant} Let us conclude with a final remark. Fix a positive integer $\rho \ge 1$. The Wronskian furnishes a ``covariant'' for the action of $\SLs_{2, K}^2$, \em{i.e.} a rational $\SLs_{2 ,K}^2$-equivariant map
$$ \Wr:  \P(\Gamma(\P, \O_\P(r))) \dashrightarrow \P(\Sym^{\rho(r_1 - \rho + 1)} K^{2 \vee} \otimes_K \Sym^{\rho(r_2 - \rho + 1)} K^{2 \vee}), $$
which is defined on the open subset $U_\rho \subset \P(\Gamma(\P, \O_\P(r)))$ of lines generated by non-zero sections $f \in \Gamma(\P, \O_\P(r))$ of tensorial rank $\ge \rho$. The Wronskian map $\Wr$ moreover induces a $\SLs_{2, K}^2$-equivariant isomorphism of line bundles  $ \Wr^\ast \O(1) \iso \O(\rho)_{\rvert U_\rho}$. In the early stages of the present work, this constituted for us one of the main evidences that the proof of Roth's was connected to Geometric Invariant Theory.

To make this intuition more precise, let us remark that for such a morphism it is a general fact on GIT that one has
$$  \mu_{\O(1)}(\lambda, [f]) \ge \frac{1}{\rho} \mu_{\O(1)}(\lambda, [\Wr(f)]),$$
for every global section $f$ of $\O_\P(r)$ of tensorial rank $\ge \rho$ and every one-parameter subgroup $\lambda : \Gm \to \SLs_{2, K}^2$. If $y_\lambda$ is the instability point associated to the choice of admissible bases for $\lambda$, Proposition \ref{prop:ComparingInstabilityCoefficientAndIndex} leads to the lower bound
$$ \ind_{m_\lambda}(\Wr(f), y_\lambda) \ge \rho  \ind_{m_\lambda}(f, y_\lambda) - \frac{\rho(\rho- 1)}{2} (m_{\lambda, 1} + m_{\lambda, 2} ) . $$
As we explained before, we want to apply Theorem \ref{thm:DysonsLemmaTwoWeights} to the point $y = y_\lambda$ and the weight $b = m_\lambda$. Thus this is the case we are interested in. Unfortunately, this lower bound is not sharp enough to deduce the semi-stability of the point $P_r$ (for $\mu_2(t_x)$ small enough) and we had to use the lower bound given by Proposition \ref{prop:IndexOfTheWronskian} in the proof of Proposition \ref{prop:AlmostDysonsLemmaInstabilityCoefficient}.
\end{npar}

\bibliography{biblio}
\bibliographystyle{abbrv}

\small

\ \\

\textsc{Marco Maculan}, Institut Math\'ematique de Jussieu, 4 place Jussieu, 75005 Paris, \em{e-mail:} \url{marco.maculan@imj-prg.fr}

\end{document}